\documentclass[10pt]{amsart}
\usepackage{latexsym}
\usepackage{amsxtra}
\usepackage{amssymb}
\usepackage{amsthm}
\usepackage{xspace}

\usepackage[hmargin=3.5cm,vmargin=3.75cm]{geometry}

\usepackage{euscript}
\input xy
\xyoption{all} \CompileMatrices

\newtheorem{theorem}{Theorem}[section]
\newtheorem{proposition}[theorem]{Proposition}
\newtheorem{corollary}[theorem]{Corollary}
\newtheorem{lemma}[theorem]{Lemma}

\theoremstyle{definition}
\newtheorem{example}[theorem]{Example}

\begin{document}

\def\SStar{{\bf{SS}}}
\def\Star{{\bf{S}}}
\def\UU{{\mathcal U}}
\def\setmin{{-}}
\def\UU{{\mathcal U}}
\def\Z{{\mathcal Z}}
\def\PPP{{\mathcal P}}
\def\CCC{{\mathcal C}}
\def\XXX{{\mathcal X}}
\def\YYY{{\mathcal Y}}
\def\id{\operatorname{id}}
\def\Lin{\operatorname{Lin}}
\def\Spec{\operatorname{Spec}}
\def\Cl{\operatorname{Cl}}
\def\Idem{\operatorname{Idem}}
\def\Fix{\operatorname{Fix}}
\def\Cld{\operatorname{-Cl}}
\def\Inv{{\operatorname{Inv}}}
\def\Int{\operatorname{Int}}
\def\Pic{\operatorname{Pic}}
\def\SS{\mathcal{SS}}
\def\U{{\textsf{\textup{U}}}}
\def\V{{\textsf{\textup{V}}}}
\def\Inv{\operatorname{Inv}}
\def\Mor{\operatorname{Mor}}
\def\Int{{\textsf{\textup{Int}}}}
\def\I{{\textsf{\textup{I}}}}
\def\G{{\textsf{\textup{G}}}}
\def\K{{\textsf{\textup{K}}}}
\def\T{{\textsf{\textup{T}}}}
\def\R{{\textsf{\textup{R}}}}
\def\M{{\textsf{\textup{M}}}}
\def\FS{{\textsf{\textup{F}}}}
\def\Idl{{\textsf{\textup{Idl}}}}
\def\A{{\textsf{\textup{A}}}}
\def\C{{\textsf{\textup{C}}}}
\def\N{{\textsf{\textup{N}}}}
\def\Fp{{\mathcal F}}
\def\F{{\mathcal K}}
\def\S{\mathcal S}
\def\P{\operatorname{Prin}}
\def\B{\mathcal B}
\def\f{f}
\def\Inver{\operatorname{-Inv}}
\def\Prin{\textup{-}\mathcal{P}}
\def\Hom{\operatorname{Hom}}
\def\TM{t\textup{-Max}}
\def\TS{t\textup{-Spec}}
\def\Max{\operatorname{Max}}
\def\SSpec{\textup{-Spec}}
\def\SGV{\star\textup{-GV}}
\def\core{\operatorname{-core}}
\def\WBF{\operatorname{WBF}}
\def\sKr{\operatorname{sKr}}
\def\Loc{\operatorname{Loc}}
\def\Ass{\operatorname{Ass}}
\def\wAss{\operatorname{wAss}}
\def\ann{\operatorname{ann}}
\def\ind{\operatorname{ind}}
\def\cl{\star}
\def\ZZ{{\mathbb Z}}
\def\CC{{\mathbb C}}
\def\NN{{\mathbb N}}
\def\RR{{\mathbb R}}
\def\QQ{{\mathbb Q}}
\def\FF{{\mathbb F}}
\def\mm{{\mathfrak m}}
\def\nn{{\mathfrak n}}
\def\aaa{{\mathfrak a}}
\def\bbb{{\mathfrak b}}
\def\ppp{{\mathfrak p}}
\def\qqq{{\mathfrak q}}
\def\MM{{\mathfrak M}}
\def\qq{{\mathfrak Q}}
\def\rr{{\mathfrak R}}
\def\DD{{\mathfrak D}}
\def\cc{{\mathfrak S}}

\newcommand\bigsqcapp{\mathop{\mathchoice
	{\vcenter{\hbox{\huge $ \sqcap $}}}	% display
	{\vcenter{\hbox{\LARGE $ \sqcap $}}}	% text
	{\vcenter{\hbox{$ \sqcap $}}}		% script
	{\vcenter{\hbox{\small $ \sqcap $}}}}}	% scriptscript

\newcommand\bigsqcupp{\mathop{\mathchoice
	{\vcenter{\hbox{\huge$ \sqcup $}}}	% display
	{\vcenter{\hbox{\LARGE$ \sqcup $}}}	% text
	{\vcenter{\hbox{$ \sqcup $}}}		% script
	{\vcenter{\hbox{\small $ \sqcup $}}}}}	% scriptscript

\title[Prequantales]{Prequantales and applications to semistar operations and module systems}
\author{Jesse Elliott} \address{California
State University, Channel Islands\\ One University Drive \\ Camarillo, California 93012}
\email{jesse.elliott@csuci.edu}

\begin{abstract}
We show that a generalization of quantales and prequantales provides a noncommutative and nonassociative abstract ideal theoretic setting for the theories of star operations, semistar operations, semiprime operations, ideal systems, and module systems, and conversely the latter theories motivate new results on quantales and prequantales.  Results include representation theorems for precoherent prequantales and multiplicative semilattices; a characterization of the simple prequantales; and a generalization to the setting of precoherent prequantales of the construction of the largest finite type semistar operation and the largest stable semistar operation smaller than a given semistar operation.
%; and two potentially useful definitions of tight closure for non-Noetherian commutative rings of prime characteristic.
\end{abstract}

\maketitle


%\tableofcontents

\section{Introduction}

Certain classes of ordered algebraic structures---Boolean algebras, Heyting algebras, multiplicative lattices, residuated lattices, locales, and quantales, to name a few---have led to the unification of aspects of logic, order theory, algebra, and topology.  One such set of unifications was initiated in the mid 1920s by Krull in the study, known as abstract commutative ideal theory, of the ideal lattice of a commutative ring as a multiplicative lattice \cite{and,kru1}.  Independently of the work on multiplicative lattices and residuated lattices \cite{ward}, another set of notions, including Brouwer lattices, Heyting algebras, and frames, eventually led to the study of locales, or ``point-free topological spaces'' \cite{isb}.  These in turn led to a common generalization of multiplicative lattices and locales known as quantales, or ``quantum locales''.  Quantales
%are a noncommutative generalization of both multiplicative lattices and locales and 
were introduced in \cite{mul} by Mulvey in order to provide a lattice theoretic setting for the foundations of quantum mechanics and the theory of $\operatorname{C}^*$-algebras and to develop the concept of noncommutative topology introduced in \cite{gil} by Giles and Kummer.

This paper concerns a generalization of quantales known as {\it prequantales} \cite[Definition 2.4.2]{ros} and their applications to abstract ideal theory and commutative ring theory.   Specifically, we show that the theory of nuclei on ordered magmas \cite[Section 3.4.11]{gal}, and more specifically on suitable generalizations of prequantales, provides a noncommutative and nonassociative abstract ideal theoretic setting for the theories of star operations, semistar operations, ideal systems, and module systems, and conversely the latter theories motivate new results on quantales and prequantales.  Sections \ref{sec:POM} through \ref{sec:stable} provide new results on nuclei, as well as on various classes of ordered magmas both new and old, including prequantales, multiplicative lattices, and multiplicative semilattices.  The last three sections, Sections \ref{sec:AMSIS} through \ref{sec:ICCICTC}, apply these results to star and semistar operations and ideal and module systems.  %The reader interested only in the results of these last three sections may skip Sections \ref{sec:POM} through \ref{sec:stable} and fill in the proofs of the results on his or her own.

This introduction summarizes in some detail our main definitions and results.   First, a {\it magma} is a set $M$ equipped with a binary operation on $M$, which we write multiplicatively.  A magma is {\it unital} if it has an identity element.  An {\it ordered magma} is a magma $M$ equipped with a partial ordering $\leq$ on $M$ such that $x \leq x'$ and $y \leq y'$ implies $xy \leq x'y'$ for all $x,x',y,y' \in M$.  A {\it prequantale} is an ordered magma $Q$ such that the supremum $\bigvee X$ of $X$ exists and $a (\bigvee X) = \bigvee(aX)$ and $(\bigvee X)a = \bigvee(Xa)$ for any $a \in Q$ and any subset $X$ of $Q$.  A {\it quantale} is an associative prequantale; a {\it multiplicative lattice} is a commutative unital quantale; and a {\it frame}, or {\it locale}, is a multiplicative lattice in which multiplication is the operation $\wedge$ of infimum.  A prequantale (resp, quantale, multiplicative lattice) is equivalently a magma object (resp., semigroup object, commutative monoid object) in the {\it monoidal category of sup-lattices}, that is, the category of complete lattices equipped with the tensor product with morphisms as the sup-preserving functions \cite[Section II.5]{joy}. 

\begin{example} \
\begin{enumerate}
\item The power set $2^M$ of a magma $M$ is a prequantale under the operation $(X,Y) \longmapsto XY = \{xy: \ x \in X, y \in Y\}$, and $2^M$ is a quantale if and only if $M$ is a semigroup.
\item The poset of all supremum-preserving self-maps of a complete lattice is a unital quantale under the operation of composition.
\item Any totally ordered complete lattice is a frame, as is any complete Boolean algebra.
\item The complete lattice $\mathcal{O}(X)$ of all open subsets of a topological space $X$ is a frame.
\item The set of all Zariski closed subsets of an affine algebraic variety over a field, partially ordered by the relation $\supset$, is a frame.
\item The complete lattice $\mathcal{I}(R)$ of all ideals of a commutative ring $R$ is a multiplicative lattice under ideal multiplication.
\item For any algebra $A$ over a ring $R$, the complete lattice $\operatorname{Mod}_R(A)$ of all (two-sided) $R$-submodules of $A$ is a quantale under the operation $(I,J) \longmapsto IJ = \{\sum_{x \in X, y \in Y} xy: X \subset I, Y \subset J \mbox{ finite}\}$.
\item The poset  $\operatorname{Max}(A)$ of all closed linear subspaces of a unital $\mbox{C}^*$-algebra $A$ is a unital quantale with involution \cite[Definition 2.4]{krum2} under the operation $(M,N) \longmapsto \overline{MN}$ and by \cite[Theorem 5.4]{krum2} is a complete invariant of $A$.
\end{enumerate}
\end{example} 

Restricting $X$ to be among certain subclasses of subsets of $Q$ in the definition above of prequantales yields more general structures.    We say that a {\it near prequantale} (resp., {\it semiprequantale}) is an ordered magma $Q$ such that $\bigvee X$ exists and $a (\bigvee X) = \bigvee(aX)$ and $(\bigvee X)a = \bigvee(Xa)$ for any $a \in Q$ and any nonempty subset $X$ (resp., any nonempty finite or bounded subset $X$) of $Q$.   A near prequantale is equivalently a semiprequantale $Q$ with a largest element, and a prequantale is equivalently a near prequantale $Q$ with a smallest element $\bigwedge Q$ that annihilates every element of $Q$. Informally we think of a near prequantale as a ``prequantale without annihilator'' and a semiprequantale as an ``unbounded prequantale''.  We define the terms {\it near quantale}, {\it semiquantale}, {\it near multiplicative lattice}, and {\it semimultiplicative lattice} in the obvious way.  A theme of this paper is that many known results about quantales and multiplicative lattices generalize to these more general structures. 

\begin{example} \
\begin{enumerate}
%\item A prequantale is equivalently a near prequantale with annihilator.  Likewise a quantale is equivalently a near quantale with annihilator.  %However, a near prequantale or near quantale having a least element need not be a prequantale.
\item If $a$ is an idempotent element of a prequantale $Q$, then the set $Q_{\leq a} = \{x \in Q: x \leq a \}$ is a subprequantale of $Q$, while the set $Q_{\geq a} = \{x \in Q : x \geq a\}$ is a sub near prequantale of $Q$.
%\item In particular, if $M$ is a unital prequantale, then $\I(M) = \{x \in M: \ x \leq 1\}$ is a subprequantale of $M$.
%\item If $p$ is a {\it prime ideal} of a unital magma $M$, that is, if $p < 1$ and $ab \leq p$ implies $a \leq p$ or $b \leq p$ for all $a,b \in \I(M)$, then $N = \{x \in \I(M): \ x \nleq p\}$ is a submagma of the unital magma $\I(M)$, and if $M$ is a unital prequantale, then $N$ is a unital near prequantale.
%\item In particular, if $M$ is a unital prequantale such that $0 = \bigwedge M$ is prime, then $\I(M)\setmin\{0\}$ is a sub near prequantale of the prequantale $\I(M)$.
\item If $Q$ is a prequantale and $x y = 0$ implies $x = 0$ or $y = 0$ for all $x,y \in Q$, where $0 = \bigwedge Q$,  then $Q\setmin\{0\}$ is a sub near prequantale of $Q$.  In particular, the set $2^M \setmin \{\emptyset\}$ of all nonempty subsets of a magma $M$ is a sub near prequantale of the prequantale $2^M$.  
\item Any complete lattice is a near multiplicative lattice under the operation $\vee$ of supremum. % In particular, for any set $X$ the poset $2^X$ is a near multiplicative lattice under the operation of union.
%\item For any integral domain $D$ with quotient field $F$, the partially ordered monoid $\F(D)$ of nonzero $D$-submodules of $F$ is a near multiplicative lattice.
%\item In Section \ref{sec:SAMCO}, we will show that if $M$ is a near sup-magma, then the set $\N(M)$ of all nuclei on $M$ has the structure of a near multiplicative lattice. 
\end{enumerate}
\end{example}

For any self-map $\star$ of a set $X$ we write $x^\star = \star(x)$ for all $x \in X$.    If $\star$ is a self-map of a magma $M$, then the binary operation $(x,y) \longmapsto (xy)^\star$ on $M$ will be called {\it $\star$-multiplication}.   A {\it closure operation} on a poset $S$ is a self-map $\star$ of $S$ satisfying the following conditions.
\begin{enumerate}
\item $\star$ is expansive: $x \leq x^\cl$ for all $x \in S$.
\item $\star$ is order-preserving: $x \leq y$ implies $x^\cl \leq y^\cl$ for all $x,y \in S$.
\item $\star$ is idempotent: $(x^\cl)^\cl = x^\cl$ for all $x \in S$.
\end{enumerate}
A closure operation on $S$ is equivalently a self-map $\star$ of $S$ such that $x \leq y^\cl$ if and only if $x^\cl \leq y^\cl$ for all $x, y \in S$.  A {\it nucleus} (resp., {\it strict nucleus}) on an ordered magma $M$ is a closure operation $\star$ on $M$ such that $x^\star y^\star \leq (xy)^\star$ (resp., $x^\star y^\star = (xy)^\star$) for all $x,y \in M$.   Nuclei were first studied in the contexts of ideal lattices and locales and later in the context of quantales as {\it quantic nuclei} \cite{nie}. 

\begin{example}\label{nearexamples}  \
\begin{enumerate}
\item For any commutative ring $R$, a {\it semiprime operation} on $R$ \cite{pet} is equivalently a nucleus on $\mathcal{I}(R)$.  For example, the ideal radical operation $I \longmapsto \sqrt{I}$ is a semiprime operation on $R$, and if $R$ is Noetherian of prime characteristic then the operation of {\it tight closure} \cite{hun} is a semiprime operation on $R$.
\item By \cite[Lemmas 10.1--10.4]{giv}, for any topological space $X$ the operation $\operatorname{reg}: U \longmapsto \operatorname{int}(\operatorname{cl}(U))$ is a strict nucleus on the frame $\mathcal{O}(X)$, where $\operatorname{cl}$ and $\operatorname{int}$ denote the topological closure and interior operations, respectively, on $2^{X}$.  (Open sets $U$ such that $U = U^{\operatorname{reg}}$ are called {\em regular open sets}.)
\item If $X$ is a topological space, then $2^{X}$ is a near multiplicative lattice under the operation of union and topological closure is a strict nucleus on $2^{X}$.
%\item If $S$ is a join semilattice, then $S$ is a multiplicative semilattice under the operation $\vee$ and the map $x \longmapsto a \vee x$ is a strict nucleus on $S$ for any $a \in S$.
%\item For any integral domain $D$ with quotient field $F$, let $\F(D)$ denote the ordered monoid of all nonzero $D$-submodules of $F$ under the operation of multiplication.  By \cite[Proposition 5]{oka} a {\it semistar operation} on an integral domain $D$ is a nucleus on the near multiplicative lattice $\F(D)$.
%\item For any magma $M$, let $M_0$ denote the magma $M \amalg \{0\}$, where $0x = 0 = x0$ for all $x \in M_0$.   By \cite[Proposition 1.2]{hal} a {\it module system} on an abelian group $G$ is a nucleus on the multiplicative lattice $2^{G_0}$.
%\item Let $\mathfrak{R} \subset A \times B$ be a binary relation between sets $A$ and $B$. For any $X \subset A$ define $\mathfrak{R}[X] = \{b \in B: x\mathfrak{R}b \mbox{ for all } x \in X\}$ and $X^{\star_\mathfrak{R}} = \mathfrak{R}^{-1}[\mathfrak{R}[X]]$.  Then $\star_\mathfrak{R}$ is a closure operation on the poset $2^A$.  Moreover, if $A$ is a magma, then $\star_\mathfrak{R}$ is a nucleus on the prequantale $2^A$ if and only if the binary relation $\mathfrak{R}$ is {\it nuclear} in the sense of \cite[Section 3.4.11]{gal}.
\end{enumerate}
\end{example} 

Star and semistar operations can be characterized as nuclei.  Let $D$ be an integral domain with quotient field $F$, and let $\F(D)$ denote the ordered monoid of all nonzero $D$-submodules of $F$ under the operation of multiplication, which is a near multiplicative lattice.  A {\it semistar operation} on $D$ is a closure operation $\star$ on the poset $\F(D)$ such that $(aI)^\star = aI^\star$ for all nonzero $a \in F$ and all $I \in \F(D)$.  Semistar operations were introduced in \cite{oka} as a generalization of {\it star operations}, which were introduced by Krull in \cite[Section 6.43]{kru}.  The following result, along with a similar characterization of star operations, is proved in Section \ref{sec:ASO}.

\begin{theorem}\label{mainprop1}
A semistar operation on an integral domain $D$ is equivalently a nucleus on the near multiplicative lattice $\F(D)$.  Moreover, the following are equivalent for any self-map $\star$ of $\F(D)$.
%The following are equivalent for any  integral domain $D$ and any self-map $\star$ of $\F(D)$.
\begin{enumerate}
\item $\star$ is a semistar operation on $D$.
\item $\star$ is a closure operation on the poset $\F(D)$ and $\star$-multiplication is associative.
\item $\star$ is a closure operation on the poset $\F(D)$ and $(I^\star J^\star)^\star = (IJ)^\star$ for all $I,J \in \F(D)$.
\item $IJ \subset K^\star$ if and only if $I J^\star \subset K^\star$ for all $I,J,K \in \F(D)$.
\end{enumerate}
\end{theorem}

In particular, one can characterize a semistar operation on $D$ as a self-map $\star$ of $\F(D)$ satisfying a single axiom, namely, statement (4) of the theorem above.  Moreover, it follows that the set of all semistar operations on $D$ depends only on the ordered magma $\F(D)$.  Most importantly, however, the theorem suggests that there is potential overlap among the theory of semistar operations, abstract ideal theory, and the theory of quantales.  In fact, the connections among these theories suggested by Theorem \ref{mainprop1} are the main inspiration for this paper.

The {\it weak ideal systems} and {\it module systems} of \cite{hal1, hal}, which generalize semiprime operations and semistar operations, respectively, can also be characterized as nuclei.  For any magma $M$, let $M_0$ denote the magma $M \amalg \{0\}$, where $0x = 0 = x0$ for all $x \in M_0$.  A {\it module system} on an abelian group $G$ is a closure operation $r$ on the multiplicative lattice $2^{G_0}$ such that $\emptyset^r = \{0\}$ and $(cX)^r = cX^r$ for all $c \in G_0$ and all $X \subset G_0$.  The following result, along with a similar characterization of weak ideal systems, is proved in Section \ref{sec:AMSIS}.

\begin{theorem}\label{modulesystems}
A module system on an abelian group $G$ is equivalently a nucleus $r$ on the multiplicative lattice $2^{G_0}$ such that $\emptyset^r = \{0\}$.   Moreover, the following are equivalent for any self-map $r$ of $2^{G_0}$ such that $\emptyset^r = \{0\}$.
\begin{enumerate}
\item $r$ is a module system on $G$.
\item $r$ is a closure operation on the poset $2^{G_0}$ and $r$-multiplication is associative.
\item $r$ is a closure operation on the poset $2^{G_0}$ and $(X^r Y^r)^r = (XY)^r$ for all $X,Y \in 2^{G_0}$.
\item $XY \subset Z^r$ if and only if $XY^r \subset Z^r$ for all $X,Y,Z \subset G_0$.
\end{enumerate}
\end{theorem}

The remainder of this paper is organized as follows.  First, in Section \ref{sec:POM} we introduce several classes of ordered magmas.  Characterizations of the classes defined in that section are given in Table 1, and the various relationships among these classes are summarized in the lattice diagrams of Figures 1, 2, and 3.  Within any of these three diagrams the intersection of any two of the classes is the class lying directly above both of them.  In Sections \ref{sec:MCO} and \ref{sec:COCL} we study properties of nuclei and the poset of all nuclei on an ordered magma, with particular attention to the classes of ordered magmas defined in Section \ref{sec:POM}.    In Section \ref{sec:PFN} we examine the poset of all finitary nuclei on an ordered magma, where a closure operation $\star$ on a poset $S$ is said to be {\it finitary} if $(\bigvee \Delta)^\star = \bigvee (\Delta^\star)$ for all directed subsets $\Delta$ of $S$ for which $\bigvee \Delta$ exists.  For example, a  {\it finite type} semistar operation \cite[Definition 3]{oka} on an integral domain $D$ is equivalently a finitary nucleus on $\F(D)$.

\begin{table} 
\caption{Characterizations of ordered magmas (om's)}
\centering  \ \\
\begin{tabular}{l|l|l|l} 
pom $M$ & abbr. & $\bigvee X$ exists & $\bigvee(XY) = \bigvee X \bigvee Y$ \\
\hline\hline
sup-magma & s &	for all $X$ &   \\ \hline
near sup-magma & ns &	for all $X \neq \emptyset$ &   \\ \hline
dcpo magma & d & 	for all directed $X$ &   \\ \hline
bounded complete & bc & 	for all $X \neq \emptyset$  &   \\ 
 & &  bounded above	 &   \\ \hline
bounded above & b & 	for $X = M$ &   \\ \hline
with annihilator & a &	for $X = \emptyset$ &  for $X = \emptyset$ or $Y = \emptyset$ \\ \hline
prequantale & p &  for all $X$ & for all $X,Y$ \\ \hline
near prequantale & np &  for all $X \neq \emptyset$ & for all $X,Y \neq \emptyset$ \\ \hline
semiprequantale & sp &  for all finite or    &  for all (finite or)  \\
&  & bounded $X \neq \emptyset$ & bounded $X,Y \neq \emptyset$ \\ \hline
 prequantic  & ps & for all finite $X$ & for all finite $X,Y$ \\ 
 semilattice & & & \\ \hline 
multiplicative & ms &  for all finite $X \neq \emptyset$   &  for all finite $X,Y \neq \emptyset$ \\
semilattice &  &  & \\ \hline
Scott topological & t &			 & for all directed $X,Y$ \\
	 & 			& &  such that $\bigvee X$, $\bigvee Y$  \\ 
	 	 & 		&	 &  exist \\ \hline
residuated & r & if $\exists x,y : \ X = $ & for all $X,Y$ such that \\ 
	 & & $\{z: zy \leq x\}$ or &  $\bigvee X$, $\bigvee Y$ exist  \\ 
	 & & $X = \{z: yz \leq x\}$ & \\ \hline
near residuated & nr & if $X \neq \emptyset$ and $\exists x,y :$ & for all $X,Y \neq \emptyset$ such \\ 
	 & & $X = \{z: zy \leq x\}$ or &   that $\bigvee X$, $\bigvee Y$ exist  \\ 
	 & & $X = \{z: yz \leq x\}$ &  
\end{tabular}
\end{table}

\begin{figure}
\caption{Relationships among classes of ordered magmas}
\begin{eqnarray*}
\SelectTips{cm}{11}\xymatrix @R=.5cm @C=.6cm {
 && & {\mbox{p}}  \ar@{=>}[dl] \ar@{=>}[dr] & & & \\ 
 && {\mbox{r+ns}} \ar@{=>}[dl] \ar@{=>}[dr] & & {\mbox{np+s}} \ar@{=>}[dl] \ar@{=>}[dr] & & \\ 
 &{\mbox{r+d}} \ar@{=>}[dr] \ar@{=>}[dl] & & {\mbox{np}} \ar@{=>}[dl] \ar@{=>}[dr] & & {\mbox{t+s}}  \ar@{=>}[dr] \ar@{=>}[dl] & \\ 
 {\mbox{r}} \ar@{=>}[dr] & & {\mbox{nr+d}} \ar@{=>}[dr] \ar@{=>}[dl] & & {\mbox{t+ns}} \ar@{=>}[dl] \ar@{=>}[dr] & & {\mbox{s}} \ar@{=>}[dl] \\ 
 & {\mbox{nr}} \ar@{=>}[dr]& & {\mbox{t+d}} \ar@{=>}[dr] \ar@{=>}[dl] & &  {\mbox{ns}} \ar@{=>}[dl] & \\
 &  & {\mbox{t}} \ar@{=>}[dr] &  & {\mbox{d}} \ar@{=>}[dl] & &  \\
 &  &  & {\mbox{om}} & & & }
\end{eqnarray*}
\end{figure}

\begin{figure}
\caption{Further relationships among classes of ordered magmas}
\begin{eqnarray*}
\SelectTips{cm}{11}\xymatrix @R=.5cm @C=.6cm {
   & & {\mbox{p}} \ar@{=>}[dl] \ar@{=>}[dr] \ar@{=>}[d] &  \\ 
 & {\mbox{ps+t}} \ar@{=>}[dl] \ar@{=>}[d] \ar@{=>}[dr] & {\mbox{ps+s}}  \ar@{=>}[dr] \ar@{=>}[dl] & {\mbox{np}} \ar@{=>}[dl] \ar@{=>}[d]   \\ 
 {\mbox{t+a}}  \ar@{=>}[d]  \ar@{=>}[dr] & {\mbox{ps}} \ar@{=>}[dr] \ar@{=>}[dl] & {\mbox{ms+t}} \ar@{=>}[d] \ar@{=>}[dl] & {\mbox{ms+ns}}  \ar@{=>}[dl] \\ 
 {\mbox{a}}  \ar@{=>}[dr]  &  {\mbox{t}}  \ar@{=>}[d]  & {\mbox{ms}}  \ar@{=>}[dl]  &  \\
 & {\mbox{om}} & &  }
\end{eqnarray*}
\end{figure}

%\begin{figure}
%\caption{Further relationships among classes of pom's}
%\begin{eqnarray*}
%\SelectTips{cm}{11}\xymatrix @R=.5cm @C=.6cm {
 %  & & {\mbox{p}} \ar@{=>}[dl] \ar@{=>}[dr] \ar@{=>}[d] &  \\ 
% & {\mbox{ps+t+b}} \ar@{=>}[dl] \ar@{=>}[d] \ar@{=>}[dr] & {\mbox{ps+ns}}  \ar@{=>}[dr] %\ar@{=>}[dl] & {\mbox{sp}} \ar@{=>}[dl] \ar@{=>}[d]   \\ 
% {\mbox{t+a+b}}  \ar@{=>}[d]  \ar@{=>}[dr] & {\mbox{ps+b}} \ar@{=>}[dr] \ar@{=>}[dl] & %{\mbox{ms+t}} \ar@{=>}[d] \ar@{=>}[dl] & {\mbox{ms+bc}}  \ar@{=>}[dl] \\ 
 %{\mbox{a+b}}  \ar@{=>}[dr]  &  {\mbox{t}}  \ar@{=>}[d]  & {\mbox{ms}}  \ar@{=>}[dl]  &  \\
% & {\mbox{pom}} & &  }
%\end{eqnarray*}
%\end{figure}

\begin{figure}
\caption{Further relationships among classes of ordered magmas}
\begin{eqnarray*}
\SelectTips{cm}{11}\xymatrix @R=.5cm @C=.6cm {
   & & {\mbox{p}} \ar@{=>}[dl] \ar@{=>}[dr] \ar@{=>}[d] &  & \\ 
 & {\mbox{ps+t+b}} \ar@{=>}[dl] \ar@{=>}[d] \ar@{=>}[dr] & {\mbox{ps+s}}  \ar@{=>}[dr] \ar@{=>}[dl] & {\mbox{np}} \ar@{=>}[dl] \ar@{=>}[d] \ar@{=>}[dr] &  \\ 
 {\mbox{t+a+b}} \ar@{=>}[d]  \ar@{=>}[dr] & {\mbox{ps+b}} \ar@{=>}[dr] \ar@{=>}[dl] & {\mbox{ms+t+b}} \ar@{=>}[d] \ar@{=>}[dl] \ar@{=>}[dr] & {\mbox{ms+ns}} \ar@{=>}[dr] \ar@{=>}[dl] &  {\mbox{sp}} \ar@{=>}[dl] \ar@{=>}[d] \\ 
 {\mbox{a+b}}   \ar@{=>}[dr] &  {\mbox{t+b}}  \ar@{=>}[dr] \ar@{=>}[d]  & {\mbox{ms+b}}  \ar@{=>}[dl]  \ar@{=>}[dr] & {\mbox{ms+t}}  \ar@{=>}[dl] \ar@{=>}[d] &  {\mbox{ms+bc}} \ar@{=>}[dl]  \\
 & {\mbox{b}} \ar@{=>}[dr]  & {\mbox{t}} \ar@{=>}[d]  & {\mbox{ms}} \ar@{=>}[dl] & \\
 &  & {\mbox{om}} & &  }
\end{eqnarray*}
\end{figure}

In Section \ref{sec:FN} we generalize the well-known construction of the largest finite type semistar operation $\star_f$ smaller than a semistar operation $\star$, as follows.  An element $x$ of a poset $S$ is said to be {\it compact} if whenever $x \leq \bigvee \Delta$ for some directed subset $\Delta$ of $S$ for which $\bigvee \Delta$ exists one has $x \leq y$ for some $y \in \Delta$.  Generalizing the notion of a precoherent quantale \cite[Definition 4.1.1(2)]{ros}, we say that an ordered magma $M$ is {\it precoherent} if every element of $M$ is the supremum of a set of compact elements of $M$ and the set $\K(M)$ of all compact elements of $M$ is a submagma of $M$.  For example, the near multiplicative lattice $\F(D)$ is precoherent for any integral domain $D$ with quotient field $F$ since $\K(\F(D))$ is the set of all nonzero finitely generated $D$-submodules of $F$. The following result is proved in Section \ref{sec:FN}.

\begin{theorem}\label{precoherentthm}
If $\star$ is any nucleus on a precoherent semiprequantale $Q$, then there is a largest finitary nucleus $\star_f$ on $Q$ that is smaller than $\star$, and one has $x^{\star_f} = \bigvee\{y^\star: \ y \in \K(Q) \mbox{ and } \ y \leq x\}$ for all $x \in Q$.
\end{theorem}

In Section \ref{sec:RML} we prove that the association $Q \longmapsto \K(Q)$ yields an equivalence between the category of precoherent near prequantales (resp., precoherent prequantales) and the category of multiplicative semilattices (resp., prequantic semilattices).  This result, along with Theorem \ref{precoherentthm} above, provides substantive motivation for the study of precoherent near prequantales.  In Section \ref{sec:DC} we generalize the {\it generalized divisorial closure} semistar operations \cite[Example 1.8(2)]{pic} to the contexts of near prequantales and residuated ordered monoids, and we use this to give a characterization of the simple near prequantales.

In Sections \ref{sec:SAMCO} and \ref{sec:CFN} we examine further the structure of the posets $\N(M)$ and $\N_f(M)$ of all nuclei and all finitary nuclei, respectively, on various classes of ordered magmas $M$.  For example, for any near sup-magma $M$ the poset $\N(M)$ has the structure of a near multiplicative lattice under the operation of supremum.  In particular, since $\F(D)$ for any integral domain $D$ is a near multiplicative lattice, the set $\N(\F(D))$ of all semistar operations on $D$ also has such a structure.  Also, in Section \ref{sec:CFN} we determine the compact elements of the poset $\N_f(M)$ for a specfic class of precoherent near multiplicative lattices $M$.

In Section \ref{sec:stable} we generalize several known results on stable semistar operations to the context of precoherent semimultiplicative lattices, where a semistar operation $\star$ on an integral domain $D$ is said to be {\it stable} if
$(I \cap J)^\star = I^\star \cap J^\star$ for all $I,J \in \F(D)$ \cite{fon0}.   Finally, in Sections \ref{sec:AMSIS} through \ref{sec:ICCICTC} we apply the theory of nuclei to ideal and module systems and star and semistar operations.  There, for example, we prove Theorems \ref{mainprop1} and \ref{modulesystems}, we give a construction of the smallest semistar operation extending a given star operation, and we provide some new examples of semistar operations induced by complete integral closure, plus closure, and tight closure.  The last of these examples leads to two potentially useful definitions of tight closure as a semiprime operation on any commutative ring, not necessarily Noetherian, of prime characteristic.  %They also lead to a construction of the smallest completely integrally closed overring of a given integral domain.

%This leads to the following representation theorems.

%\begin{theorem} Let $M$ be an ordered magma.
%\begin{enumerate}
%\item $M$ is a precoherent near prequantale if and only if $M$ is isomorphic to %$(2^N\setmin\{\emptyset\})^\star$ for some multiplicative semilattice $N$ and some finitary %nucleus $\star$ on $2^N \setmin \{\emptyset\}$.
%\item $M$ is a precoherent prequantale if and only if $M$ is isomorphic to $(2^N)^\star$ %for some prequantic semilattice $N$ and some finitary nucleus $\star$ on $2^N$.
%\end{enumerate}
%\end{theorem}

%\begin{theorem} Let $N$ be an ordered magma.
%\begin{enumerate}
%\item $N$ is a multiplicative semilattice if and only if $N$ is isomorphic to $\K(M)$ for %some precoherent near prequantale $M$.
%\item $N$ is a prequantic semilattice if and only if $N$ is isomorphic to $\K(M)$ for some %precoherent prequantale $M$.
%\end{enumerate}
%\end{theorem}

\section{Ordered algebraic structures}\label{sec:POM}

If a proof in this paper is omitted then its reconstruction should be routine.  All rings and algebras are assumed unital and all magmas are written multiplicatively.  For any subset $X$ of a poset $S$ we write $\bigvee X = \bigvee_S X$ for the supremum of $X$ and $\bigwedge X = \bigwedge_S X$ for the infimum of $X$ if either exists.  (Note the cases $\bigvee \emptyset = \bigwedge S$ and $\bigwedge \emptyset = \bigvee S$.)

The {\it category of posets} has as morphisms the order-preserving functions.
% Note that a bijective poset morphism need not be an isomorphism.
A poset in which every pair of elements of has a supremum (resp., infimum) is said to be a {\it join semilattice} (resp., {\it meet semilattice}).  A {\it lattice} is a poset that is both a join and meet semilattice.  A poset $S$ is {\it complete} if every subset of $S$ has a supremum in $S$, or equivalently if every subset of $S$ has an infimum $S$.  A complete poset is also known as a {\it complete lattice}, or {\it sup-lattice}.  A poset $S$ is {\it bounded complete} if every nonempty subset of $S$ that is bounded above has a supremum in $S$, or equivalently if every nonempty subset of $S$ that is bounded below has an infimum in $S$. 

We will say that a poset $S$ is {\it near sup-complete}, or a {\it near sup-lattice}, if every nonempty subset of $S$ has a supremum in $S$. A near sup-lattice is equivalently a bounded complete poset with a largest element, or equivalently a poset $S$ such that every subset of $S$ that is bounded below has an infimum in $S$.   A sup-lattice is equivalently a near sup-lattice having a least element.  A map $f: S \longrightarrow T$ between posets is said to be {\it sup-preserving} if $f(\bigvee X) = \bigvee f(X)$ for every subset $X$ of $S$ for which $\bigvee X$ exists.   We will say that a map $f: S \longrightarrow T$ between posets is {\it near sup-preserving} if $f(\bigvee X) = \bigvee f(X)$ for every nonempty subset $X$ of $S$ such that $\bigvee X$ exists.  If $f$ is near sup-preserving, then $f$ is sup-preserving if and only if $f(\bigwedge S) = \bigwedge T$ or $\bigwedge S$ does not exist.  The morphisms in the {\it category of sup-lattices} (resp., {\it category of near sup-lattices}) are the sup-preserving (resp., near sup-preserving) functions. 

A nonempty subset $\Delta$ of a poset $S$ is said to be {\it directed} if every finite subset of $\Delta$ has an upper bound in $\Delta$.   A poset $S$ is {\it directed complete}, or a {\it dcpo}, if each of its directed subsets has a supremum in $S$.  We will say that $S$ is a {\it bdcpo} if every directed subset of $S$ that is bounded above has a supremum in $S$.  A subset $X$ of $S$ is said to be {\it downward closed} (resp., {\it upward closed}) if $y \in X$ whenever $y \leq x$ (resp., $y \geq x$) for some $x \in X$.  The set $X$ is said to be {\it Scott closed} if $X$ is a downward closed subset of $S$ and for any directed subset $\Delta$ of $X$ one has $\bigvee \Delta \in X$ if $\bigvee \Delta$ exists.  The set $X$ is {\it Scott open} if its complement is Scott closed, or equivalently if $X$ is an upward closed subset of $S$ and $X \cap \Delta \neq \emptyset$ for any directed subset $\Delta$ of $S$ with $\bigvee \Delta \in X$.  The Scott open subsets of $S$ form a topology on $S$ called the {\it Scott topology}.  A function $f: S \longrightarrow T$ between posets is said to be {\it Scott continuous} if $f$ is continuous when $S$ and $T$ are endowed with the Scott topologies.  Equivalently, $f$ is Scott continuous if and only if $f(\bigvee \Delta) = \bigvee f(\Delta)$ for every directed subset $\Delta$ of $S$ for which $\bigvee \Delta$ exists.  Every near sup-preserving function is Scott continuous, and every Scott continuous function is order-preserving.

If $S$ and $T$ are posets, then the set $S \times T$ is a poset under the relation $\leq$ defined by $(x,y) \leq (x',y')$ iff $x \leq x'$ and $y \leq y'$. The Scott topology on $S \times T$ may be strictly finer than the product of the Scott topologies on $S$ and $T$ \cite[Exercise II-4.26]{gie}.  The category of posets equipped with the product $\times$ with morphisms as the Scott continuous functions is monoidal.

A magma is {\it unital} (resp., {\it left unital}, {\it right unital}) if it has an identity element (resp., left identity element, right identity element).   A {\it submagma} of a magma $M$ is a subset of $M$ that is closed under multiplication.  A map $f: M \longrightarrow N$ of magmas is a {\it homomorphism} of magmas if $f(xy) = f(x)f(y)$ for all $x,y \in M$.   %A submagma of a unital magma is {\it unital} if it contains the identity element $1$ of $M$, and a homomorphism $f$ of unital magmas is {\it unital} if $f(1) = 1$.  All submagmas and homomorphisms of unital magmas will be assumed unital. 
The morphisms in the {\it category of ordered magmas} are the order-preserving magma homomorphisms.  We will say that a {\it sup-magma} (resp., {\it near sup-magma}, {\it dcpo magma}, {\it bdcpo magma}) is an ordered magma that is complete (resp., near sup-complete, directed complete, a bdcpo) as a poset.  The morphisms in the {\it category of sup-magmas} (resp., {\it category of near sup-magmas}, {\it category of dcpo magmas},  {\it category of bdcpo magmas}) are the sup-preserving (resp., near sup-preserving, Scott continuous, Scott continuous) magma homomorphisms.

An {\it annihilator} of an ordered magma $M$ is a least element $0 = \bigwedge M$ of $M$ such that $0x = 0 = x0$ for all $x \in M$.  If an annihilator of $M$ exists then we say that $M$ is {\it with annihilator}.

\begin{lemma}\label{scott2}
The following are equivalent for any ordered magma $M$.
\begin{enumerate}
\item The map $M \times M  \longrightarrow M$ of multiplication in $M$ is sup-preserving and $\bigwedge M$ is an annihilator of $M$ if $\bigwedge M$ exists (resp., multiplication in $M$ is near sup-preserving, multiplication in $M$ is Scott continuous).
\item For all $a \in M$, the left and right multiplication by $a$ maps on $M$ are sup-preserving (resp., near sup-preserving, Scott continuous).
\item $a (\bigvee X) = \bigvee(a X)$ and $(\bigvee X) a = \bigvee(Xa)$ for any $a \in M$ and any subset (resp., any nonempty subset, any directed subset) $X$ of $M$ such that $\bigvee X$ exists.
\item $\bigvee (X Y) = \bigvee X \bigvee Y$ for any subset (resp., any nonempty subset, any directed subset) $X$ and $Y$ of $M$ such that $\bigvee X$ and $\bigvee Y$ exist.
\end{enumerate}
\end{lemma}

We will say that an ordered magma $M$ is {\it Scott-topological} if the map $M \times M \longrightarrow M$ of multiplication in $M$ is Scott continuous.  In other words, a Scott-topological ordered magma is a magma object in the monoidal category of posets equipped with the product $\times$ and with morphisms as the Scott continuous functions.  Likewise, a Scott-topological ordered semigroup is equivalently a semigroup object in that monoidal category.  If $M$ is a topological magma (resp., semigroup) when endowed with the Scott topology, then $M$ is a Scott-topological magma (resp., semigroup); the converse appears to be false, although we do not know a counterexample.

A near prequantale is equivalently a magma object in the monoidal category of near sup-lattices equipped with the product $\times$.  They form a full subcategory of the category of near sup-magmas.  A near quantale (resp., near multiplicative lattice) is equivalently a semigroup object (resp., commutative monoid object) in the monoidal category of near sup-lattices.  See Example \ref{nearexamples} for examples.  For a further example, note that the poset of all near sup-preserving self-maps of a near sup-lattice is a unital near quantale under the operation of composition.

An element $a$ of an ordered magma $M$ is said to be {\it residuated} if for all $x \in M$ there exists a largest element $x/a$ of $M$ such that $(x/a)a \leq x$ and a largest element $a \backslash x$ of $M$ such that $a(a \backslash x) \leq x$.  (Some authors denote $x/a$ and $a \backslash x$ by $a \rightarrow x$ and $x \leftarrow a$, respectively.)  An ordered magma $M$ is said to be {\it residuated} if every element of $M$ is residuated.  By \cite[Theorem 3.10]{gal}, if $M$ is residuated, then for all $a \in M$ the left and right multiplication by $a$ maps on $M$ are sup-preserving.

\begin{example}\label{nearresexamples} \
\begin{enumerate}
\item By \cite[Corollary 3.11]{gal} a prequantale is equivalently a complete and residuated ordered magma.
\item Let $L$ be a bounded lattice.  Then $L$ is an ordered monoid under the operation $\wedge$ of infimum; $L$ is residuated if and only if $L$ is a {\it Heyting algebra} \cite[Section 1.1.4]{gal}; and $L$ is a quantale if and only if $L$ is a complete Heyting algebra, also known as a {\it frame} or {\it locale}.
\item If $S$ is a nontrivial complete lattice, then the near multiplicative lattice $S$ under the operation $\vee$ of supremum is $S$ is not residuated since one has $x \vee \bigvee\{y \in S: \ x \vee y \leq 1\} = x \geq 1$ for all $x \in S$.  %In particular, for any nonempty set $X$ the near multiplicative lattice $2^X$ under the operation of union is not residuated.
%\item Let $M$ be a magma and $X$ a nonempty set.  The set $P(X,M)$ of all partial functions from $X$ to $M$ is a Scott-topological dcpo magma, where the product of two functions is the pointwise product on the intersection of their domains and where $f \leq g$ if $g$ is an extension of $f$.  However, $P(X,M)$ is complete if and only if either $M$ or $X$ is a singleton.
\end{enumerate} 
\end{example}

\begin{proposition}\label{quantales}
The following are equivalent for any ordered magma $M$.
\begin{enumerate}
\item $M$ is a prequantale.
\item $M$ is complete and residuated.
\item $M$ is a near prequantale with annihilator.
\item $M$ is complete with annihilator and the multiplication map $M \times M \longrightarrow M$ is sup-preserving.
\item $M$ is complete and the left and right multiplication by $a$ maps on $M$ are sup-preserving for all $a \in M$.
%\item $a (\bigvee X) = \bigvee(a X)$ and $(\bigvee X) a = \bigvee(Xa)$ for all $a \in M$ and all subsets $X$ of $M$.
\item The map $2^M \longrightarrow M$ acting by $X \longmapsto \bigvee X$ is a well-defined magma homomorphism.
\item $\bigvee (X Y) = \bigvee X \bigvee Y$ for all subsets $X$ and $Y$ of $M$.
\end{enumerate}
\end{proposition}

Although near prequantales are not necessarily residuated, they are ``near residuated''.  We say that an element $a$ of an ordered magma $M$ is {\it near residuated} if for all $a \in M$ such that $z a \leq x$ (resp., $a z \leq x$) for some $z \in M$ there exists a largest such element $z = x/a$ (resp., $z = a \backslash x$) of $M$.  We say that $M$ {\it near residuated} if every element of $M$ is near residuated.  For example, Example \ref{nearresexamples}(3) is near residuated.  If $M$ is near residuated, then for all $a \in M$ the left and right multiplication by $a$ maps on $M$ are near sup-preserving.

\begin{proposition}\label{nearprequantales}
The following are equivalent for any ordered magma $M$.
\begin{enumerate}
\item $M$ is a near prequantale.
%\item $M_x = \{y \in M: \ y \geq x\}$ is a prequantale for any $x \in M$ that is not least in $M$.
\item $M$ is near sup-complete and near residuated.
\item The ordered magma $M_0 = M \amalg \{0\}$, where $0x = 0 = x0$ and $0 \leq x$ for all $x \in M_0$, is a prequantale. 
\item $M$ is near sup-complete and the multiplication map $M \times M \longrightarrow M$ is near sup-preserving.
\item $M$ is near sup-complete and the left and right multiplication by $a$ maps on $M$ are near sup-preserving for all $a \in M$.
%\item $a (\bigvee X) = \bigvee(a X)$ and $(\bigvee X) a = \bigvee(Xa)$ for all $a \in M$ and all nonempty subsets $X$ of $M$.
\item The map $2^M \setmin\{\emptyset\} \longrightarrow M$ acting by $X \longmapsto \bigvee X$ is a well-defined magma homomorphism.
\item $\bigvee (X Y) = \bigvee X \bigvee Y$ for all nonempty subsets $X$ and $Y$ of $M$.
\end{enumerate}
\end{proposition}

A {\it multiplicative semilattice} is an ordered magma $M$ such that $M$ is a join semilattice and $a (x \vee y) = ax \vee ay$ and $(x \vee y)a = xa \vee ya$ for all $a,x,y \in M$ \cite[Section XIV.4]{bir}.  We will say that a {\it prequantic semilattice} is a multiplicative semilattice with annihilator.

\begin{example}  The set $\K(2^M)$ of all finite subsets of a magma $M$ is a sub prequantic semilattice of the prequantale $2^M$, and the set $\K(2^M) \setmin \{\emptyset\}$ of all finite nonempty subsets of $M$ is a sub multiplicative semilattice of $\K(2^M)$.
%\item If $S$ is a join semilattice, then $S$ is a Scott-topological multiplicative semilattice under the operation $\vee$ of supremum.
\end{example}

\begin{proposition}\label{semiquantle}
The following are equivalent for any ordered magma $M$.
\begin{enumerate}
\item $M$ is a prequantic semilattice (resp., multiplicative semilattice).
\item $a (\bigvee X) = \bigvee (a X)$ and $(\bigvee X) a = \bigvee(Xa)$ for any $a \in M$ and any finite subset (resp., any finite nonempty subset) $X$ of $M$.
\item The map $\K(2^M) \longrightarrow M$ (resp., $\K(2^M)\setmin\{\emptyset\} \longrightarrow M$) acting by $X \longmapsto \bigvee X$ is a well-defined magma homomorphism.
\item $\bigvee (X Y) = \bigvee X \bigvee Y$ for all finite subsets (resp., all finite nonempty subsets) $X$ and $Y$ of $M$.
\end{enumerate}
\end{proposition}

%Note that all prequantales, near prequantales, and prequantic semilattices are multiplicative semilattices.  Also note the following.

%\begin{proposition}\label{semilatticemaximal} Every maximal ideal of a unital multiplicative %semilattice is prime.
%\end{proposition}

%\begin{proof} Let $x$ be a maximal ideal.  Let $a$ and $b$ be ideals such that $ab \leq x$ but %$a \nleq x$.  Let $y = \bigvee(x,a)$.  It follows that $x < y \leq 1$, whence $y = 1$, so that $b = %\bigvee(x,a)b = \bigvee(xb, ab) \leq x$. Thus $x$ is a prime ideal.
%\end{proof}

The prequantic semilattices (resp., multiplicative semilattices) form a category, where a morphism is a magma homomorphism $f: M \longrightarrow M'$ such that $f(\bigvee X) = \bigvee f(X)$ for all finite subsets (resp., all finite nonempty subsets) $X$ of $M$.  

\begin{proposition}\label{preq}
A prequantale (resp., near prequantale, semiprequantale) is equivalently a complete (resp., near sup-complete, bounded complete) Scott-topological prequantic semilattice (resp., multiplicative semilattice, multiplicative semilattice).  
\end{proposition}

\section{Nuclei}\label{sec:MCO}

Let $\star$ be any self-map of a set $S$. For any $X \subset S$ we let $X^\star = \{x^\star: \ x \in X\}$.  %An element $x$ of $S$ is said to be {\it $\star$-closed} if $x^\cl = x$.
If $S$ is a poset, then $\star$ is a {\it preclosure} on $S$ if $\star$ is expansive and order-preserving.  Thus a closure operation is equivalently an idempotent preclosure.  An {\it interior operation} on $S$ is a closure operation on the dual poset $S^{\operatorname{op}}$ of $S$.    

\begin{lemma}\label{starlemma}
Let $\star$ be a closure operation on a poset $S$ and let $X \subset S$.
\begin{enumerate}
\item $\bigvee_{S^\star} X^\star = (\bigvee_S X^\star)^\star = (\bigvee_S X)^\star$ if $\bigvee_S X$ exists.
\item $\bigwedge_{S^\star} X^\star =  (\bigwedge_S X^\star)^\star = \bigwedge_S X^\star$ if $\bigwedge_S X^\star$ exists.
\item The map $\star: S \longrightarrow S^\star$ is sup-preserving.
\item If $S$ is complete (resp., near sup-complete, bounded complete, directed complete, a bdcpo, a join semilattice, a meet semilattice, a lattice), then so is $S^\star$.
\end{enumerate}
\end{lemma}

A closure operation $\star: S \longrightarrow S$ may not be Scott continuous, even though the corestriction ${_{S^\star}}|\star: S \longrightarrow S^\star$ of $\star$ to $S^\star$ is always sup-preserving.  If $\star: S \longrightarrow S$ is Scott continuous then we will say that $\star$ is {\it finitary}.  (Such closure operations are often said to be {\it algebraic}.)  Equivalently this means that $(\bigvee \Delta)^\star = \bigvee(\Delta^\star)$, or $\bigvee(\Delta^\star) \in S^\star$, for any directed subset $\Delta$ of $S$ such that $\bigvee \Delta$ exists.  This generalizes the notion of a finite type star or semistar operation, ideal system, or module system.

Recall that a {\it nucleus} (resp., {\it strict nucleus}) on an ordered magma $M$ is a closure operation $\star$ on $M$ such that $x^\star y^\star \leq (xy)^\star$ (resp., $x^\star y^\star = (xy)^\star$) for all $x,y \in M$.  The following elementary results give several characterizations of nuclei.

\begin{proposition}\label{closureprop1}
The following conditions are equivalent for any closure operation $\cl$ on an ordered magma $M$.
\begin{enumerate}
\item $\cl$ is a nucleus on $M$.
\item $(x^\cl y^\cl)^\cl = (xy)^\cl$ for all $x,y \in M$.
\item $x y^\cl \leq (xy)^\cl$ and  $x^\cl y \leq (xy)^\cl$ for all $x,y \in M$.
\end{enumerate}
If $M$ is an ordered monoid, then the above conditions are equivalent to the following.
\begin{enumerate}
\item[(4)] $\cl$-multiplication on $M$ is associative.
\end{enumerate}
\end{proposition}

\begin{proposition}\label{closureprop1a}
The following are equivalent for any self-map $\cl$ of a left or right unital  ordered magma $M$.
\begin{enumerate}
\item $\cl$ is a nucleus on $M$.
\item $xy \leq z^\cl  \Leftrightarrow  x y^\cl \leq z^\cl  \Leftrightarrow x^\cl y \leq z^\cl$ for all $x,y,z \in M$.
\item $x \leq x^\star$, and $xy \leq z^\cl \Rightarrow x^\cl y^\cl \leq z^\cl$, for all $x,y,z \in M$.
\end{enumerate}
\end{proposition}

If $\star$ is a nucleus on an ordered magma $M$, then the set $M^\star$ is an ordered magma under $\star$-multiplication, the corestriction ${_{M^\star}}|\star : M \longrightarrow M^\star$ of $\star$ to $M^\star$ is a sup-preserving morphism of ordered magmas, and if $1$ is an identity element of $M$ then $1^\star$ is an identity element of $M^\star$.   We will say that a subset $\Sigma$ of $M$ is a {\it sup-spanning subset} of $M$ if $$xy = \bigvee\{ay: a \in \Sigma \mbox{ and } a \leq x\} = \bigvee\{xb: b \in \Sigma \mbox{ and } b \leq y\}$$ for all $x,y \in M$.

%The following result is one reason why sup-spanning subsets are useful for studying nuclei.

\begin{proposition}\label{closureprop2}
Let $\cl$ be a closure operation on an ordered magma $M$ and let $\Sigma$ be any sup-spanning subset of $M$.  Then $\cl$ is a nucleus on $M$ if and only if $ax^\cl \leq (ax)^\cl$ and $x^\cl a \leq (xa)^\cl$ for all $a \in \Sigma$ and all $x \in M$, and in that case $\Sigma^\cl$ is a sup-spanning subset of $M^\cl$.
\end{proposition}

\begin{proof}
Necessity of the given condition is clear.  Suppose that $ax^\cl \leq (ax)^\cl$ and $x^\cl a \leq (xa)^\cl$ for all $a \in \Sigma$ and all $x \in M$.  Then
$xy^\cl = \bigvee_{a \in \Sigma: \ a \leq x} ay^\cl \leq \left(\bigvee_{a \in \Sigma: \ a \leq x} ay\right)^\cl = (xy)^\cl,$ and by symmetry $x^\cl y \leq (xy)^\cl$, for all $x,y \in M$, so that $\cl$ is a nucleus.  That $\Sigma^\star$ is a sup-spanning subset of $M^\star$ follows easily from Lemma \ref{starlemma}(1).
\end{proof}

For any self-map $\star$ of a magma $M$ we say that $a \in M$ is {\it transportable through $\star$} if $(ax)^\star = ax^\star$ and $(xa)^\star = x^\star a$ for all $x \in M$, and we let $\T^\star(M)$ denote the set of all elements of $M$ that are transportable through $\star$.

\begin{corollary}\label{joinspan}
Any closure operation $\star$ on an ordered magma $M$ such that $\T^\star(M)$ is a sup-spanning subset of $M$ is a nucleus on $M$.
\end{corollary}

%For a (nonassociative) magma or ordered magma there are several nonequivalent notions of an ``invertible'' element.  For our purposes the following will be the most useful.  
For any magma $M$ we let $\Inv(M)$ denote the set of all $u \in M$ for which there exists $u^{-1} \in M$ such that the left and right multiplication by $u^{-1}$ maps are inverses, respectively, to the left and right multiplication by $u$ maps.
%We let $\U(M)$ denote the set of all $u \in \Inv(M)$ such that $u(xy) = (ux)y$ and $(xy)u = x(yu)$ for all $x,y \in M$.  Then $\U(M) \neq \emptyset$ if and only if $M$ is unital, in which case $\U(M)$ is a subgroup of $M$.  If $M$ is a monoid then $\U(M) = \Inv(M)$ is the group of units of $M$.
For any ordered magma $M$ we let $\U(M)$ denote the set of all $u \in M$ such that the left and right multiplication by $u$ maps are poset automorphisms of $M$.  One has $\Inv(M) \subset \U(M)$ for any ordered magma $M$, with equality holding if $M$ is a monoid.

\begin{proposition}\label{closureprop3}
One has $\Inv(M) \subset \T^\star(M)$ for any nucleus $\cl$ on an ordered magma $M$.
\end{proposition}

\begin{proof}
If $u \in \Inv(M)$, then $x^\cl = u^{-1}(u x^\cl) \leq u^{-1}(ux)^\cl \leq (u^{-1}(ux))^\cl   = x^\cl$, whence $ux^\cl = (ux)^\cl$, for all $x \in M$.  By symmetry one has $u \in \T^\star(M)$.
\end{proof}

\begin{corollary}
The only nucleus on a partially ordered abelian group is the identity operation.  More generally, the only nucleus on an ordered unital magma $M$ with $M = \Inv(M)$ is the identity operation.
\end{corollary} 

We will say that a {\it $\U$-lattice} (resp., {\it near $\U$-lattice}, {\it semi-$\U$-lattice}) is a sup-magma (resp., near sup-magma, ordered magma that is a bounded complete join semilattice) $M$ such that $\U(M)$ is a sup-spanning subset of $M$.  %We define a {\it $\V$-lattice} (resp., {\it near $\V$-lattice}, and {\it semi-$\V$-lattice}) similarly.  Note that the former properties, respectively, imply the latter, and they are equivalent for ordered monoids.

\begin{example} For any integral domain $D$ with quotient field $F$, the multiplicative lattice $\mbox{Mod}_D(F)$ of all $D$-submodules of $F$ is a $\U$-lattice, and the near multiplicative lattice $\F(D)$ of all nonzero $D$-submodules of $F$ is a near $\U$-lattice.
%$2^{G_0} \setmin 2^G$ is a $\U$-lattice for any abelian group $G$.
\end{example}

\begin{proposition} An ordered magma $M$ is a prequantale (resp., near prequantale, semiprequantale) if and only if $M$ is complete (resp., near sup-complete, a bounded complete join semilattice) and the set of all $a \in M$ such that the left and right multiplication by $a$ maps on $M$ are sup-preserving (resp., near sup-preserving, near sup-preserving) is a sup-spanning subset of $M$.
\end{proposition}

\begin{corollary}\label{ulattices}
Any $\U$-lattice (resp, near $\U$-lattice, semi-$\U$-lattice) is a prequantale (resp.,  near prequantale, semiprequantale).  %Moreover, a $\U$-lattice (resp., {\it near $\U$-lattice}, {\it semi-$\U$-lattice}) is equivalently an associative and unital $\V$-lattice (resp., {\it near $\V$-lattice}, {\it semi-$\V$-lattice}).
\end{corollary}

\begin{proposition}\label{CSTstar}
Let $\star$ be a nucleus on an ordered magma $M$.
\begin{enumerate}
\item If $M$ is a prequantale (resp., near prequantale, semiprequantale, quantale, near quantale, semiquantale, multiplicative lattice, near multiplicative lattice, semimultiplicative lattice, $\U$-lattice, near $\U$-lattice, semi-$\U$-lattice, multiplicative semilattice, prequantic semilattice), then so is $M^\star$.
\item If $M$ is near residuated, then so is $M^\star$, and $(x/y)^\star = x/y^\star = x/y$ (resp., $(y \backslash x)^\star = y^\star \backslash x = y \backslash x$) for all $x \in M^\star$ and all $y \in M$ such that $zy \leq x$ (resp., $yz \leq x$) for some $z \in M$.  Likewise, if $M$ is residuated, then so is $M^\star$.
\end{enumerate}
\end{proposition}

\begin{proof}  Suppose that $M$ is a prequantale.  By Lemma \ref{starlemma}(4) the partial ordering  on $M^\star$ is complete.  For any $a \in M^\star$ and $X \subset M^\star$ we have $a \star  {\bigvee}_{M^\star} X = a \star ({\bigvee}_M X)^\star = (a \ {\bigvee}_M X)^\star = ( {\bigvee}_M a X)^\star = ({\bigvee}_M (a X)^\star)^\star  = ({\bigvee}_M (a \star X))^\star =     {\bigvee}_{M^\star} (a \star X)$, and therefore $a \star {\bigvee}_{M^\star} X = {\bigvee}_{M^\star} (a \star X)$.   By symmetry the corresponding equation holds for right multiplication.  Thus $M^\star$ is a prequantale.  The proof of the rest of statement (1) is similar.  Statement (2) follows from the fact that $yz \leq x \Leftrightarrow y^\star z \leq x \Leftrightarrow y z^\star \leq x \Leftrightarrow y^\star z^\star \leq x$ if $x \in M^\star$.
\end{proof}

By following proposition, which is a straightforward generalization of \cite[Theorem 3.1.1]{ros}, the set of all nuclei on a near prequantale or prequantale $Q$ classifies the quotient objects of $Q$.

\begin{proposition}\label{supremark}
Let $f: Q \rightarrow M$ be a morphism of near sup-magmas (resp., sup-magmas), where $Q$ is a near prequantale (resp., prequantale).
\begin{enumerate}
\item If $\star$ is any nucleus on $Q$, then ${_{Q^\star}}|\star: Q \longrightarrow Q^\star$ is a surjective morphism of near prequantales (resp., prequantales).
\item There exists a unique nucleus $\star = \star(f)$ on $Q$ such that $f = (f|_{Q^\star}) \circ ({_{Q^\star}}|\star)$ and $f|_{Q^\star}$ is injective; moreover, one has $x^{\star(f)} = \bigvee\{y \in Q: \ f(y) = f(x)\}$ for all $x \in Q$.
\item $f|_{Q^\star}: Q^\star \longrightarrow M$ is an embedding of near sup-magmas (resp., sup-magmas).
\item $f|_{Q^\star}: Q^\star \longrightarrow \operatorname{im} f$ is an isomorphism of ordered magmas.
\end{enumerate}
\end{proposition}
 
A quantale $Q$ is {\it simple} if every nontrivial sup-preserving semigroup homomorphism from $Q$ is injective.  More generally, we will say that a near sup-magma $M$ is {\it simple} if every nonconstant near sup-preserving magma homomorphism from $M$ is injective.  If $M$ is a sup-magma, then this holds if and only if every nontrivial sup-preserving magma homorphism from $M$ is injective, so our definition is consistent with that of a simple quantale.  Let $d$ be the identity map on $M$, and let $e: M \longrightarrow M$ be defined by $x^e = \bigvee M$ for all $x \in M$.  Both $d$ and $e$ are nuclei.

\begin{corollary}\label{desimple}
A near prequantale $Q$ is simple if and only if $d$ and $e$ are the only nuclei on $Q$.
\end{corollary}

%A characterization of simple multiplicative lattices will be given in Section \ref{sec:DC}.

\section{The poset of nuclei}\label{sec:COCL}

Let $S$ be a poset and $X$ a set.  The set $S^X$ of all functions from $X$ to $S$ is partially ordered, where $f \leq g$ if $f(x) \leq g(x)$ for all $x \in X$.  If $f \leq g$, then we say that $f$ is {\it smaller}, or {\it finer}, than $g$, or equivalently $g$ is {\it larger}, or {\it coarser}, than $f$.  The set $\C(S)$ of all closure operations on $S$ inherits a partial ordering from the poset $S^S$.  For any $\star_1, \star_2 \in \C(S)$ one has $\star_1 \leq \star_2$ if and only if $S^{\star_1} \supset S^{\star_2}$.  The identity operation $d$ on $S$ is the smallest closure operation on $S$.  If $\bigvee S$ exists then there is a largest closure operation $e$ on $S$, given by $x^e = \bigvee S$ for all $x \in S$.   If $M$ is an ordered magma, then we let $\N(M)$ denote the subposet of $\C(M)$ consisting of all nuclei on $M$.   In this section we show that various properties of the poset $\N(M)$ are inherited from corresponding properties of $M$. 

For any self-map $\star$ of a set $S$ we let $\Fix(\star) = \{x \in S: \ x^\star = x\}$.

\begin{lemma}\label{preclosurelemma}
Let $S$ be a bounded complete poset, and let $+$ be a preclosure on $S$ that is bounded above by some closure operation on $S$.
\begin{enumerate}
\item There exists a finest closure operation $\star$ on $S$ that is coarser that $+$, and one has $S^\star = \Fix(+)$ and $x^\star = \bigwedge\{y \in \Fix(+): \ y \geq x\}$ for all $x \in S$.
\item Define $x^{+_0} = x$ and $x^{+_\alpha} = \bigvee\{(x^{+_\beta})^+: \ \beta < \alpha\}$ for all $x \in S$ and all (finite or transfinite) ordinals $\alpha$.  One has $+_\alpha = \star$ for all $\alpha \gg 0$.  
\item If $S = M$ is a near residuated ordered magma and $xy^+ \leq (xy)^+$ and $x^+ y \leq (xy)^+$ for all $x,y \in M$, then $\star$ is a nucleus on $M$.
\end{enumerate}
\end{lemma}

\begin{proof} \
\begin{enumerate}
\item The set $\{y \in \Fix(+): \ y \geq x\}$ is nonempty and bounded by the hypothesis on $+$. Define $x^\star = \bigwedge\{y \in \Fix(+): \ y \geq x\}$ for all $x \in S$.  Clearly $\star$ is a preclosure on $S$ with $\Fix(+) \subset S^\star$.  For the reverse inclusion note that
$$(x^+)^\star = \bigwedge\{y  \in \Fix(+): \ y \geq x^+\} = \bigwedge\{y \in \Fix(+): \ y \geq x\} =  x^\star$$
and therefore $x \leq x^+ \leq (x^+)^\star = x^\star$, whence $S^\star \subset \Fix(+)$. 
Therefore
$$(x^\star)^\star  =  \bigwedge\{y  \in \Fix(+): \ y \geq x^\star\} = \bigwedge\{y \in \Fix(+): \ y \geq x\} = x^\star,$$
whence $\star$ is a closure operation on $S$.  It is then clear that $\star$ is the finest closure operation on $S$ that is coarser than $+$.
\item This follows readily from statement (1).  %(Note that $+_\alpha = \star$ for any ordinal $\alpha$ whose cardinality is greater than the cardinality of $S$.)
\item  Let $a,x \in Q$.  Let $z$ be any element of $Q$ such that $z \geq ax$ and $z^+ = z$.  Set $y = a \backslash z$.  Then $y \geq x$ and $y^+ = (a \backslash z)^+ \leq a \backslash z^+ = a \backslash z = y$, hence $y^+ = y$.  Moreover, we have $ay \leq z$.
Therefore we have $$ax^\star \leq \bigwedge\{ay: \ y \geq x, \ y^+ = y\} 
  \leq \bigwedge\{z \in Q: z \geq ax, \  z^+ = z\}
  = (ax)^\star.$$
By symmetry we also have $x^\star a \leq (xa)^\star$, whence $\star$ is a nucleus.
\end{enumerate}
\end{proof}

Let $S$ be a poset and $\Gamma \subset \C(S)$.  Define partial self-maps $\bigsqcapp \Gamma$ and $\bigsqcupp \Gamma$ of $S$ by $$x^{\bigsqcapp\Gamma} = \bigwedge\{x^\star: \ \star \in \Gamma\},$$
$$x^{\bigsqcupp\Gamma} = \bigwedge\{y \in S: \ y \geq x \mbox{ and } \forall \star \in \Gamma \ (y^\star = y)\},$$ respectively, for all $x \in S$ such that the respective infima exist.  

\begin{lemma}\label{closelem}
If $S$ is a near sup-lattice (resp., bounded complete poset, a meet semilattice), then $\C(S)$ is a complete lattice (resp., bounded complete poset, a meet semilattice), and one has the following.
\begin{enumerate}
\item $\bigwedge_{\C(S)} \Gamma = \bigsqcapp \Gamma$ and $S^{\bigsqcupp \Gamma} \supset \bigcup_{\star \in \Gamma} S^\star$ for all subsets (resp., all nonempty subsets, all nonempty finite subsets) $\Gamma$ of $\C(S)$.
\item $\bigvee_{\C(S)} \Gamma = \bigsqcupp \Gamma$ and $S^{\bigsqcupp\Gamma} = \bigcap_{\star \in \Gamma} S^\star$ for all bounded subsets $\Gamma$ of $\C(S)$ if $S$ is bounded complete, and for all subsets $\Gamma$ of $\C(S)$ if $S$ is a near sup-lattice.
\end{enumerate}
\end{lemma}

\begin{proof} \
\begin{enumerate}
\item When $\bigsqcapp \Gamma$ is defined on $S$ it is a preclosure on $S$, and for all $x \in S$ one has $(x^{\bigsqcapp\Gamma})^{\bigsqcapp\Gamma}  = \bigwedge_{\star \in \Gamma} (x^{\bigsqcapp\Gamma})^\star \leq  \bigwedge_{\star \in \Gamma} x^\star = x^{\bigsqcapp\Gamma}.$
Therefore $\bigsqcapp\Gamma$ is a closure operation on $S$, from which it follows that $\bigsqcapp\Gamma = \bigwedge_{\C(S)} \Gamma$.   The given inclusion follows immediately from the definition of $\bigsqcapp \Gamma$.
\item Define $x^+ = \bigvee\{x^\star: \ \star  \in \Gamma\}$ for all $x \in S$.  Clearly $+$ is a preclosure on $S$, and one has $x^+ = x$ if and only if $x^\star = x$ for all $\star \in \Gamma$.  Therefore $\bigsqcupp \Gamma$ is a closure operation on $S$ by Lemma \ref{preclosurelemma}(1), and it follows easily that $\bigsqcupp \Gamma = \bigvee_{\C(S)} \Gamma$.  Let $x \in S$ and $\star \in \Gamma$.  One has
$$(x^\star)^{\bigsqcupp \Gamma} 
	 = \bigwedge\{y \in S: \ y \geq x^\star \mbox{ and } \forall \star' \in \Gamma \ (y^{\star'} = y)\}  \leq x^{\bigsqcupp\Gamma}.$$
Therefore, if $x^{\bigsqcupp \Gamma} = x$, then $x^\star \leq x^{\bigsqcupp\Gamma} = x$, whence $x^\star = x$, for all $\star \in \Gamma$.  Since the converse is clear, we have $S^{\bigsqcupp \Gamma} = \bigcap_{\star \in \Gamma} S^\star$.
\end{enumerate}
\end{proof}

The following result generalizes \cite[Proposition 3.1.3]{ros} and \cite[Example 1.5]{fon2}.

\begin{proposition}\label{CMC}
Let $M$ be an ordered magma.  If $M$ is near sup-complete (resp., bounded complete, a meet semilattice), then $\N(M)$ is a complete lattice (resp., bounded complete poset, a meet semilattice), and one has the following. 
\begin{enumerate}
\item $\bigwedge_{\N(M)} \Gamma = \bigsqcapp \Gamma$ and $M^{\bigsqcapp \Gamma} \supset \bigcup_{\star \in \Gamma} M^\star$ for all subsets (resp., all nonempty subsets, all nonempty finite subsets) $\Gamma$ of $\N(S)$.
\item $\bigvee_{\N(M)} \Gamma = \bigsqcupp\Gamma$ and $M^{\bigsqcupp \Gamma} = \bigcap_{\star \in \Gamma} M^\star$ for all bounded subsets $\Gamma$ of $\N(M)$ if $M$ is bounded complete and near residuated, and for all subsets $\Gamma$ of $\N(M)$ if $M$ is a near prequantale.
\end{enumerate}
\end{proposition}

\begin{proof} \
\begin{enumerate}
\item By Lemma \ref{closelem} it suffices to show that $\bigsqcapp \Gamma$ is a nucleus.  We have
$$xy^{\bigsqcapp\Gamma}   =  x\bigwedge_{\star \in \Gamma} y^\star \leq  \bigwedge_{\star \in \Gamma} xy^\star \leq \bigwedge_{\star \in \Gamma} (xy)^\star = (xy)^{\bigsqcapp\Gamma},$$ and similarly $x^{\bigsqcapp\Gamma}y \leq  (xy)^{\bigsqcapp\Gamma}$, for all $x,y \in M$, whence $\bigsqcapp\Gamma$ is a nucleus.
\item By Lemma \ref{closelem} it suffices to show that $\bigsqcupp \Gamma$ is a nucleus.  Since $d = \bigsqcupp \emptyset$ is a nucleus we may assume that $\Gamma$ is nonempty.  Let $+$ be the preclosure on $M$ defined in the proof of Lemma \ref{closelem}(2).  For all $x,y \in M$ we have
$$xy^+ = x\bigvee_{\star \in \Gamma} y^\star = \bigvee_{\star \in \Gamma} xy^\star \leq \bigvee_{\star \in \Gamma} (xy)^\star = (xy)^+,$$ and similarly $x^+ y \leq (xy)^+$.  Therefore $\star$ is a nucleus by Lemma \ref{preclosurelemma}(3).
\end{enumerate}
\end{proof}

\begin{corollary}
Let $M$ be an ordered magma, and write $M \setmin \{\bigvee M\} = M$ if $\bigvee M$ does not exist.
\begin{enumerate}
\item The inclusion $\N(M) \longrightarrow \C(M)$ is a poset embedding.
\item The map $\C(M) \longrightarrow 2^{M\setmin \{\bigvee M\}}$ acting by $\star \longmapsto M\setmin M^\star$ is a poset embedding, and if $M$ is bounded complete then the map is sup-preserving.
\item If $M$ is bounded complete and near residuated,
then the inclusions $\N(M) \longrightarrow \C(M)$ and $\C(M) \longrightarrow 2^{M\setmin \{\bigvee M\}}$ are sup-preserving poset embeddings.
\end{enumerate}
\end{corollary}

\begin{example} If $L$ is a complete lattice, then by Corollary \ref{characterizingclosures}(1) below the poset embedding $\C(L) \longrightarrow 2^{L \setmin \{\bigvee L\}}$ is an isomorphism if and only if $L$ is well-ordered.
\end{example}

The following result characterizes nuclei on near residuated ordered magmas and generalizes \cite[Lemma 3.33]{gal}.

\begin{proposition}\label{characterizingclosures2}
Let $C$ be a subset of a poset $S$.
\begin{enumerate}
\item There exists a closure operation $\star$ on $S$ with $C = S^\star$ if and only if $\bigwedge\{a \in C: a \geq x\}$ exists in $S$ for all $x \in S$.
\item For any such closure operation $\star$ one has $x^\star = \bigwedge\{a \in C: a \geq x\}$ for all $x \in S$, and therefore $\star = \star_C$ is uniquely determined by $C$. 
\item If $S = M$ is a near residuated ordered magma and $\star_C$ exists, then $\star_C$ is a nucleus on $M$ if and only, for all $x \in C$ and $y \in M$, one has $x/y \in C$ provided $x/y$ exists and $y \backslash x \in C$ provided $y \backslash x$ exists.
\end{enumerate}
\end{proposition}

\begin{proof} Statements (1) and (2) are well-known and easy to check.
We prove (3). If $\star$ is a nucleus, then it follows from Proposition \ref{CSTstar}(2) that $x/y \in C$ and $y \backslash x \in C$ for all $x,y$ as in statement (3).  Conversely, suppose that this condition on $C$ holds.  Let $x, y \in M$.  Then we claim that $x^\star y =  \bigwedge\{a \in C: \ a \geq x\} y$ is less than or equal to $(xy)^\star = \bigwedge\{b \in C: \ b \geq xy\}$.  For let $b \in C$ with $b \geq xy$.  Set $a = b/y$, whence $a \geq x$.  By hypothesis one has $a \in C$.  Therefore $x^\star y \leq ay = (b/y)y \leq b$.  Taking the infimum over all such $b$, we see that $x^\star y \leq \bigwedge\{b \in C: \ b \geq xy\} = (xy)^\star$.  By symmetry one also has $xy^\star \leq (xy)^\star$.  It follows that $\star$ is a nucleus.
\end{proof}

%The following corollary can be used to give an alternative proof of Proposition \ref{CMC}(2) that doesn't require Lemma \ref{closelem}.

\begin{corollary}\label{characterizingclosures}
Let $C$ be a nonempty subset of a poset $S$.
\begin{enumerate}
\item If $S$ is bounded complete, then there exists a closure operation $\star$ on $S$ with $C = S^\star$ if and only if $\bigwedge X \in C$ for all nonempty $X \subset C$ bounded below.
\item If $S = M$ is a bounded complete and near residuated ordered magma, then there exists a nucleus $\star$ on $M$ with $C = M^\star$ if and only if $\bigwedge X \in C$ for all nonempty $X \subset C$ bounded below and  for all $x \in C$ and $y \in M$ one has $x/y \in C$ provided $x/y$ exists and $y \backslash x \in C$ provided $y \backslash x$ exists.
\end{enumerate}
\end{corollary}

The next result describes a common situation in which $\star_1 \vee \star_2 = \star_1 \circ \star_2$ for two nuclei $\star_1$ and $\star_2$. The latter operation, though not necessarily closed, is studied in \cite[Section 2]{pic} and \cite{vas}.

%\begin{lemma}\label{complemma}
%Let $\star_1$ and $\star_2$ be closure operations on a poset $S$.  Then $\star_1 \vee \star_2$ exists in $\C(S)$ and equals the $n$-fold composition $\star = \cdots \circ \star_2 \circ \star_1 \circ \star_2$ for some positive integer $n$ if and only if $\star$ is coarser than the $n$-fold composition $\cdots \circ \star_1 \circ \star_2 \circ \star_1$.
%\end{lemma}

\begin{proposition}\label{complemmacor}
Let $\star_1$ and $\star_2$ be nuclei on an ordered magma $M$.  Then $\star_1  \vee \star_2$ exists in $\N(M)$ and equals the $n$-fold composition $\star = \cdots \circ \star_2 \circ \star_1 \circ \star_2$ for some positive integer $n$ if and only if $\star$ is coarser than the $n$-fold composition $\cdots \circ \star_1 \circ \star_2 \circ \star_1$.
\end{proposition}

If $N$ is a submagma of a near sup-magma $M$, then for any nucleus $\star$ on $N$ we may define
$$\ind_M(\star) = {\bigwedge}_{\N(M)} \{\star' \in \N(M):  \ \star'|_N = \star\},$$
$$\ind^M(\star) = {\bigvee}_{\N(M)} \{\star' \in \N(M):  \ \star'|_N = \star\},$$  both of which are nuclei on $M$.  We say that a subset $X$ of an ordered magma $M$ is {\it saturated} if whenever $xy \in X$ for some $x,y \in M$ that are not annihilators of $M$ one has $x \in X$ and $y \in X$.

\begin{proposition}\label{extendingclosures}
Let $M$ be a near sup-magma and $\star$ a nucleus on a submagma $N$ of $M$.
\begin{enumerate}
\item Suppose that $\{\star' \in \N(M): \ \star'|_N = \star\}$ is nonempty.  Then  $\ind_M(\star)$ is the finest nucleus on $M$ whose restriction to $N$ is equal to $\star$; and if $M$ is a near prequantale then $\ind^M(\star)$ is the coarsest nucleus on $M$ whose restriction to $N$ is equal to $\star$ and one has $\star'|_N = \star$ if and only if $\ind_M(\star) \leq \star' \leq \ind^M(\star)$.
\item If $M$ is a near prequantale and $N$ is a sup-spanning subset of $M$, then
$$x^{\ind_M(\star)} = \bigwedge \{y \in M: \ y \geq x \mbox{ and } \forall z \in N \ (z \leq y \Rightarrow z^\star \leq y)\}$$
for all $x \in M$, and $\ind_M(\star)|_N = \star$.
\item If $N$ is a saturated downward closed subset of $M$, then $$x^{\ind^M(\star)} = \left\{\begin{array}{ll} x^\star & \mbox{if } x \in N \\ \bigvee M & \mbox{otherwise}\end{array}\right.$$ 
for all $x \in M$, and $\ind^M(\star)|_N = \star$.
\end{enumerate}
\end{proposition}

\begin{proof} \
\begin{enumerate}
\item This follows easily from Proposition \ref{CMC}.
\item First we define $x^+ = \bigvee\{y^\star: \ y \in N \mbox{ and } y \leq x\}$ for all $x \in M$.  As $N$ is a sup-spanning subset of $M$ one has $x \leq x^+$ for all $x \in M$.  Since $+$ is order-preserving it follows that $+$ is a preclosure on $M$.  One has $y^+ = y$ if and only if $\forall z \in N \ (z \leq y \Rightarrow z^\star \leq y)$, whence by Lemma \ref{preclosurelemma}(1) it follows that the operation $\star_M$ on $M$ defined by
$$x^{\star_M} = \bigwedge \{y \in M: \ y \geq x \mbox{ and } \forall z \in N \ (z \leq y \Rightarrow z^\star \leq y)\}$$ is a closure operation on $M$.
Now, for all $a \in N$ and $x \in M$ one has
$$ay^+ = \bigvee_{z \in N: \ z \leq y} az^\star
 \leq \bigvee_{z \in N: \ z \leq y} (az)^\star
 \leq \bigvee_{w \in N: \ w \leq ay} w^\star
 = (ay)^+.$$
Since $N$ is a sup-spanning subset of $M$, it follows that
$xy^+ \leq (xy)^+$, and by symmetry $x^+y \leq (xy)^+$, for all $x,y \in M$.  Thus $\star_M$ is a nucleus by Lemma \ref{preclosurelemma}(3).  Now, $+$ is the finest preclosure on $M$ whose restriction to $N$ is $\star$, and by Lemma \ref{preclosurelemma}(1) the operation $\star_M$ is the finest closure operation on $M$ that is coarser than $+$, and one has $(\star_M)|_N = \star$.  It follows that $\star_M$ is the finest closure operation on $M$ whose restriction to $N$ is $\star$.  Therefore, since $\star_M$ is a nucleus we must have $\star_M = \ind_M(\star)$ by statement (1). 
\item It suffices to show that the operation defined in the statement is a nucleus.  This is easy to check.
\end{enumerate}
\end{proof}

\section{The poset of finitary nuclei}\label{sec:PFN}

%We say that a self-map of a poset $S$ is {\it finitary} if it is Scott continuous.
For any poset $S$, let $\C_f(S)$ denote the poset of all finitary closure operations on $S$, and for any ordered magma $M$, let $\N_f(M) = \N(M) \cap \C_f(M)$ denote the poset of all finitary nuclei on $M$.  In this section we show that various properties of the poset $\N_f(M)$ are inherited from corresponding properties of $M$.  A summary of these results and the results for $\N(M)$ from the previous section is provided in Table 2.

\begin{table} 
\caption{The posets $\N(M)$ and $\N_f(M)$}
\centering  \ \\
\begin{tabular}{l|l|l} 
If $M$ is $\ldots$ & $\N(M)$ is $\ldots$ & $\N_f(M)$ is $\ldots$  \\ \hline\hline
near sup-magma & complete & \\ \hline
bounded complete & bounded complete  &  \\ \hline
meet semilattice & meet semilattice & \\ \hline
Scott-topological dcpo magma & & complete \\ \hline
Scott-topological bdcpo magma & & bounded complete \\ \hline
Scott-topological near sup-magma & complete & complete \\ \hline
Scott-topological, bounded complete & bounded complete & bounded complete  \\ \hline
algebraic meet semilattice &  meet semilattice & meet semilattice
\end{tabular}
\end{table}

Recall that an element $x$ of a poset $S$ is said to be {\it compact} if, whenever $x \leq \bigvee \Delta$ for some directed subset $\Delta$ of $S$ such that $\bigvee \Delta$ exists, one has $x \leq y$ for some $y \in \Delta$.  Any minimal element of $S$, for example, is compact.  We let $\K(S)$ denote the set of all compact elements of $S$.

\begin{example} \
\begin{enumerate}
\item For any set $X$, the set $\K(2^X)$ of compact elements of the complete lattice $2^X$ is equal to the set of all finite subsets of $X$.
\item  For any algebra $A$ over a ring $R$, the compact elements of the quantale $\operatorname{Mod}_R(A)$ of all $R$-submodules of $A$ are precisely the finitely generated $R$-submodules of $A$.
%\item One has $\K(\ZZ) = \ZZ$, while $\K(\QQ) = \emptyset$ and $\K(\RR) = \emptyset$.
\end{enumerate}
\end{example}

A element $x$ of a poset $S$ is {\it algebraic} if $x$ is the supremum of a set of compact elements of $S$, or equivalently, if one has $x = \bigvee \{y \in \K(S): \ y \leq x\}$.   We say that a poset is {\it algebraic} if all of its elements are algebraic.  Note, however, that an algebraic complete lattice is said to be an {\it algebraic lattice}.  For example, for any set $X$, the poset $2^X$ is an algebraic lattice, and the poset $2^X \setmin \{\emptyset\}$ is an algebraic near sup-lattice with $\K(2^X \setmin \{\emptyset\}) = \K(2^X) \setmin \{\emptyset\}$.

For any set $X$ and any set $\Gamma$ of self-maps of $X$, let $\langle \Gamma \rangle$ denote the submonoid generated by $\Gamma$ of the monoid of all self-maps of $X$ under composition.

\begin{lemma}
Let $S$ be a poset.
\begin{enumerate}
\item If $S$ is a dcpo (resp., bdcpo) then $\C_f(S)$ is complete (resp., bounded complete) and one has $\bigvee_{\C_f(S)} \Gamma = \bigvee_{\C(S)} \Gamma$ and $x^{\bigvee_{\C(S)} \Gamma} = \bigvee\{x^\gamma: \ \gamma \in \langle \Gamma \rangle\}$ for all $x \in S$ and for any subset (resp., any bounded above subset) $\Gamma$ of $\C_f(S)$. %, where $\langle \Gamma \rangle$ denotes the submonoid generated by $\Gamma$ of the monoid of all self-maps of $S$.
\item If $S$ is a near sup-lattice, then $\C_f(S)$ is a sub sup-lattice of the sup-lattice $\C(S)$.
\item If $S$ is an algebraic meet semilattice, then $\C_f(S)$ is a sub meet semilattice of the meet semilattice $\C(S)$.
\end{enumerate}
\end{lemma}

%\begin{lemma}
%Let $S$ be a dcpo.  Then $\C_f(S)$ is a complete lattice.  Moreover, for any $\Gamma \subset \C_f(S)$, one has $\bigvee_{\C_f(S)} \Gamma = \bigvee_{\C(S)} \Gamma$, and $x^{\bigvee_{\C(S)} \Gamma} = \bigvee\{x^\gamma: \ \gamma \in \langle \Gamma \rangle\}$ for all $x \in S$, where $\langle \Gamma \rangle$ denotes the submonoid  generated by $\Gamma$ of the monoid of all self-maps of $S$.
%\end{lemma}

\begin{proof} \
\begin{enumerate}
\item Define a preclosure $\sigma$ of $S$ by $x^\sigma = \bigvee\{x^\gamma: \ \gamma \in \langle \Gamma \rangle\}$ for all $x \in S$.  The map $\sigma$ is defined because the set $\{x^\gamma: \ \gamma \in \langle \Gamma \rangle\}$ is directed (and bounded above if $\Gamma$ is bounded above) for any  $x \in S$.  Let $x \in S$, and let $X = \{x^\gamma: \ \gamma \in \langle \Gamma \rangle\}$.  Then $(\bigvee X)^\star = \bigvee\{(x^\gamma)^\star: \ \gamma \in \langle \Gamma \rangle\} \leq \bigvee X$ for all $\star \in \Gamma$, whence $(\bigvee X)^\gamma = \bigvee X$, and therefore $(x^\sigma)^\gamma = x^\sigma$, for all $\gamma \in \langle \Gamma \rangle$.  Therefore $(x^\sigma)^\sigma = x^\sigma$, so $\sigma$ is a closure operation on $S$, and clearly $\sigma =\bigvee_{\C(S)} \Gamma$.  It remains only to show that $\sigma$ is finitary, since it will follow that $\sigma = \bigvee_{\C_f(S)} 
\Gamma$.  Let $\Delta$ be a directed subset of $S$.  Then we have
$\left(\bigvee \Delta\right)^\sigma = \bigvee\left\{\left(\bigvee \Delta\right)^\gamma: \ \gamma \in \langle \Gamma \rangle\right\}
  = \bigvee\left\{\bigvee (\Delta^\gamma): \ \gamma \in \langle \Gamma \rangle\right\}  
  = \bigvee \{x^\gamma: \ x \in \Delta, \gamma \in \langle \Gamma \rangle\} 
  = \bigvee (\Delta^\sigma).$
Thus $\sigma$ is finitary.
\item This follows from (1).
\item Let $\star, \star' \in \C_f(S)$.  By Lemma \ref{closelem} it suffices to show that $\star \wedge \star' \in \C(S)$ is finitary.  Let $\Delta$ be a directed subset of $S$ such that $x = \bigvee \Delta$ exists, and let $z = x^{\star \wedge \star'}$.  Note first that $z$ is an upper bound of $\Delta^{\star \wedge \star'}$.  Let $u$ be any upper bound of $\Delta^{\star \wedge \star'}$.  We claim that $u \geq z$.   To show this, let $t$ be any compact element of $S$ with $t \leq z = x^\star \wedge x^{\star'}$.   Then $t \leq x^\star = \bigvee(\Delta^\star)$ and $t \leq x^{\star'} = \bigvee(\Delta^{\star'})$, whence $t \leq y^\star$ and $t \leq (y')^{\star'}$ for some $y,y' \in \Delta$.  Since $\Delta$ is directed, we may choose $w \in \Delta$ with $y,y' \leq w$.  Then $t \leq w^\star \wedge w^{\star'} = w^{\star \wedge \star'} \leq u$.  Therefore, taking the supremum over all compact $t \leq z$, we see that $z \leq u$.  Therefore $(\bigvee \Delta)^{\star \wedge \star'} = z = \bigvee (\Delta^{\star \wedge \star'})$.  Thus $\star \wedge \star'$ is finitary.
\end{enumerate}
\end{proof}

\begin{proposition}\label{Nf}
Let $M$ be an ordered magma.
\begin{enumerate}
\item If $M$ is a Scott-topological dcpo magma (resp., Scott-topological bdcpo magma) then $\N_f(M)$ is complete (resp., bounded complete) and one has $\bigvee_{\N_f(M)} \Gamma = \bigvee_{\N(M)} \Gamma$ and $x^{\bigvee_{\N(M)} \Gamma} = \bigvee\{x^\gamma: \ \gamma \in \langle \Gamma \rangle\}$ for all $x \in M$ and for any subset (resp., any bounded above subset) $\Gamma$ of $\N_f(M)$. %where $\langle \Gamma \rangle$ denotes the submonoid generated by $\Gamma$ of the monoid of all self-maps of $M$.
\item If $M$ is a Scott-topological near sup-magma, then $\N_f(M)$ is a sub sup-lattice of the sup-lattice $\N(M)$.
\item If $M$ is an algebraic meet semilattice, then $\N_f(M)$ is a sub meet semilattice of the meet semilattice $\N(M)$.
\end{enumerate}
\end{proposition}

\begin{proof}
By the lemma it suffices to observe that the closure operation $x \longmapsto \bigvee\{x^\gamma: \ \gamma \in \langle \Gamma \rangle\}$, when defined for $\Gamma \subset \C_f(M)$, is multiplicative if $M$ is Scott-topological.
\end{proof}

\section{Finitary nuclei on precoherent semiprequantales}\label{sec:FN}

In this section we prove Theorem \ref{precoherentthm} of the introduction.

\begin{lemma}\label{joinlemma}
Let $S$ be a join semilattice.  Then $x \in S$ is compact if and only if, whenever $x \leq \bigvee X$ for some subset $X$ of $S$ such that $\bigvee X$ exists, one has $x \leq \bigvee Y$ for some finite subset $Y$ of $X$.  Moreover, $\K(S)$ is closed under finite suprema.  
\end{lemma}

\begin{lemma}\label{finitetype}
One has $\K(S)^\star \subset \K(S^\star)$ for any finitary closure operation $\star$ on a bdcpo $S$.
\end{lemma}

\begin{proof}
Let $x \in \K(S)$.  Suppose that $x^\star \leq \bigvee_{S^\star}\Delta$ for some directed subset $\Delta$ of $S^\star$ such that $\bigvee_{S^\star}\Delta$ exists.  Then $\Delta$ is directed and bounded above, whence $\bigvee_S \Delta$ exists.  It follows that $\bigvee_{S^\star}\Delta = (\bigvee_S \Delta)^\star = \bigvee_S (\Delta^\star) = \bigvee_S \Delta$.  Therefore $x \leq x^\star \leq \bigvee_S \Delta$, so $x \leq y$ for some $y \in \Delta$, whence $x^\star \leq y^\star = y$.  Thus we have $x^\star \in \K(S^\star)$.
\end{proof}

%The following lemma characterizes the finitary closure operations on an algebraic join semilattice.

\begin{lemma}\label{finitetypeprop}
Let $\star$ be a closure operation on a poset $S$.
\begin{enumerate}
\item If $x^\star = x \vee \bigvee\{y^\star: \ y \in \K(S), \ y \leq x\}$ for all $x \in S$, then $\star$ is finitary.
\item If $\star$ is finitary and $S$ is an algebraic join semilattice, then $x^\star = \bigvee\{y^\star: \ y \in \K(S), \ y \leq x\}$ for all $x \in S$, and one has $\K(S^\star) \subset \K(S)^\star$.
\item If $\star$ is finitary, and if $S$ is an algebraic bounded complete join semilattice (resp., algebraic near sup-lattice, algebraic lattice), then so is $S^\star$, and $\K(S^\star) = \K(S)^\star$.
\end{enumerate}
\end{lemma}

\begin{proof} \
\begin{enumerate}
\item Let $\Delta$  be a directed subset of $S$ such that $x = \bigvee \Delta$ exists.   We wish to show that $\bigvee (\Delta^\star)$ exists and equals $x^\star$.  Clearly $x^\star$ is an upper bound of $\Delta^\star$, and if $w$ is any upper bound of $\Delta^\star$, then we claim that $x^\star \leq w$.  Let $y \in \K(S)$ with $y \leq x$.  Then $y \leq z$ for some $z \in \Delta$.  Therefore $y^\star \leq z^\star$, where $z^\star \in \Delta^\star$.   Thus we have $y^\star \leq w$.  Taking the supremum over all such $y$, we see that $\bigvee\{y^\star: \ y \in \K(S), \ y \leq x\} \leq w$.  Moreover, we have $x = \bigvee \Delta \leq w$, and therefore $x^\star = x \vee \bigvee\{y^\star: \ y \in \K(S), \  y \leq x\} \leq w$, as claimed.
\item Let $x \in S$.  One has $x = \bigvee \Delta$, where $\Delta =  \{y \in \K(S): \ y \leq x\}$ is directed by Lemma \ref{joinlemma}.  Therefore $x^\star = \bigvee (\Delta^\star)$.   Suppose  that $x \in \K(S^\star)$.  Then $x = \bigvee(\Delta^\star)$ and $\Delta^\star \subset S^\star$ is directed, so $x \leq y^\star$ for some $y \in \Delta$, whence $x = y^\star$.  Therefore $\K(S^\star) \subset \K(S)^\star$.
\item Since $\K(S)^\star \subset \K(S^\star)$ by Lemma \ref{finitetype}, it follows from (2) that $S^\star$ is algebraic and $\K(S)^\star = \K(S^\star)$.  Since $S^\star$ is bounded complete (resp., near sup-complete, complete) by Lemma \ref{starlemma}(4), it follows that $S^\star$ is an algebraic bounded complete join semilattice (resp., algebraic near sup-lattice, algebraic lattice). 
\end{enumerate}
\end{proof}

For any closure operation $\star$ on a bounded complete poset $S$, let $\star_f$ denote the operation on $S$ defined by $x^{\star_f} = \bigvee\{y^\star: \ y \in \K(S) \mbox{ and } y \leq x\}$ for all $x \in S$.

\begin{proposition}\label{finitaryclosure}
Let $\star$ be a closure operation on an algebraic bounded complete join semilattice $S$.  Then $\star_f$ is the coarsest finitary closure operation on $S$ that is finer than $\star$, one has $x^{\star_f} = x^\star$ for all $x \in \K(S)$, and $\K(S^{\star_f}) = \K(S)^{\star_f}$.
\end{proposition}

\begin{proof}
Clearly $\star_f$ is a preclosure on $S$.  Let $x \in S$.  The set $\{y^\star: \ y \in \K(S), \ y \leq x\}$ is directed by Lemma \ref{joinlemma}.  Therefore if $z \in \K(S)$ and $z \leq x^{\star_f}$ then $z \leq y^\star$ for some $y \in \K(S)$ such that $y \leq x$, in which case $z^\star \leq y^\star$.  Therefore $(x^{\star_f})^{\star_f} =  \bigvee \{z^\star: \ z \in \K(S), \ z \leq x^{\star_f}\} \leq \bigvee \{y^\star: \ y \in \K(S), \ y \leq x\} = x^{\star_f}.$
Thus $\star$ is a closure operation on $S$, and the rest of the proposition follows from Lemma \ref{finitetypeprop}.
\end{proof}

Generalizing \cite[Definition 4.1.1]{ros}, we say that an ordered magma $M$ is {\it precoherent} if $M$ is algebraic and $\K(M)$ is closed under multiplication, and we say that $M$ is {\it coherent} if $M$ is precoherent and unital with $1$ compact.

\begin{example} \
\begin{enumerate}
\item For any magma $M$, the prequantale $2^M$ is precoherent, and if $M$ is unital then $2^M$ is coherent.
\item A {\it $\K$-lattice} is a precoherent multiplicative lattice.  We say that a {\it near $\K$-lattice} is a precoherent near multiplicative lattice.  For example, if $M$ is a commutative monoid, then $2^M$ is a $\K$-lattice and $2^M\setmin\{\emptyset\}$ is a near $\K$-lattice.
\item For any integral domain $D$ with quotient field $F$, the ordered monoid $\F(D)$ of all nonzero $D$-submodules of $F$ is a near $\K$-lattice.
\item Let $A$ be an algebra over a ring $R$.  The quantale $\operatorname{Mod}_R(A)$ of all $R$-submodules of $A$ is precoherent, and it is coherent if and only if $A$ is finitely generated as an $R$-module.  In particular, if $R$ and $A$ are commutative then $\operatorname{Mod}_R(A)$ is a $\K$-lattice.
%\item The near multiplicative lattice $M = \RR_{\geq 1} \cup \{0, \infty\}$ under addition has $\K(M) = \{0,1\}$, so $\K(M)$ is not closed under addition.
\end{enumerate}
\end{example}

%Let us say that a {\it semiprequantale} is a bounded complete multiplicative semilattice $M$ such that the multiplication map $M \times M \longrightarrow M$ is near sup-preserving.  Equivalently, a semiprequantale is a bounded complete Scott-topological multiplicative semilattice.  For example, any near prequantale is a semiprequantale.  Further examples will be given in Section \ref{sec:ASO}.

The following result generalizes the corresponding well-known facts for star and semistar operations and ideal and module systems.

\begin{theorem}\label{klattice}
If $\star$ is a nucleus on a precoherent semiprequantale (resp., precoherent near prequantale, precoherent prequantale) $Q$, then $\star_f$ is the coarsest finitary nucleus on $Q$ that is finer than $\star$, the ordered magma $Q^{\star_f}$ is also a precoherent semiprequantale (resp., precoherent near prequantale, precoherent prequantale), and $\K(Q^{\star_f}) = \K(Q)^{\star_f}$.
\end{theorem}

\begin{proof}
By Proposition \ref{finitaryclosure}, to prove the first claim we need only show that $\star_f$ is a nucleus.  Let $x \in Q$ and $a \in \K(Q)$.  We have $(ax)^{\star_f} = \bigvee \{y^\star: \ y \in \K(Q), \ y \leq ax\}$ and $ax^{\star_f} = \bigvee \{az^\star: \ z \in \K(Q), \ z \leq x\}$.  Let $z \in \K(Q)$ with $z \leq x$, and let $y = az$.  By our hypotheses on $Q$ we have $y \in \K(Q)$, and $y = az \leq ax$.  Therefore $az^\star \leq y^\star 
\leq (ax)^{\star_f} $.  Taking the supremum over all such $z \in \K(Q)$ we see that $ax^{\star_f} \leq (ax)^{\star_f}$.  By symmetry we also have $x^{\star_f}a \leq (xa)^{\star_f}$.  By Proposition \ref{closureprop2}, then, it follows that $\star_f$ is a nucleus.

To prove the second claim we may assume $\star = \star_f$ is finitary.  Then $Q^\star$ is a an algebraic semiprequantale (resp., algebraic near prequantale, algebraic prequantale) by Lemma \ref{finitetypeprop}(3) and Proposition \ref{CSTstar}(1).  Let $x,y \in \K(Q^\star)$, and suppose that $x \star y \leq \bigvee_{Q^\star} \Delta$ for some directed subset $\Delta$ of $Q^\star$ such that $\bigvee_{Q^\star} \Delta$ exists.  Then $\bigvee_Q \Delta$ exists, and $xy \leq \bigvee_{Q^\star} \Delta = (\bigvee_Q \Delta)^\star = \bigvee_Q \Delta$.  Therefore, since $xy \in \K(Q)$, we have $xy \leq z$, whence $x \star y \leq z$, for some $z \in \Delta$.  Thus $x \star y \in \K(Q^\star)$  %This proves that $\K(Q^\star)$ is closed under $\star$-multiplication.
and $Q^\star$ is precoherent.  Finally, the third claim follows from Proposition \ref{finitaryclosure}.
\end{proof}

%Note that, if $M$ is a precoherent semiprequantale, then the map $-_f: \N(M) \longrightarrow \N(M)$ is an interior operation on the poset $\N(M)$.

%\begin{corollary}\label{compactlyrep}
%If $M$ is a precoherent near prequantale (resp., precoherent prequantale), then $M^\star$ is a %precoherent near prequantale (resp., precoherent prequantale) with $\K(M^\star) = \K(M)^\star$ %for any finitary nucleus $\star$ on $M$.
%\end{corollary}

The following result reveals two subclasses of the precoherent semiprequantales.

\begin{proposition}\label{ucoherent}
Let $M$ be an ordered monoid.  Then $M$ is a near $\U$-lattice (resp., semi-$\U$-lattice) with $1$ compact if and only if $M$ is complete (resp., a bounded complete join semilattice), coherent, and every compact element of $M$ is the supremum of a subset of $\U(M)$; in that case the compact elements of $M$ are precisely the suprema of the finite subsets of $\U(M)$. 
\end{proposition}

\begin{corollary}
Any near $\U$-lattice (resp., semi $\U$-lattice) that is associative and unital with $1$ compact is a coherent near quantale (resp., coherent semiquantale).
\end{corollary}

Proposition \ref{ucoherent} follows easily from the following lemma, which generalizes the fact that every invertible fractional ideal of an integral domain is finitely generated.

\begin{lemma}\label{1compact}
For any ordered magma $M$ one has $\U(M) \K(M) \cup \K(M) \U(M) \subset \K(M)$, and if $M$ is unital then the following conditions are equivalent.
\begin{enumerate}
\item $1 \in \K(M)$.
\item $\U(M) \subset \K(M)$.
\item $\U(M) \cap \K(M) \neq \emptyset$.
\end{enumerate}
\end{lemma}

\begin{proof}
The given inclusion is clear, and the implications $(1) \Rightarrow (2) \Rightarrow (3) \Rightarrow (1)$ follow.
\end{proof}

\section{Precoherent prequantales and multiplicative semilattices}\label{sec:RML}

In this section we apply the theory of nuclei to construct a representation of any precoherent near prequantale (resp., precoherent prequantale) as the ideal completion of some multiplicative semilattice (resp., prequantic semilattice), and conversely a representation of any multiplicative semilattice (resp., prequantic semilattice) as the ordered magma of the compact elements of some precoherent near prequantale (resp., precoherent prequantale).  These representations yield appropriate category equivalences.

A subset $I$ of a poset $S$ is said to be an {\it ideal} of $S$ if
$I$ is a directed downward closed subset of $S$.  For any $x \in S$, the set
$$\downarrow\!\! x= \{y \in S: \ y \leq x\}$$ is an ideal of $S$ called  the {\it principal ideal generated by $x$}.  The {\it ideal completion} $\Idl(S)$ of $S$ is the set of all ideals of $S$ partially ordered by the subset relation.  
If $S$ is a join semilattice, then, for any nonempty subset $X$ of $S$, we let $$\downarrow \!\! X = \{y \in S: \ y \leq \bigvee T \mbox{ for some finite nonempty } T \subset X\}.$$

\begin{lemma}\label{downarrowlemma}
Let $S$ be a join semilattice.
\begin{enumerate}
\item For any nonempty subset $X$ of $S$, the set $\downarrow\!\! X$ is the smallest ideal of $S$ containing $X$.
\item The operation $\downarrow$ is a finitary closure operation on the algebraic near sup-lattice $2^S \setmin \{\emptyset\}$ with $\operatorname{im}\!\! \downarrow \ = \Idl(S)$. 
\item If $S$ has a least element, then the operation $\downarrow$ extends uniquely to a finitary closure operation on the algebraic sup-lattice $2^S$ such that $\downarrow\!\! \emptyset = \bigwedge S$.
\end{enumerate}
\end{lemma}

\begin{proof}
Clearly $\downarrow\!\! X$ is nonempty and downward closed.  Let $y, y' \in \downarrow\!\! X$.  Say $y \leq \bigvee T$ and $y' \leq \bigvee T'$, where $T, T'$ are finite nonempty subsets of $X$.  Then
$y \vee y' \leq \bigvee(T \cup T')$, so $y \vee y' \in \downarrow\!\! X$.  Therefore $\downarrow \!\! X$ is an ideal of $S$.  Let $I$ be any ideal of $S$ containing $X$.  For any $x \in \downarrow \!\! X$, one has $x \leq \bigvee T$ with $T \subset X$ finite and nonempty, so $\bigvee T \leq z$ for some $z \in I$, whence $x \in I$ since $x \leq z$.   Thus $\downarrow \!\! X \subset I$.  Therefore $\downarrow \!\! X$ is the smallest ideal of $S$ containing $X$ and $\downarrow$ is a closure operation on $2^S \setmin\{\emptyset\}$.  To show that $\downarrow$ is finitary, we show that $\downarrow \!\! X = \bigcup\downarrow \!\! \{T:\ T \subset X \mbox{ finite and nonempty}\}$ for any $X \in 2^S \setmin\{\emptyset\}$.  Let $x \in \downarrow \!\! X$.  Then $x \leq \bigvee T$, where $T \subset X$ is finite and nonempty. But $\bigvee T \in \downarrow \!\! T$, so $x \in \downarrow \!\! T$.  This proves (1) and (2), and (3) is clear.
\end{proof}

The following result generalizes \cite[Proposition I-4.10]{gie}.

\begin{proposition}\label{algebraicprop2}
Let $L$ be an algebraic near sup-lattice and $S$ a join semilattice.
\begin{enumerate}
\item $\Idl(S)$ is an algebraic near sup-lattice.
\item $\K(L)$ is a sub join semilattice of $L$.
\item The map $S \longrightarrow \K(\Idl(S))$ acting by $x \longmapsto \downarrow\!\! x$ is a poset isomorphism.
\item The map $\Idl(\K(L)) \longrightarrow L$ acting by $I \longmapsto \bigvee I$ is a poset isomorphism with inverse acting by $x \longmapsto (\downarrow \!\! x) \cap \K(L)$.
\end{enumerate}
\end{proposition}

\begin{proof}
By Lemma \ref{downarrowlemma} the map $\downarrow$ is a finitary nucleus on the algebraic near sup-lattice $2^S \setmin \{\emptyset \}$ with $\operatorname{im} \downarrow \ = \Idl(S)$.  By Lemma \ref{finitetypeprop}, then, it follows that $\Idl(S)$ is an algebraic near sup-lattice with $\K(\Idl(S)) = \downarrow \!\!(\K(2^S\setmin \{\emptyset\})) = \{\downarrow\!\! x: x \in S\}$.  This implies (1) and (3).  Next, $\K(L)$ is nonempty and therefore a sub join semilattice of $L$ by Lemma \ref{joinlemma}.  This proves (2).  Finally, the map $f: I \longmapsto \bigvee I$ in (4) is well-defined because $L$ is a near sup-lattice, the map $f$ is clearly order-preserving, and the map $x \longmapsto (\downarrow\!\! x) \cap \K(L)$ is an order-preserving inverse to $f$ because $L$ is algebraic and $\K(L)$ is a join semilattice.
\end{proof}

Theorems \ref{maintheorem} and \ref{maintheorem2} below generalize \cite[Proposition 4.1.4]{ros} and are analogues of Proposition \ref{algebraicprop2} for precoherent near prequantales and precoherent prequantales, respectively.  First, we show that the ideal completion of a multiplicative semilattice (resp., prequantic semilattice) is a precoherent near prequantale (resp., precoherent prequantale) under the operation of {\it $\downarrow$-multiplication}, defined by  $(I,J) \longmapsto \downarrow\!\! (IJ)$.

\begin{lemma}\label{replemma} One has the following.
\begin{enumerate}
\item Let $M$ be a multiplicative semilattice.  The operation $\downarrow: X \longmapsto \downarrow\!\! X$ is a finitary nucleus on the precoherent near prequantale $2^M \setmin\{\emptyset\}$, one has $\Idl(M) = \downarrow\!\! (2^M \setmin\{\emptyset\})$, and $\Idl(M)$ is a precoherent near prequantale under $\downarrow$-multiplication.
\item Let $M$ be a prequantic semilattice.  The operation $\downarrow: X \longmapsto \downarrow\!\! X$ is a finitary nucleus on the precoherent prequantale $2^M$, one has $\Idl(M) = \downarrow\!\! (2^M)$, and $\Idl(M)$ is a precoherent prequantale under $\downarrow$-multiplication.
\end{enumerate}
\end{lemma}

\begin{proof}
We prove statement (1).  The proof of (2) is similar.  Let $X,Y \in 2^M \setmin\{\emptyset\}$, and let $z \in \downarrow\!\! X$ and $w \in \downarrow\!\! Y$.  Then $z \leq \bigvee S$ and $w \leq \bigvee T$ for some finite nonempty sets $S \subset X$ and $T \subset Y$.  Therefore $zw \leq \bigvee S \bigvee T = \bigvee (ST)$, where $ST$ is a finite nonempty subset of $XY$, so $zw \in \downarrow\!\! (XY)$.  Thus we have $\downarrow\!\! X \downarrow\!\! Y \subset \downarrow\!\! (XY)$, so $\downarrow$ is a finitary nucleus on $2^M \setmin\{\emptyset\}$ by Lemma \ref{downarrowlemma}.  By Theorem \ref{finitaryclosure}, then, $\Idl(M)$ is a precoherent near prequantale.
\end{proof}

The morphisms in the {\it category of precoherent near prequantales} (resp., {\it category of precoherent prequantales}) are morphisms $f: Q \longrightarrow Q'$ of near prequantales (resp., prequantales), with $Q$ and $Q'$ precoherent, such that $f(\K(Q)) \subset \K(Q')$.

\begin{theorem}\label{maintheorem}
Let $Q$ be a precoherent near prequantale and let $M$ be a multiplicative semilattice.
\begin{enumerate}
\item $\Idl(M)$ is a precoherent near prequantale under $\downarrow$-multiplication.
\item $\K(Q)$ is a multiplicative semilattice.
\item The map $M \longrightarrow \K(\Idl(M))$ acting by $x \longmapsto \downarrow\!\! x$ is an isomorphism of ordered magmas.
\item The map $\Idl(\K(Q)) \longrightarrow Q$ acting by $I \longmapsto \bigvee I$ is an isomorphism of ordered magmas.
\item If $f: Q \longrightarrow Q'$ is a morphism of precoherent near prequantales, then the map $\K(f): \K(Q) \longrightarrow \K(Q')$ given by $\K(f)(x) = f(x)$ for all $x \in \K(Q)$ is a morphism of multiplicative semilattices.
\item If $g: M \longrightarrow M'$ is a morphism of multiplicative semilattices, then the map $\Idl(g): \Idl(M) \longrightarrow \Idl(M')$ given by $\Idl(g)(I) = \downarrow\!\! (g(I))$ for all $I \in \Idl(M)$ is a morphism of precoherent near prequantales.
\item The associations $\K$ and $\Idl$ are functorial and provide an equivalence of categories between category of precoherent near prequantales and the category of multiplicative semilattices.
\end{enumerate}
\end{theorem}

\begin{proof}
Statement (1) follows from Lemma \ref{replemma}, and (2) is clear.  To prove (3), note that the map $M \longrightarrow \K(\Idl(M)) \subset \Idl(M)$ is the composition $M \longrightarrow  2^M\setmin\{\emptyset\} \longrightarrow \ \downarrow\!\! (2^M \setmin \{\emptyset\}) = \Idl(M)$ of magma homomorphisms and is therefore a magma homomorphism.  Thus it is an isomorphism of ordered magmas, by Proposition \ref{algebraicprop2}(3).  The map in statement (4) is a magma homomorphism by Proposition \ref{nearprequantales} and therefore is an isomorphism of ordered magmas by Proposition \ref{algebraicprop2}(4).  Statement (5) is clear.  To prove (6), first note that $g(\downarrow\!\! X) \subset \ \downarrow\!\! (g(X))$ for any nonempty subset $X$ of $M$.   The map $\Idl(g)$ is then easily shown to be a morphism of near prequantales, and since $\Idl(g)(\downarrow\!\! x) = \downarrow\!\! (g(x))$ for all $x \in Q$, it follows from (3) that $\Idl(g)$ is a morphism of precoherent near prequantales.  Finally, (7) follows from (1) through (6).
\end{proof}

\begin{corollary} The following are equivalent for any ordered magma $M$.
\begin{enumerate}
\item $M$ is a precoherent near prequantale.
\item $M$ is isomorphic to the ideal completion $\Idl(N)$ under $\downarrow$-multiplication of some multiplicative semilattice $N$.
\item $\K(M)$ is a sub multiplicative semilattice of $M$ and $M$ is isomorphic to the ordered magma $\Idl(\K(M))$ under $\downarrow$-multiplication.
\item $M$ is isomorphic to $(2^N \setmin\{\emptyset\})^\star$ for some multiplicative semilattice $N$ and some finitary nucleus $\star$ on $2^N \setmin\{\emptyset\}$.
\end{enumerate}
\end{corollary}

\begin{corollary} An ordered magma $M$ is a multiplicative semilattice if and only if $M$ is isomorphic to $\K(Q)$ for some precoherent near prequantale $Q$.
\end{corollary}

Similar proofs of the above results yield the following.

\begin{theorem}\label{maintheorem2}
One has the following.
\begin{enumerate}
\item If $f: Q \longrightarrow Q'$ is a morphism of precoherent prequantales, then the map
$\K(f): \K(Q) \longrightarrow \K(Q')$ given by $\K(f)(x) = f(x)$ for all $x \in \K(Q)$ is a morphism of prequantic semilattices.
\item If $g: M \longrightarrow M'$ is a morphism of prequantic semilattices, then the map $\Idl(g): \Idl(M) \longrightarrow \Idl(M')$ given by $\Idl(g)(I) = \downarrow\!\! (g(I))$ for all $I \in \Idl(M)$ is a morphism of precoherent prequantales.
\item The associations $\K$ and $\Idl$ are functorial and provide an equivalence of categories between category of precoherent prequantales and the category of prequantic semilattices.
\end{enumerate}
\end{theorem}

%This amounts to showing that
%$$(xy)\!\!\downarrow  \cap \K(Q) = ((\downarrow\!\! x  \cap \K(Q)) (y\!\!\downarrow  \cap %\K(Q)))\!\!\downarrow$$
%for all $x,y \in Q$.  Let $z \in \K(Q)$.  Now, $z$ is in the right hand side of the above %equation if and only if $z \leq \bigvee S$ for some finite subset $S$ of $(\downarrow\!\! x  %\cap \K(Q)) (y\!\!\downarrow  \cap \K(Q))$, say, some subset $S = \{x_1 y_1, \ldots, x_n %y_n\}$, where each $x_i \in (\downarrow\!\! x  \cap \K(Q))$ and each $y_i \in %(y\!\!\downarrow  \cap \K(Q))$.  In that case one has $z \leq \bigvee(TU)$, where $T = %\{x_1,\ldots,x_n\}$ and $U = \{y_1,\ldots,y_n\}$. 
%Thus, $z$ is in the right hand side if and only if $z \leq \bigvee(TU) = \bigvee T \bigvee U$ for %some finite subsets $T$ and $U$ of $\downarrow\!\! x  \cap \K(Q)$ and $y\!\!\downarrow  %\cap \K(Q)$, respectively, if and only if $z \leq x'y'$ for some $x',y' \in \K(Q)$ such %that $x' \leq x$ and  $y' \leq y$.  By Lemma \ref{CRn}, this is equivalent to $z \in %(xy)\!\!\downarrow  \cap \K(Q)$.
%\end{proof}

\begin{corollary} The following are equivalent for any ordered magma $M$.
\begin{enumerate}
\item $M$ is a precoherent prequantale.
\item $M$ is isomorphic to the ideal completion $\Idl(N)$ under $\downarrow$-multiplication of some prequantic semilattice $N$.
\item $\K(M)$ is a sub prequantic semilattice of $M$ and $M$ is isomorphic to the ordered magma $\Idl(\K(M))$ under $\downarrow$-multiplication.
\item $M$ is isomorphic to $(2^N)^\star$ for some prequantic semilattice $N$ and some finitary nucleus $\star$ on $2^N$.
\end{enumerate}
\end{corollary}

\begin{corollary} An ordered magma $M$ is a prequantic semilattice if and only if $M$ is isomorphic to $\K(Q)$ for some precoherent prequantale $Q$.
\end{corollary}

The results in this section also specialize to coherent prequantales and unital prequantic semilattices, as well as to coherent near prequantales and unital multiplicative semilattices.  Likewise, all of these results specialize to the corresponding settings where all operations in question are associative and/or commutative.

\section{Divisorial closure and simple near prequantales}\label{sec:DC}

In this section we generalize the generalized divisorial closure semistar operations \cite[Example 1.8(2)]{pic}, and we use this to derive a characterization of the simple near prequantales that is of a different vein than the characterization in \cite[Theorem 2.5]{krum} of the simple quantales.  If $a$ is an element of an ordered magma $M$ and there exists a coarsest nucleus $\star$ on $M$ such that $a^\star = a$, then we denote it by $v(a) = v_M(a)$ and call it {\it divisorial closure on $M$ with respect to $a$}.  For example, if $\bigvee M$ exists then $v(\bigvee M)$ exists and equals $e = \bigvee \N(M)$. 

\begin{proposition}\label{divprop}
Let $M$ be a near sup-magma.
\begin{enumerate}
\item $v(a)$ exists for a given $a \in M$ if and only if  $a^{v'} = a$, where $v' = \bigvee \{\star \in \N(M): \ a^\star = a\}$, in which case $v(a) = v'$.
\item If $\star$ is a nucleus on $M$ such that $v(a)$ exists for all $a \in M^\star$, then $\star = \bigwedge \{v(a): \ a \in M^\star\}$.
\end{enumerate}
\end{proposition}

\begin{proof}
Statement (1) is clear.  To prove (2), let $\star \in \N(M)$, and let $\star' = \bigwedge \{v(a): \ a \in M^\star\}$.  For any $x \in M^\star$ one has $x \leq x^{\star'} = \bigwedge\{x^{v(a)}:\ a \in M^\star\} \leq x$, whence $M^\star \subset M^{\star'}$, so $\star \geq \star'$.  Moreover, for any $a \in M^\star$ one has $a^\star = a$ and therefore $\star \leq v(a)$, so $\star \leq \star'$.  It follows that $\star = \star'$.
\end{proof}

The following result follows from Propositions \ref{divprop} and \ref{CMC}(2).

\begin{proposition}\label{vclosureprop}
If $Q$ is a near prequantale, then $v(a)$ exists and equals $\bigvee \{\star \in \N(Q): \ a^\star = a\}$ for all $a \in Q$, and one has $\star = \bigwedge \{v(a): \ a \in Q^\star\}$ for all $\star \in \N(Q)$.
\end{proposition}

\begin{corollary}\label{simpleprequantales}
The following are equivalent for any near prequantale $Q$.
\begin{enumerate}
\item $Q$ is simple.
\item The only nuclei on $Q$ are $d$ and $e$.
\item $v(a) = d$ for all $a < \bigvee Q$.
\end{enumerate}
\end{corollary}

If $Q$ is a unital near quantale then we can determine a simple and explicit formula for the divisorial closure operations $v(a)$.

\begin{lemma}\label{vlemma}
Let $Q$ be an ordered monoid and $a \in Q$.  If for any $x \in Q$ there is a largest element $x^\star$ of $Q$ such that $rxs \leq a$ implies $r x^\star s \leq a$ for all $r,s \in Q$, then $v(a)$ exists and $x^{v(a)} = x^\star$ for all $x \in Q$.
\end{lemma}

\begin{proof}
First we show that the map $\star$ defined in the statement of the lemma is a nucleus on $Q$.  Let $x,y \in Q$.  If $x \leq y$, then $rys \leq a \Rightarrow rxs \leq a \Rightarrow rx^{\star}s \leq a$, whence $x^{\star} \leq y^{\star}$.   Likewise $rxs \leq a \Rightarrow rxs \leq a$, whence $x \leq x^{\star}$.  Moreover, one has
$rxs \leq a \Rightarrow rx^{\star} s \leq a \Rightarrow r(x^{\star})^{\star}s \leq a$, whence $(x^{\star})^{\star} \leq x^{\star}$, whence equality holds.  Thus $\star$ is a closure operation on $Q$.  Next, note that $rxys \leq a \Rightarrow rx^{\star}ys \leq a \Rightarrow rx^{\star}y^{\star}s \leq a$, whence $x^{\star}y^{\star} \leq (xy)^{\star}$.  Thus $\star$ is a nucleus on $Q$.

Next we observe that $1a1 \leq a$, which implies $1a^{\star}1 \leq a$, whence $a^{\star} = a$.
Finally, let $\star'$ be any nucleus on $Q$ with $a^{\star'} = a$.  Then for any $x \in Q$ we have
$rxs \leq a \Rightarrow rx^{\star'}s \leq a^\star = a$ for all $r,s \in Q$, whence $x^{\star'} \leq x^\star$.  Thus $\star' \leq \star$.  It follows that $\star$ is the largest nucleus on $Q$ such that $a^{\star} = a$.  Thus $v(a)$ exists and equals $\star$.
\end{proof}

\begin{proposition}\label{vquantales}
Let $Q$ be a unital near quantale and $a \in Q$.  Then $v(a)$ exists, and one has
$$x^{v(a)} = \bigvee\{y \in Q: \ \forall r,s \in Q \ (rxs \leq a \Rightarrow rys \leq a)\}$$
for all $x \in Q$; alternatively, $x^{v(a)}$ is the largest element of $Q$ such
that $rxs \leq a$ implies $rx^{v(a)}s \leq a$ for all $r,s \in Q$.
\end{proposition}

\begin{proof}
Let $x^\star =  \bigvee\{y \in Q: \ \forall r,s \in Q \ (rxs \leq a \Rightarrow rys \leq a)\}$ for all $x \in Q$.  Since $Q$ is a near quantale, one has $y \leq x^\star$ if and only if $rxs \leq a$ implies $r y s \leq a$ for all $r,s \in Q$.  Therefore by Lemma \ref{vlemma} one has $\star = v(a)$.
\end{proof}

\begin{corollary}\label{simplequantales}
A unital near quantale $Q$ is simple if and only if, for any $x,y,a \in Q$ with $a < \bigvee Q$,
one has $x \leq y$ if and only if $rys \leq a$ implies $rxs \leq a$ for all $r,s \in Q$.
\end{corollary}

\begin{corollary}
A multiplicative lattice $Q$ is simple if and only if $Q = \{0,1\}$.
\end{corollary}

\begin{proof}
We may assume $0 < 1$.  Note that $r1s \leq 0$ implies $rxs = rsx \leq 0$ for all $x,r,s \in Q$, whence by Corollary \ref{simplequantales} one has $x \leq 1$ for all $x \in Q$.  Therefore, for any $a \in Q$, one has $(a \vee x)(a \vee y) = a^2 \vee ay \vee xa \vee xy \leq a \vee xy$ for all $x,y \in Q$.  It follows that the map $x \longmapsto a \vee x$ is a nucleus on $Q$, whence $a \vee x = x$ for all $x$ if $a < 1$. It follows, then, that $Q = \{0,1\}$.
\end{proof}

%\begin{example}
%f $D$ is an integral domain with quotient field $F$, then the multiplicative lattice $\operatorname{Mod}_D(F)$ is simple if and only if $D = F$ is a field.  However, by Corollary \ref{desimple} and \cite[Theorem 48]{oka} the near multiplicative lattice $\F(D)$ is simple if and only if $D$ is a field or a DVR.
%\end{example}

Proposition \ref{vquantales} and Corollary \ref{simplequantales} may be generalized to any near prequantale as follows.  Let $M$ be any magma.  For any $r \in M$, define self-maps $L_r$ and $R_r$ of $M$ by $L_r(x) = rx$ and $R_r(x) = xr$ for all $x \in M$, which we call {\it translations}.  Let $\Lin(M)$ denote the submonoid of the monoid of all self-maps of $M$ generated by the translations.  If $M$ is a monoid then any element of $\Lin(M)$ can be written in the form $L_r \circ R_s = R_s \circ L_r$ for some $r,s \in M$.  The proofs above readily generalize to yield the following.

\begin{proposition}\label{vprequantales}
Let $Q$ be a near prequantale and $a \in Q$.  Then one has
$$x^{v(a)} = \bigvee\{y \in Q: \ \forall f \in \Lin(Q) \ (f(x) \leq a \Rightarrow f(y) \leq a)\}$$
for all $x \in Q$; alternatively, $x^{v(a)}$ is the largest element of $Q$ such
that $f(x) \leq a$ implies $f(x^{v(a)}) \leq a$ for all $f \in \Lin(Q)$.
\end{proposition}

\begin{corollary}
A near prequantale $Q$ is simple if and only if, for any $x,y,a \in Q$ with $a < \bigvee Q$,
one has $x \leq y$ if and only if $f(y) \leq a$ implies $f(x) \leq a$ for all $f \in \Lin(Q)$.
\end{corollary}

The following notion of cyclic elements generalizes \cite[Definition 3.3.2]{ros}.  Let $M$ be an ordered magma.  We say that $a \in M$ is {\it cyclic} if $xy \leq a$ implies $yx \leq a$ for all $x,y \in M$.   If $M$ is associative, then $a$ is cyclic if and only if $x_1 x_2 x_3 \cdots x_n \leq a$ implies $x_2 x_3 \cdots x_n x_1 \leq a$ for any positive integer $n$ and any $x_1, x_2, \ldots, x_n \in M$.  If $M$ is near residuated, then we let $M_a^L$ (resp., $M_a^R$) denote the set of all $x \in M$ for which $a/x$ (resp., $x \backslash a$) is defined and we let $M_a^{LR} = M_a^L \cap M_a^R$.   In that case $a$ is cyclic if and only if $M_a^L = M_a^R$ and $a/x = x\backslash a$ for all $x \in M_a^{LR}$.  The following result generalizes \cite[Lemma 3.35]{gal} and the fact that $I^{v(D)} = (I^{-1})^{-1}$ for any integral domain $D$ with quotient field $F$ and any $I \in \F(D)$, where $I^{-1} = (D:_F I)$.

\begin{proposition}\label{vclosureprop3}
Let $M$ be a residuated ordered monoid, or a near residuated ordered monoid such that $\bigvee M$ exists.  Then $v(a)$ exists for any cyclic element $a$ of $M$ and is given by $$x^{v(a)} = \left\{\begin{array}{ll} a/(a/x) & \mbox{if } x \in M_a^{LR} \\ \bigvee M & \mbox{otherwise}\end{array}\right.$$ for all $x \in M$.
\end{proposition}

\begin{proof}
For any $x \in M$, let $x^\star = a/(a/x)$ if $a \in M_a^{LR}$ and $x^\star = \bigvee M$ otherwise.  We claim that $x^\star$ is the largest element $y$ of $M$ such that $rxs \leq a$ implies $rys \leq a$ for all $r,s \in M$.  By Lemma \ref{vlemma} the proposition follows from this claim.  First, note that $rxs \leq a$ implies $srx \leq a$, which implies that  $x \in M_a^{LR}$.  Therefore, if $x \notin M_a^{LR}$, then $\bigvee M$ is the largest element $y$ of $M$ described above.  Suppose, on the other hand, that $x \in M_a^{LR}$.  Then
$rxs \leq a \Rightarrow srx \leq a \Rightarrow sr \leq a/x \Rightarrow (a/(a/x))sr  \leq (a/(a/x))(a/x) \leq a \Rightarrow x^\star sr \leq a \Rightarrow rx^\star s \leq a$.  Suppose that $y'$ is any element of $M$ such that $rxs \leq a \Rightarrow ry's \leq a$ for all $r,s \in M$.  Then since $(a/x)x1 \leq a$ we have $(a/x)y' \leq a$, whence $y'(a/x) \leq a$ and thus $y' \leq a/(a/x) = x^\star$.  This proves our claim.
\end{proof}

\begin{corollary}\label{simplenearmult}
A near multiplicative lattice $Q$ is simple if and only if $x/y$ exists and $x/(x/y) = y$ for all $x,y \in Q$ such that $x,y < \bigvee Q$.  In particular, a multiplicative lattice $Q$ is simple if and only if $x/(x/y) = y$ for all $x,y < \bigvee Q$.
\end{corollary}

\begin{example}
Let $G$ be a bounded complete partially ordered abelian group.  %(for example, $G = \ZZ^X$ or $G = \RR^X$, where $X$ is any set).
Let $G[\infty] = G \amalg \{\infty\}$, where $a \infty = \infty a = \infty$ and $a < \infty$ for all $a \in G$.  Then $G[\infty]$ is a simple near multiplicative lattice, by Corollary \ref{simplenearmult}.  However, the multiplicative lattice
$G[\pm \infty] = G[\infty] \amalg \{-\infty\}$, where $-\infty = \bigwedge G$ is an annihilator of $G[\infty]$, is not simple.  For example, if $D$ is a DVR with quotient field $F$, then $\F(D) \cong \ZZ[\infty]$ is simple, but $\operatorname{Mod}_D(F) \cong \ZZ[\pm \infty]$ is not simple.
\end{example}

Recall that a {\it near $\U$-lattice} is a near sup-magma $Q$ such that $\U(Q)$ is a sup-spanning subset of $Q$.  We end this section by generalizing the well-known fact that $I^{v(D)} = \bigcap \{xD : x \in F \mbox{ and } xD \supset I\}$ for any integral domain $D$ with quotient field $F$ and any $I \in \F(D)$.

\begin{proposition}
For any associative unital near $\U$-lattice $Q$ one has $x^{v(a)} = \bigwedge\{uav: \ u,v \in \U(Q) \mbox{ and } x \leq uav\}$ for all $a,x \in Q$.
\end{proposition}

\begin{proof}
The nucleus $v(a)$ exists by Propositions \ref{vclosureprop} and \ref{ulattices}.  Let $x^\star = \bigwedge\{uav: \ u,v \in \U(Q), \ x \leq uav\}$ for all $x \in Q$.  Clearly $\star$ is a closure operation on $Q$ with $a^\star = a$.  Moreover, for all $x \in Q$ and $w \in \U(Q)$ one has
$wx^\star  =  \bigwedge\{wuav:  \ u,v \in \U(Q), \ wx \leq wuav\} 
	 =  \bigwedge\{u'av:  \ u',v \in \U(Q), \ wx \leq u'av\} 
	 =  (wx)^\star,$
and likewise $x^\star w = (xw)^\star$.  Thus $\star$ is a nucleus on $Q$ by Proposition \ref{closureprop2}.  Now let $x \in Q$.  If $x \leq uav$ for $u,v \in \U(Q)$, then $x^{v(a)} \leq ua^{v(a)}v = uav$, whence $x^{v(a)} \leq x^\star$.  Therefore $v(a) \leq \star$.  Finally, since $a^\star = a$, we have $v(a) \geq \star$ by definition of $v(a)$, so $v(a) = \star$.
\end{proof}

%\begin{lemma}\label{uvlemma2}
%Let $\star$ be a nucleus on an ordered magma $M$.  If $a$ is an element of $M$ such that $v(a)$ exists, then $v(b)$ exists and equals $v(a)$ for all $b \in a \Inv(M) \cup (M)a$. 
%\end{lemma}

%\begin{proof}
%This follows easily from  Proposition \ref{closureprop3}.
%\end{proof}

%\begin{corollary}
%For any near $\U$-lattice $Q$ one has $x^{v(1)} = \bigwedge\{u \in \U(Q): x \leq u\}$ for all $x \in Q$.
%\end{corollary}

\section{Further structure on the poset of nuclei}\label{sec:SAMCO}

In this section we examine further the structure of the poset $\N(M)$ of all nuclei on $M$ for various classes of ordered magmas $M$.  %We show that $\N(M)$ has the structure of a near multiplicative lattice, where the operation on $\N(M)$ is $\vee$.  For example, if $D$ is an integral domain with quotient field $F$, then the set $\F(D)$ of all nonzero $D$-submodules of $F$ has the structure of a near multiplicative lattice, so the set $\N(\F(D))$, which coincides with the set of all semistar operations on $D$, also has such a structure. 
For any magma $M$, let $\Idem(M)$ denote the set of all idempotent elements of $M$. For any ordered unital magma $M$, let $\R(M) = \Idem(M)_{\geq 1} = \{x \in \Idem(M): \ x \geq 1\}$.

\begin{proposition}\label{RMlemma}
Let $M$ be an ordered unital magma.
\begin{enumerate}
\item For any $X \subset \R(M)$ such that $\bigwedge X$ exists one has $\bigwedge X \in \R(M)$.
\item If $M$ is Scott-topological, then for any directed subset $\Delta$ of $\R(M)$ such that $\bigvee \Delta$ exists one has $\bigvee \Delta \in \R(M)$.
\item Suppose that $\R(M)$ is a submagma of $M$ (which holds, for example, if $M = \R(M)$ or if $M$ is a commutative monoid).  Then $\R(M)$ is a Scott-topological multiplicative semilattice in which multiplication coincides with $\vee$;  if $M$ is bounded complete or a Scott-topological bdcpo, then $\R(M)$ is a semimultiplicative lattice; and if $M$ is near sup-complete or a Scott-topological dcpo, then $\R(M)$ is a near multiplicative lattice.
\end{enumerate}
\end{proposition}

\begin{proof} \
\begin{enumerate}
\item We have $\bigwedge X \geq 1$ and $\bigwedge X \leq (\bigwedge X)^2 \leq \bigwedge \{a^2: a \in X\} = \bigwedge X$.
\item If $a, b \in \Delta$, then $a,b \leq c$ for some $c \in \Delta$, whence $ab \leq c^2 = c$, so $(\bigvee \Delta)^2 = \bigvee(\Delta\Delta) \leq \bigvee \Delta$, whence equality holds.  Therefore $\bigvee \Delta \in \R(M)$. 
\item Let $a, b, c \in \R(M)$.  Then $a, b \leq a b$, and if $a, b \leq c$, then $a b \leq c^2 = c$.  Thus $a b = a \vee b$.   Statement (3) follows, then, from statements (1) and (2) and Proposition \ref{preq}.
\end{enumerate}
\end{proof}

%The statements in the following lemma are important for studying residuated partially ordered semigroups.

%\begin{lemma}\label{reslemma}%[{\cite[Lemma 3.15]{gal}}]
%Let $x,y,z$ be elements of a near left residuated (resp., near right residuated) partially ordered %monoid.
%\begin{enumerate}
%\item If either $(x/y)/z$ or $x/(zy)$ (resp., $z \backslash (y \backslash x)$ or $(yz) %\backslash x)$ exists then both exist and are equal.
%\item $x/x$ (resp., $x \backslash x$) exists.
%\item $(x/x)x = x$ (resp., $x(x \backslash x) = x$).
%\item $x/x \in \R(M)$ (resp., $(x \backslash x) \in \R(M)$).
%\item $x/x = x$ (resp., $x \backslash x = x$) if and only if $x \in \R(M)$.
%\end{enumerate}
%\end{lemma}

%For any extension $R \subset S$ of commutative rings and any $R$-submodule $I$ of $S$, the %set $A = (I:_S I)$ is the largest subring of $S$ containing $R$ such that $I$ is an ideal of $A$.  %The lemma above can be used to prove the following generalization of this fact.

%\begin{proposition}
%Let $M$ be an ordered monoid.  If $M$ is near left residuated, then the self-map %$\lambda$ of $M$ given by $\lambda(x) = x/x$ is idempotent, one has $\operatorname{im} %\lambda= \R(M)$, and $\lambda(x)$ for any $x \in M$ is the largest element $a$ of $\R(M)$ %such that $a x = x$.  Likewise, if $M$ is near right residuated, then the self-map $\rho$ of $M$ %given by $\rho(x) = x \backslash x$ is idempotent, one has $\operatorname{im} \rho = \R(M)$, %and $\rho(x)$ for any $x \in M$ is the largest element $a$ of $\R(M)$ such that $x a = x$.
%\end{proposition} 

An {\it order-preserving Galois connection} from a poset $S$ to a poset $T$ is a pair of order-preserving functions $f: S \longrightarrow T$ and $g : T \longrightarrow S$ such that $f(a) \leq b$ if and only if $a \leq g(b)$ for all $a \in S$ and all $b \in T$, or equivalently such that $g \circ f \geq \id_S$ and $f \circ g \leq \id_T$.  In that case $f \circ g \circ f = f$ and $g \circ f \circ g = g$; the map $g \circ f$ is a closure operation on $S$ and $f \circ g$ is an interior operation on $T$; and $f: S \longrightarrow T$ and $g: T^{\operatorname{op}} \longrightarrow S^{\operatorname{op}}$ are sup-preserving.  Moreover, $f$ is injective if and only if $g$ is surjective if and only if $g \circ f = \operatorname{id}_S$, in which case the Galois connection is said to be {\it faithful}.

%\begin{example}
%If $M$ is an ordered magma, then $a \in M$ is residuated if and only if there are self-maps $-/a$ and $a \backslash -$ of $M$ such that the pairs $-a$, $-/a$ and $a-$, $a \backslash -$ are order-preserving Galois connections, where $-a$ and $a-$ denote the right and left multiplication by $a$ maps on $M$, respectively.  In particular, if $M$ is residuated, then the right and left multiplication by $a$ maps are sup-preserving for all $a \in M$.
%\end{example}

Let $M$ be an ordered commutative monoid.  For any $a \in \R(M) = \Idem(M)_{\geq 1}$, the self-map $d_a: x \longmapsto xa$ of $M$ is a (strict) nucleus on $M$.   
We define $d_-: \R(M) \longrightarrow \N(M)$ by $d_-: a \longmapsto d_a$.  Also, we define $d^-: \N(M) \longrightarrow \R(M)$ by $d^-: \star \longmapsto 1^\star$.

%More generally, for any $a \in \R(M)$ and any nucleus $\star$ on $M$, the self-map $\star_a: x \longmapsto (xa)^\star$ of $M$ is a nucleus.  Note that $\star_{a b} = (\star_b)_a = \star_b \circ \star_a $ for all $a, b \in \R(M)$.

\begin{proposition}\label{dalpha}
Let $M$ be an ordered commutative monoid.  The pair $d_-: \R(M) \longrightarrow \N(M)$ and $d^-: \N(M) \longrightarrow \R(M)$ is a faithful order-preserving Galois connection, and in particular $d_-$ is a sup-preserving poset embedding.  Moreover, if $M$ is Scott-topological, then $d_a$ is finitary for all $a \in \R(M)$ and $d_-: \R(M) \longrightarrow \N_f(M)$ is a sup-preserving poset embedding.
\end{proposition}

\begin{proof}
Let $a \in \R(M)$ and $\star \in \N(M)$.  If $d_a \leq \star$, then $a = 1^{d_a} \leq 1^\star$.  Conversely, if $a \leq 1^\star$, then $x^{d_a} = x a \leq x 1^\star \leq x^\star$ for all $x \in M$, whence $d_a \leq \star$.  Thus $d_-, d^-$ is an order-preserving Galois connection, and it is faithful since $d_-$ is injective.  The last statement of the proposition is clear.
\end{proof}

\begin{corollary}\label{galoisconnectioncor}
Let $S$ be a join semilattice with least element.
\begin{enumerate}
\item $S$ is a Scott-topological unital multiplicative semilattice under the operation $\vee_S$ with $S = \R(S)$ and $\N(S) = \C(S)$.
\item For any $a \in S$, one has $d_a = a \vee -$ and $d_a$ is finitary.
\item The pair $d_-: S \longrightarrow \N(S)$ and $d^-: \N(S) \longrightarrow S$ is a faithful order-preserving Galois connection.
\end{enumerate}
\end{corollary}

Combining the above corollary with Lemma \ref{closelem}, we obtain the following.

\begin{proposition}\label{structure}
Let $S$ be a complete lattice. 
\begin{enumerate}
\item $S$ is a near multiplicative lattice under the operation $\vee_S$ with  $S = \R(S)$ and $\N(S) = \C(S)$.
\item $\N(S)$ is a near multiplicative lattice under the operation $\vee_{\N(S)} = \vee_{\C(S)}$ with $\N(S) = \R(\N(S))$ and $\N(\N(S)) = \C(\N(S))$.
\item The pair $d_-: S \longrightarrow \N(S)$ and $d^-: \N(S) \longrightarrow S$ is a faithful order-preserving Galois connection, and $d_-$ is an embedding of near multiplicative lattices.
\end{enumerate}
\end{proposition}

\begin{corollary}\label{alphalemmacor}
If $M$ is a near sup-complete ordered commutative monoid, then the map $d_-: \R(M) \longrightarrow \N(M)$ is an embedding of near multiplicative lattices.
\end{corollary}

For any near sup-magma $M$ the poset $\N(M)$ is complete, so by Proposition \ref{structure} we may set $\N^1(M) = \N(M)$ and $\N^{n+1}(M) = \N(\N^n(M))$ for any positive integer $n$.

\begin{corollary}\label{structure2}
Let $M$ be a near sup-magma and $n$ a positive integer. 
\begin{enumerate}
\item $\N^n(M)$ is a near multiplicative lattice under the operation $\vee$ with $\N^n(M) = \R(\N^n(M))$ and $\N(\N^n(M)) = \C(\N^n(M))$.
\item The pair $d_-: \N^n(M) \longrightarrow \N^{n+1}(M)$ and $d^-: \N^{n+1}(M) \longrightarrow \N^n(M)$ is a faithful order-preserving Galois connection.
\item The map $d_- : \N^n(M) \longrightarrow \N^{n+1}(M)$ is an embedding of near multiplicative lattices.
\end{enumerate}
\end{corollary}

Next, we show that a near prequantale $Q$ is simple if and only if the sequence $\N^n(Q)$ stabilizes, that is, if and only if there is a positive integer $n_0$ such that the map $d_- : \N^n(Q) \longrightarrow \N^{n+1}(Q)$ is an isomorphism for all $n \geq n_0$ (or equivalently for $n = n_0$).

\begin{lemma}\label{desimple2}
Let $M$ be a unital near sup-magma such that $M = \R(M)$.  Then $M$ is simple if and only if
$|M| \leq 2$.
\end{lemma}

\begin{proof}
If $|M| \leq 2$, then $M$ is clearly simple.  If $|M| > 2$, then $1 < a < \bigvee M$ for some $a \in M$ and $a \vee -$ is a nucleus on $M$ with $d < a \vee - < e$, whence $M$ is not simple by Corollary \ref{desimple}.
\end{proof}

%\begin{proof}
%f $M$ is simple, then $\N(M) = \{d,e\}$, which is easily seen to be simple.  If $M$ is not simple, then by Corollary \ref{desimple} one has $d < \star < e$ for some $\star \in \N(M)$.  Then $\bigvee(\star,-)$ is a nucleus on $\N(M)$ with $d < \bigvee(\star,-)< e$, whence $\N(M)$ is not simple.
%\end{proof}

\begin{proposition}
The following are equivalent for any near prequantale $Q$.
\begin{enumerate}
\item $Q$ is simple.
\item The near multiplicative lattice $\N(Q)$ is simple.
\item $|\N(Q)| \leq 2$.
\item The map $d_- : \N(Q) \longrightarrow \N(\N(Q))$ is an isomorphism.
\item The sequence $\N^n(Q)$ stabilizes.
\end{enumerate}
\end{proposition}

\begin{proof}
By Corollary \ref{desimple} and Lemma \ref{desimple2} we have $(1) \Leftrightarrow (2) \Leftrightarrow (3)$, and clearly $(3) \Rightarrow (4) \Rightarrow (5)$.  To show $(4) \Rightarrow (3)$ suppose that $(4)$ holds.  Define a self-map $\sigma$ of $\N(Q)$ by $\sigma(\star) = d$ if $\star = d$ and $\sigma(\star) = e$ otherwise.   Then $\sigma$ is a closure operation, hence a nucleus, on $\N(Q)$.  Therefore $\sigma = \star \vee -$ for some $\star \in \N(Q)$, whence $\star = \sigma(d) = d$, whence $\sigma$ is the identity on $\N(Q)$ and therefore $\N(Q) = \{d,e\}$.  Thus $(4) \Rightarrow (3)$.  Finally, if $(5)$ holds, then the map $\N^n(Q) \longrightarrow \N(\N^n(Q))$ is an isomorphism for some $n$, whence $\N^n(Q)$ is simple since $(4) \Rightarrow (1)$, whence $Q$ is simple since $(2) \Rightarrow (1)$.  Thus $(5) \Rightarrow (1)$.  This completes the proof.
\end{proof}

\section{Compact finitary nuclei}\label{sec:CFN}

It is a natural problem to determine the compact elements of the poset $\N(M)$ of all nuclei on an ordered magma $M$ and to determine for which such $M$ the lattice $\N(M)$ is algebraic.  The same problem may be posed for the poset $\N_f(M)$ of all finitary nuclei on $M$.  In this section we take some preliminary steps towards a solution to these problems.  Recall that $\R(M)$ denotes the set $\Idem(M)_{\geq 1}$ for any ordered unital magma $M$.

\begin{proposition}
Let $Q$ be a unital near quantale.  For any $x \in Q$, let $1[x] = 1 \vee x \vee x^2  \vee \cdots$.
\begin{enumerate}
\item  For all $x \in Q$, the element $1[x]$ is the least element $y$ of $\R(Q)$ such that $x \leq y$.
\item The map $1[-]: x \longmapsto 1[x]$ is a finitary closure operation on $Q$ with image equal to $1[Q] = \R(Q)$.
\item If $Q$ is algebraic, then $1[Q]$ is algebraic and $\K(1[Q]) = 1[\K(Q)]$.
\end{enumerate}
\end{proposition}

\begin{proof}
Let $\Delta_x = \{1,x,x^2,\ldots\}$, so $1[x] = \bigvee(\Delta_x)$.  Note then that $(1[x])^2 = \bigvee(\Delta_x \Delta_x) = \bigvee(\Delta_x) = 1[x]$.  Therefore $1[x] \in \R(Q)$ and $1[x] \geq x$.  Moreover, if $y \in \R(Q)$ and $y \geq x$, then $y \geq 1,x,x^2,\ldots$, whence $y \geq 1[x]$.  This proves (1).  It follows that the map $1[-]$ is a closure operation on $Q$.  Moreover, one has $1[x] = x$ if and only if $x \geq x^n$ for all non-negative integers $n$, if and only if $x \in \R(Q)$, so the image of $1[-]$ is equal to $\R(Q)$.  Also, if $\Delta$ is any directed subset of $Q$, then by Proposition \ref{RMlemma}(2) one has $\bigvee 1[\Delta] \in \R(Q)$, and therefore it follows from (1) that $1[\bigvee \Delta] = \bigvee 1[\Delta]$.  This proves (2).  Finally, since $1[-]$ is finitary, statement (3) follows from Lemma \ref{finitetypeprop}(3).
\end{proof}

\begin{corollary}\label{rmcor}
Let $Q$ be an algebraic near multiplicative lattice.  Then $\R(Q)$ is a coherent near multiplicative lattice equal to $\{1[x]: \ x \in Q\}$, where $1[x] = 1 \vee x \vee x^2 \vee \cdots$ and $1[x]1[y] = 1[x \vee y]$ for all $x,y \in Q$, and $\K(\R(Q)) = \{1[x]: \  x \in \K(Q)\}$.
\end{corollary}

\begin{example}
Let $A$ be an algebra over a ring $R$ and let $Q = \operatorname{Mod}_R(A)$.  Then $1[Q] = \R(Q)$ is the complete lattice of all $R$-subalgebras of $A$, and $\K(1[Q])$ is the poset of all finitely generated $R$-subalgebras of $A$.  It is clear in this case that $1[Q]$ is algebraic.
\end{example}

We say that a morphism $f: S \longrightarrow T$ of posets is {\it coherent} if $f$ is sup-preserving and $f(\K(S)) \subset \K(f(S))$.  %Note that any coherent morphism of posets is Scott continuous.  

\begin{proposition}
Let $M$ be a Scott-topological ordered commutative monoid that is a dcpo (resp., bdcpo).  Then the map $d_-: \R(M) \longrightarrow \N_f(M)$ is a coherent embedding of near multiplicative lattices (resp., bounded complete posets).
\end{proposition}

\begin{proof}
First, $d_-$ is well-defined and sup-preserving by Proposition \ref{dalpha}.  Let $a \in \K(\R(M))$.  Suppose that $d_a \leq \bigvee \Gamma$ for some $\Gamma \subset \N_f(M)$.  Then $a = 1^{d_a} \leq 1^{\bigvee \Gamma} = \bigvee\{1^\gamma: \ \gamma \in \langle \Gamma \rangle\}$ by Proposition \ref{Nf}(1).  Moreover, $1^\star \in \R(M)$ for all $\star \in \Gamma$, whence $1^\gamma \in \R(M)$ for all $\gamma \in \langle \Gamma \rangle$.  Therefore, since $1^\gamma, 1^\delta \leq 1^{\gamma \circ \delta}$ for all $\gamma, \delta \in \langle \Gamma \rangle$, the set $\{1^\gamma:  \ \gamma \in \langle \Gamma \rangle\}$ is a directed subset of $\R(M)$.   Thus $a \leq 1^{\delta}$ for some $\delta \in \langle \Gamma \rangle$.  Writing $\delta = \star_1 \circ \star_2 \circ \cdots \circ \star_n$ for $\star_i \in \Gamma$, we see that $a \leq \bigvee\{1^\gamma: \ \gamma \in \langle \Delta \rangle\} = 1^{\bigvee \Delta}$,
where $\Delta = \{\star_1, \star_2, \ldots, \star_n\}$,  whence $d_a \leq d_{1^{\bigvee \Delta}} \leq \bigvee \Delta$.  Therefore $d_a \in \K(\N_f(M))$.
\end{proof}

\begin{corollary}
Let $Q$ be a near multiplicative lattice.  Then $d_{1[x]} \in \K(\N_f(Q))$ for all $x \in \K(Q)$.
\end{corollary}

If $M$ is a near multiplicative lattice, then the set $\N_{ns}(M)$ of all near sup-preserving nuclei on $M$ is a sub near multiplicative lattice of the near multiplicative lattice $\N_f(M)$ that is closed under arbitrary suprema, and $d_a \in \N_{ns}(M)$ for all $a \in \R(M)$.

\begin{proposition}\label{ns1}
Let $M$ be a near multiplicative lattice.  Then the map $d_-: \R(M) \longrightarrow \N_{ns}(M)$ is a coherent embedding of near multiplicative lattices and is an isomorphism if $M$ is a near $\U$-lattice without a least element.
\end{proposition}

\begin{proof}
Suppose that $M$ is a near $\U$-lattice without a least element, and let $\star \in \N_{ns}(M)$.  Then for any $x \in M$ one has $x = \bigvee\{u \in \U(M): \ u \leq x\}$, whence $x^\star = \bigvee\{u^\star: \ u \in \U(M), \ u \leq x\}  = \bigvee\{ 1^\star u:  \ u \in \U(M), \ u \leq x\} = 1^\star x$.
Therefore $\star = d_{1^\star}$.
\end{proof}

\begin{proposition}\label{ns2}
The following are equivalent for any ordered monoid $M$.
\begin{enumerate}
\item $M$ is coherent, near sup-complete, and $\K(M) \subset \U(M)$.
\item $M$ is a near $\U$-lattice and $\K(M) = \U(M)$.
\item $M$ is a near $\U$-lattice and $u \vee v \in \U(M)$ for all $u,v \in M$.
\end{enumerate}
Moreover, if any of these conditions holds and $M$ is an ordered commutative monoid without a least element, then the map $d_-: \R(M) \longrightarrow \N_f(M) = \N_{ns}(M)$ is an isomorphism of near multiplicative lattices and therefore $\N_f(M)$ is a coherent near multiplicative lattice with $\K(\N_f(M)) = \{d_{1[x]}: \ x \in \K(M)\}$.
\end{proposition}

\begin{proof}
The three conditions are easily seen to be equivalent.  Let $\star \in \N_f(M)$ and $x \in M$.  By hypothesis on $M$ the set $\Delta = \{u \in \U(M): u \leq x\}$ is directed and $x = \bigvee \Delta$.  Since $\star$ is finitary we thus have
$x^\star = \bigvee\{1^\star u: \ u \in \U(M) \mbox{ and } u \leq x\} = 1^\star x$.  Therefore $\star = d_{1^\star}$.  The rest of the proposition follows from Corollary \ref{rmcor}.
\end{proof}

Results similar to Propositions \ref{ns1} and \ref{ns2} hold for the set $\N_s(M)$ of sup-preserving nuclei on a multiplicative lattice $M$ that is a $\U$-lattice.

\begin{example}\label{Pruferexample}
Let $D$ be an integral domain and $M = \F(D)$.  Then $\K(M)$ (resp., $\U(M)$) is the group of finitely generated (resp., invertible) fractional ideals of $D$.  Thus, $M$ is a coherent near multiplicative lattice with $\U(M) \subset \K(M)$, and one has $\K(M) = \U(M)$ if and only if $D$ is a Pr\"ufer domain.   In particular, if $M$ is a Pr\"ufer domain, then the complete lattice $\N_f(M)$ of all finite type semistar operations on $D$ is algebraic, equal to the set $\{d_R: \ R \mbox{ is an overring of } D\}$ and with compact elements
as those $d_R$ with $R$ finitely generated as a $D$-algebra.
\end{example}

\section{Stable nuclei}\label{sec:stable}

A semistar operation $\star$ on an integral domain $D$ is said to be {\it stable} if 
$(I\cap J)^\star = I^\star \cap J^\star$ for all $I,J \in \F(D)$.  If $\star$ is stable then $(I:_F J)^\star = (I^\star :_F J)$ for all $I,J \in \F(D)$ with $J$ finitely generated, where $F$ is the quotient field of $D$.  For any semistar operation $\star$ on $D$ there exists a coarsest stable semistar operation $\overline{\star}$ on $D$ that is finer than $\star$, given by
$$I^{\overline{\star}} = \bigcup\{(I:_F J): \ J \subset D \mbox{ and } J^\star = D^\star\}$$
for all $I \in \F(D)$.  If $\star$ is of finite type then so is $\overline{\star}$, whence $\star_w = \overline{\star_f}$ is the coarsest stable semistar operation of finite type that is finer than $\star$.  The {\it $w$ operation} $w = v_w = \overline{t}$ is the coarsest stable semistar operation on $D$ of finite type such that $D^w = D$. % Along with $v$ and $t$ it is among the most useful of all semistar operations.

In this section we generalize the above results to the context of near multiplicative lattices.  We will say that a nucleus $\star$ on a near residuated ordered magma $M$ is {\it stable} if $(\bigwedge X)^\star = \bigwedge(X^\star)$ for any finite subset $X$ of $M$ such that $\bigwedge X$ exists and $(x/t)^\star = x^\star/t$ (resp., $(t \backslash x)^\star = t \backslash x^\star)$ for all $x,t \in M$ with $t$ compact such that $x/t$ (resp., $t \backslash x$) exists.  In general the first condition above does not necessarily imply the second.  However, by the following proposition, the implication holds if every compact element of $M$ is the supremum of a finite subset of $\Inv(M)$, which in turn holds for any associative unital near $\U$-lattice $M$, such as the near $\U$-lattice $M = \F(D)$.

\begin{proposition}\label{stableU}
Let $M$ be near residuated ordered magma $M$ such that every compact element of $M$ is the supremum of a finite subset of $\Inv(M)$.  A nucleus $\star$ on $M$ is stable if and only if $(\bigwedge X)^\star = \bigwedge(X^\star)$ for any finite subset $X$ of $M$ such that $\bigwedge X$ exists.  Moreover, if $M$ is also a meet semilattice, then $x/t$ and $t \backslash x$ exist in $M$ for all $x,t \in M$ with $t$ compact.
\end{proposition}

\begin{proof}
Suppose that $\star$ distributes over finite meets.  Let $x,t \in M$ with $t$ compact, and suppose that $x/t$ exists.  We may write $t = u_1 \vee u_2 \vee \cdots \vee u_n$ with each $u_i \in \Inv(M)$.  Since $x/u$ exists and equals $xu^{-1}$ and $(x/u)^\star = (xu^{-1})^\star = x^\star u^{-1} = x^\star/u$ for all $u \in \Inv(M)$, it straightforward to check that $x/u_1 \wedge x/u_2 \wedge \cdots \wedge x/u_n$ exists and equals $x/t$.  Therefore we have
$$(x/t)^\star = (x/u_1)^\star \wedge \cdots \wedge (x/u_n)^\star = x^\star/u_1 \wedge \cdots \wedge x^\star/u_n
= x^\star/t.$$
The proof for $t \backslash x$ is similar, so $\star$ is stable.  Also, if $M$ is a meet semilattice and $x,t \in M$ with $t$ compact, then, writing $t = u_1 \vee u_2 \vee \cdots \vee u_n$ with each $u_i \in M$, one easily verifies that $x/t = x/u_1 \wedge x/u_2 \wedge \cdots \wedge x/u_n$ exists, and likewise for $t \backslash x$.
\end{proof}

If $M$ is an ordered unital magma and $\star \in \N(M)$, then, borrowing terminology from the theory of semistar operations, we will say that $z \in M$ is {\it $\star$-Glaz-Vasconcelos}, or {\it $\star$-GV}, if $z \leq 1$ and $z^\star = 1^\star$.  We let $\SGV(M)$ denote the set of all $\star$-GV elements of $M$, which is a submagma of $M$.  If $X \subset \SGV(M)$ is nonempty, then $\bigvee X \in \SGV(M)$ if $\bigvee X$ exists, and $\bigwedge X \in \SGV(M)$ if $\bigwedge X$ exists and $X$ is finite.  If $\star$ is a nucleus on a near multiplicative lattice $Q$, then we let
$$x^{\overline{\star}} = \bigvee\{x/z: \ z \in \SGV(Q)\}$$
for all $x \in Q$, which is well-defined since $x/z$ is defined for all $z \leq 1$.

\begin{lemma}\label{stableprop}
Let $\star$ be a nucleus on a near multiplicative lattice $Q$.
\begin{enumerate}
\item $\overline{\star}$ is a preclosure on $Q$ that is finer than $\star$, and one has
$xy^{\overline{\star}} \leq (xy)^{\overline{\star}}$ and $x^{\overline{\star}}y \leq (xy)^{\overline{\star}}$ for all $x,y \in Q$.
\item For any $x,t \in Q$ with $t$ compact one has $t \leq x^{\overline{\star}}$ if and only if $tz \leq x$ for some $z \in \SGV(Q)$.
\item If $Q$ is algebraic then $(\bigwedge X)^{\overline{\star}} = \bigwedge(X^{\overline{\star}})$ for any finite subset $X$ of $Q$ such that $\bigwedge X$ exists.
\item If $Q$ is precoherent then $\overline{\star}$ is a stable nucleus on $Q$.
\end{enumerate}
\end{lemma}

\begin{proof} \
\begin{enumerate}
\item One has $x^{\overline{\star}} \leq \bigvee\{x^\star/z: \ z \in \SGV(Q)\} =  \bigvee\{x^\star/z^\star: \ z \in \SGV(Q)\} = x^\star$ for all $x \in Q$, whence $\overline{\star}$ is finer than $\star$, and the rest of statement (1) is equally trivial to verify.
\item Let $x,t \in Q$ with $t$ compact.  If $tz \leq x$ for some $z \in \SGV(Q)$, then $t \leq x/z \leq x^{\overline{\star}}$.  Conversely, if $t \leq x^{\overline{\star}}$, then we have $t \leq x/z_1 \vee x/z_2 \vee \cdots \vee x/z_n \leq x/(z_1 z_2 \cdots z_n)$ for some $z_1, z_2, \ldots, z_n \in \SGV(Q)$, whence $tz \leq x$, where $z = z_1 z_2 \cdots z_n \in \SGV(Q)$.
\item Let $X = \{x_1, x_2, \ldots, x_n\}$ be a finite subset of $Q$ such that $a = \bigwedge X$ exists.   We must show that $a^{\overline{\star}} = \bigwedge(X^{\overline{\star}})$.   Clearly $a^{\overline{\star}} \leq x^{\overline{\star}}$ for all $x \in X$.  Let $b$ be any lower bound of $X^{\overline{\star}}$.  If $t$ is any compact element of $Q$ such that $t \leq b$, then, by statement (2), for each $i$ there exists $z_i \in \SGV(Q)$ such that $tz_i \leq x$, whence $tz \leq \bigwedge X = a$, where $z = z_1 z_2 \cdots z_n \in \SGV(Q)$.  Thus $t \leq a^{\overline{\star}}$.  Taking the supremum over all such $t$ we see that $b = \bigvee\{t \in \K(Q): \ t \leq b\} \leq a^{\overline{\star}}$.  Therefore $\bigwedge(X^{\overline{\star}})$ exists and equals $a^{\overline{\star}}$.
\item We first show that $\overline{\star}$ is idempotent and therefore a nucleus on $Q$.  Let $x \in Q$.  Let $t$ be any compact element such that $t \leq (x^{\overline{\star}})^{\overline{\star}}$.  Then
$tz \leq x^{\overline{\star}}$ for some $z \in \SGV(Q)$.   If $u$ is a compact element of $Q$ such that $u \leq z$, then $tu \leq x^{\overline{\star}}$ and $tu$ is compact, whence $tu z_u \leq x$ for some $z_u \in \SGV(Q)$.  Let $z' = \bigvee\{uz_u: \ u \in \K(Q) \mbox{ and } u \leq z\}$.  Then $z' \leq 1$ and $(z')^\star  = (\bigvee\{u z_u^\star: \ u \in \K(Q) \mbox{ and } u \leq z\})^\star = z^\star = 1^\star$, whence $z' \in \SGV(Q)$.  Moreover, one has
$tz' = \bigvee\{tuz_u: \ u \in \K(Q) \mbox{ and } u \leq z\} \leq x$.  It follows that
$t \leq x^{\overline{\star}}$.  Taking the supremum over all $t$, we see that
$(x^{\overline{\star}})^{\overline{\star}} \leq x^{\overline{\star}}$.  Thus
$\overline{\star}$ is a nucleus on $Q$.  To show that $\overline{\star}$ is stable, let $x,t$ be elements of $Q$ with $t$ compact such that $x/t$ exists.  Clearly $x^{\overline{\star}}/t$ also exists and $(x/t)^{\overline{\star}} \leq x^{\overline{\star}}/t$.  Let $u \in \K(Q)$ with $u \leq x^{\overline{\star}}/t$.  Then $tu \leq x^{\overline{\star}}$ and $tu$ is compact, whence $tuz \leq x$ for some $z \in \SGV(Q)$.  Therefore $uz \leq x/t$, whence  $u \leq (x/t)^{\overline{\star}}$.  Taking the supremum over all $u$ we see that $x^{\overline{\star}}/t \leq (x/t)^{\overline{\star}}$, whence equality holds.  Combining this with statement (3), we see that $\overline{\star}$ is stable.
\end{enumerate}
\end{proof}

\begin{theorem}\label{stabletheorem}
Let $Q$ be a precoherent near multiplicative lattice such that every compact element of $Q$ is residuated and $x \wedge 1$ exists for all $x \in Q$, and let $\star \in \N(Q)$.   For any $x \in Q$ let $x^{\overline{\star}} = \bigvee\{x/z: \ z \in \SGV(Q)\}$.
Then $\overline{\star}$ is the coarsest stable nucleus on $Q$ that is finer than $\star$.  Moreover, the following conditions on $\star$ are equivalent.
\begin{enumerate}
\item $\star$ is stable.
\item $(x\wedge 1)^\star = x^\star \wedge 1^\star$ and $(x/t)^\star = x^\star/t$ for all $x,t \in Q$ with $t$ compact.
\item $(x/t \wedge 1)^\star = x^\star/t \wedge 1^\star$ for all $x,t \in Q$ with $t$ compact.
\item $\star = \overline{\star}$.
\end{enumerate}
\end{theorem}

\begin{proof}
We first show that the four statements of the proposition are equivalent.  Clearly we have $(1) \Rightarrow (2) \Rightarrow (3)$, and by Lemma \ref{stableprop}(4) we have $(4) \Rightarrow (1)$.   To show that $(3)$ implies $(4)$, we suppose that $(3)$ holds, and then we need only show that $x^\star \leq x^{\overline{\star}}$ for any $x \in Q$.  Let $t$ be any compact element of $Q$ with $t \leq x^\star$.  By the hypothesis on $Q$ the element $z = x/t \wedge 1$ exists in $Q$.  Condition $(3)$ then implies that $z^\star = x^\star/t \wedge  1^\star$.  Note then that, since $1^\star t \leq 1^\star x^\star = x^\star$ one has $x^\star/t \geq 1^\star$, whence $z^\star = x^\star/t \wedge 1^\star = 1^\star$.  Since also $z \leq 1$ we have $z \in \SGV(Q)$.  Therefore, since $zt \leq (x/t)t \leq x$, we have $t \leq x^{\overline{\star}}$.  Finally, since this holds for all compact $t \leq x^\star$, we have $x^\star \leq x^{\overline{\star}}$, as desired.
It remains only to show that $\overline{\star}$ is the coarsest stable nucleus on $Q$ that is finer than $\star$.  But by Lemma \ref{stableprop} the nucleus $\overline{\star}$ is itself a stable nucleus that is finer than $\star$, and if $\star'$ is any other such nucleus, then one has
$\star' = \overline{\star'} \leq \overline{\star}$, whence $\overline{\star}$ is coarser than $\star'$.
\end{proof}

Any precoherent multiplicative lattice satisfies the hypotheses of Theorem \ref{stabletheorem} above.  By Propositions \ref{ucoherent} and \ref{stableU} and Corollary \ref{ulattices}, so does any unital near $\U$-lattice $Q$ with $1$ compact that is also a commutative monoid and meet semilattice, such as the near $\U$-lattice $Q = \F(D)$, where $D$ is any integral domain.

Let $\star_w = \overline{\star_f}$ for any nucleus $\star$ on a precoherent near multiplicative lattice $Q$.

\begin{proposition}\label{stablecor}
Let $\star$ be a nucleus on a coherent near multiplicative lattice $Q$.
\begin{enumerate}
\item One has $x^{\star_w} = \bigvee\{x/z: \ z \in \SGV(Q)\cap \K(Q)\}$ for all $x \in Q$, and $\star_w$ is finitary.  In particular, if $\star$ is finitary, then so is
$\overline{\star} = \star_w$.
\item If every compact element of $Q$ is residuated and $x \wedge 1$ exists for all $x \in Q$, then $\star_w$ is the coarsest stable finitary nucleus on $Q$ that is finer than $\star$.  
\end{enumerate}
\end{proposition}

\begin{proof}
Statement (2) follows from statement (1) and Theorems \ref{stabletheorem} and \ref{klattice}.  To prove (1), let $x \in Q$.  One has $x^{\star_w} = \bigvee\{x/y: \ y\leq 1 \mbox{ and } y^{\star_f} = 1^{\star_f}\}$.  Let $t$ be any compact element of $Q$ with $t \leq x^{\star_w}$.  Then $t \leq x/y$ for some $y \in Q$ with $y \leq 1$ and $y^{\star_f} = 1^{\star_f}$.   Since $1$ is compact, the condition $y^{\star_f} = 1^{\star_f}$ implies $z^\star = 1^\star$ for some compact $z \leq y$.  Note then that $t \leq x/z$ since $tz \leq ty \leq x$.  Therefore, since $z \in \SGV(Q) \cap \K(Q)$, one
has $t \leq \bigvee\{x/z: \ z \in \SGV(Q)\cap \K(Q)\}$.  Since this holds for all compact $t \leq x^{\star_w}$, it follows that $x^{\star_w} \leq \bigvee\{x/z: \ z \in \SGV(Q)\cap \K(Q)\}$, and therefore equality holds since the reverse inequality is obvious.

Suppose now that $\star = \star_f$ is finitary.  To show that $\overline{\star}$ is also finitary,  we let $x \in Q$ and verify that $x^{\overline{\star}} = \bigvee\{y^{\overline{\star}}: \ y \in \K(Q) \mbox{ and } y \leq x\}$.  Let $t$ be any compact element of $Q$ with $t \leq x^{\overline{\star}}$.   Since $\star_w = \overline{\star}$, one has $x^{\overline{\star}} = \bigvee\{x/z: \ z \in \SGV(Q)\cap \K(Q)\}$.  Since $t \leq x^{\overline{\star}}$ is compact and the set $\{z/x:  \ z \in \SGV(Q)\cap \K(Q)\}$ is directed, it follows that $t \leq x/z$ for some $z \in \SGV(Q) \cap \K(Q)$.  Therefore $tz \leq x = \bigvee\{y \in \K(Q): \ y \leq x\}$,
whence $tz \leq y$ for some compact $y \leq x$ since $tz$ is compact.  Thus we have $t \leq y/z$, where $z \in \SGV(Q)$, whence $t \leq y^{\overline{\star}}$.  It follows that
$x^{\overline{\star}} \leq  \bigvee\{y^{\overline{\star}}: \ y \in \K(Q) \mbox{ and } y \leq x\}$, and the desired equality follows.
\end{proof}

The following result follows from Propositions \ref{vclosureprop} and \ref{stablecor} Theorems \ref{klattice} and \ref{stabletheorem}.

\begin{proposition}
Let $\star$ be a nucleus on a precoherent near prequantale $Q$.
\begin{enumerate}
\item $t(a) = v(a)_f$ for any $a \in Q$ is the coarsest finitary nucleus on $Q$ such that $a^{t(a)} = a$, and $\star_f = \bigwedge\{t(a): \ a \in Q^{\star_f}\}.$
\item If $Q$ is a commutative monoid, every compact element of $Q$ is residuated, and $x \wedge 1$ exists for all $x \in Q$, then $\overline{v}(a)= \overline{v(a)}$ for any $a \in Q$ is the coarsest stable nucleus on $Q$ such that $a^{\overline{v}(a)} = a$, and $\overline{\star} = \bigwedge\{\overline{v}(a): \ a\in Q^{\overline{\star}}\}.$
\item If $Q$ is a commutative monoid, $1 \in Q$ is compact, every compact element of $Q$ is residuated, and $x \wedge 1$ exists for all $x \in Q$, then $w(a) = \overline{t(a)}$ for any $a \in Q$ is the coarsest stable finitary nucleus on $Q$ such that $a^{w(a)} = a$, and $\star_w = \bigwedge\{w(a): \ a \in Q^{\star_w}\}.$
\end{enumerate}
\end{proposition}

\begin{lemma}
Let $Q$ be a near prequantale.  Then $\bigwedge \Gamma$ is a stable nucleus on $Q$ for any set $\Gamma$ of stable nuclei on $Q$.  
\end{lemma}

\begin{proof}
Let $X$ be a finite subset of $Q$ that is bounded below.  Then one has
$\left(\bigwedge X\right)^{\bigwedge \Gamma} = \bigwedge\left\{\left(\bigwedge X\right)^\star: \ \star \in \Gamma\right\} = \bigwedge\left\{\bigwedge (X^\star): \ \star \in \Gamma\right\} = \bigwedge\{x^\star: \ x \in X, \ \star \in \Gamma\} = \bigwedge \left(X^{\bigwedge \Gamma}\right)$.
Moreover, for any $x, t \in Q$ with $t$ compact, if $x/t$ exists, then one has
$(x/t)^{\wedge \Gamma} =  \bigwedge\{(x/t)^\star: \ \star \in \Gamma\} =  \bigwedge\{x^\star/t: \ \star \in \Gamma\} = x^{\bigwedge \Gamma}/t$,
and a similar proof holds for $t \backslash x$ if $t \backslash x$ exists.  Thus
$\bigwedge \Gamma$ is stable.
\end{proof}

\begin{corollary}
Let $Q$ be a precoherent near multiplicative lattice such that every compact element of $Q$ is residuated and $x \wedge 1$ exists for all $x \in Q$.  A nucleus $\star$ on $Q$ is stable if and only if $\star = \bigwedge \{\overline{v}(a): \ a \in X\}$ for some subset $X$ of $Q$.
\end{corollary}

Note that $\bigwedge \Gamma$, by Proposition \ref{Nf}(3), is a finitary nucleus on $M$ for any finite set $\Gamma$ of finitary nuclei on an algebraic meet semilattice $M$.  However, our proof requires that $\Gamma$ be finite.  We therefore naturally ask for a characterization of those subsets $X$ of a precoherent near prequantale $Q$ such that $\bigwedge \{t(a): \ a \in X\}$ is finitary.  We leave this as an open problem.

The results of this section generalize known results on semistar operations.  To also accommodate the theory of star operations one may restate the results in greater generality utilizing the following definitions.  We say that an ordered magma $M$ is {\it semiresiduated} if for all $x,y \in M$ such that the set $\{z \in M: \ zy \leq x\}$ (resp., $\{z \in M: \ yz \leq x\}$) is nonempty and bounded, there exists a largest $z \in M$ such that $zy \leq x$ (resp., $yz \leq x$).  Any semiprequantale or near residuated ordered magma, for example, is semiresiduated.  If throughout this section we replace ``near residuated'' with ``semiresiduated'', and ``near multiplicative lattice'' with ``semimultiplicative lattice'', then the results remain true and our proofs easily generalize.  However, we do not know if the assumptions of associativity, commutativity, and unitality can be eliminated.

\section{Ideal and module systems}\label{sec:AMSIS}

For any magma $M$, let $M_0$ denote the magma $M \amalg \{0\}$, where $0x = 0 = x0$ for all $x \in M_0$.  Recall from the introduction that a {\it module system} on an abelian group $G$ is a closure operation $r$ on the $\K$-lattice $2^{G_0}$ such that $\emptyset^r = \{0\}$ and $(cX)^r = cX^r$ for all $c \in G_0$ and all $X \in 2^{G_0}$.  By Propositions \ref{closureprop1} and \ref{closureprop1a}, Theorem \ref{modulesystems} of the introduction can be proved by showing that a module system on $G$ is equivalently a nucleus $r$ on $2^{G_0}$ such that  $\emptyset^r = \{0\}$.

\begin{proof}[Proof of Theorem \ref{modulesystems}]
Let $\Sigma = \{\{c\}: c \in G_0\}$.  Let $r$ be a module system on $G$.  Since  $\Sigma$ is a sup-spanning subset of $2^{G_0}$ and $\Sigma \subset \T^r(2^{G_0})$, it follows from Corollary \ref{joinspan} that $r$ is a nucleus on $2^{G_0}$.  Conversely, let $r$ be a nucleus on $2^{G_0}$ such that $\emptyset^r = \{0\}$.  Since $\Sigma\setmin\{\{0\}\} \subset \Inv(2^{G_0})$, by Proposition \ref{closureprop2}, we have $(\{c\}X)^r = \{c\}X^r$ for all $c \in G$ and all $X \in 2^{G_0}$.  Also, we have $(\{0\}X)^r = \{0\}^r = \{0\} = \{0\}X^r$ for all $X \in 2^{G_0}$.  Thus $r$ is a module system on $G$.
\end{proof}

Next, a {\it weak ideal system} \cite{hal1} on a commutative monoid $M$ is a closure operation $r$ on the $\K$-lattice $2^{M_0}$ such that $0 \in \emptyset^r$, $cM_0 \subset \{c\}^r$, and $cX^r \subset (cX)^r$ for all $c \in M_0$ and all $X \in 2^{M_0}$.  A weak ideal system on $M$ is said to be an {\it ideal system} on $M$ if $(cX)^r = cX^r$ for all such $c$ and $X$.  We have the following result, whose proof is similar to that of Theorem \ref{modulesystems}.

\begin{proposition}\label{idealystems}
Let $M$ be a commutative monoid.  A weak ideal system on $M$ is equivalently a nucleus $r$ on the $\K$-lattice $2^{M_0}$ such that $\{0\}^r = \emptyset^r$ and $\{1\}^r = M_0$. Moreover, a weak ideal system $r$ on $M$ is an ideal system on $M$ if and only if every singleton in $2^{M_0}$ is transportable through $r$. 
\end{proposition}

A module system on an abelian group $G$ may be seen equivalently as a nucleus on the $\K$-lattice $2^{G_0} \setmin 2^G$.  Since $2^{G_0}$ and $2^{G_0} \setmin 2^G$ are (coherent) $\K$-lattices and $2^{G_0} \setmin 2^G$ is a $\U$-lattice, most of the results of this paper, and in particular the results of Sections \ref{sec:MCO} through \ref{sec:FN} and \ref{sec:DC} through \ref{sec:stable}, apply specifically to module systems.  In the next section we will apply our results to star and semistar operations.  In fact those results may be stated more generally in terms of module systems, since the ordered monoids $2^{G_0}-2^G$ and $\F(D)$, where $D$ is any integral domain, share the relevant properties.  Finally, as $2^{M_0}$ is a $\K$-lattice for any commutative monoid $M$, many of our results apply as well to weak ideal systems.  We leave it to the interested reader to carry out the details.

%For example, as a consequence of Theorem \ref{modulesystems} and Proposition %\ref{extendingclosures}, we have the following.

%\begin{proposition}
%Let $G$ be an abelian group.  Let $\gamma$ be any nucleus on a submonoid $\mathcal{M}$ of %$2^{G_0}$, and suppose that $\emptyset, \{0\} \in \mathcal{M}$ and $\emptyset^\gamma = %\{0\}$.  Then there exists a unique finest module system $r$ on $G$ such that $X^r = %X^\gamma$ for all $X \in \mathcal{M}$.  Moreover, if $\mathcal{M}$ contains all singletons, %then one has
%$$X^r = \bigcap\{Y \in 2^{G_0}: \ Y \supset X \mbox{ and } \forall Z \in \mathcal{M} \ (Z %\subset Y \Rightarrow Z^\gamma \subset Y)\}$$
%for all $X \in 2^{G_0}$.
%\end{proposition} 

\section{Star and semistar operations}\label{sec:ASO}

Throughout this section let $D$ denote an integral domain with quotient field $F$.  A {\it Kaplansky fractional ideal}, or {\it K-fractional ideal}, of $D$ is a $D$-submodule of $F$.  Let $\F(D)$ denote the ordered monoid of all nonzero K-fractional ideals of $D$.  Recall that a {\it semistar operation} on $D$ is a closure operation $\star$ on the poset $\F(D)$ such that $(aI)^\star = aI^\star$ for all nonzero $a \in F$ and all $I \in \F(D)$.  The submonoid $\P(D)$ of $\F(D)$ of all nonzero principal $D$-submodules of $F$ is a sup-spanning subset of $\F(D)$ contained in $\Inv(\F(D))$.  Therefore, by Propositions \ref{closureprop2} and \ref{closureprop3}, a semistar operation on $D$ is equivalently a nucleus on the ordered monoid $\F(D)$.  Together with Propositions \ref{closureprop1} and \ref{closureprop1a}, this yields a proof of Theorem \ref{mainprop1} of the introduction.  Alternatively, the theorem follows from Theorem \ref{modulesystems}, Propositions \ref{closureprop1} and \ref{closureprop1a}, and the following lemma.

\begin{lemma}\label{ssexample} Let $D$ be an integral domain, $F^\times$ the group of nonzero elements of the quotient field $F$ of $D$, and $\star$ a self-map of $\F(D)$.  Let $\emptyset^r = \{0\}^ r = \{0\}$ and $X^r = (DX)^\star$ for all subsets $X$ of $F$ containing a nonzero element, where $DX$ denotes the $D$-submodule of $F$ generated by $X$.  Then $\star$ is a semistar operation on $D$ if and only if $r$ is a module system on $F^\times$.
\end{lemma}

The following is an immediate corollary of Theorem \ref{mainprop1}.

\begin{corollary}\label{latticeisom}
Let $D$ and $D'$ be integral domains.  If the ordered monoids $\F(D)$ and $\F(D')$ are isomorphic, then the lattices of all semistar operations on $D$ and $D'$, respectively, are isomorphic.
\end{corollary}

A semistar operation on $D$ may be defined alternatively as a closure operation on the poset $\operatorname{Mod}_D(F)$ such that $(aI)^\star = aI^\star$ for all $a \in F$ and all $I \in \operatorname{Mod}_D(F)$, or equivalently a nucleus $\star$ on the $\K$-lattice $\operatorname{Mod}_D(F)$ such that $(0)^\star = (0)$.  This alternative definition is advantageous because $\F(D)$ is typically neither complete nor residuated, while every $\K$-lattice is complete and residuated, and the definition thereby eliminates the need for many results, such as \cite[Proposition 5, Theorem 20, Lemma 40]{oka}, to make exception for the zero ideal.  The definition also allows for the possibility of relaxing the condition $(0)^\star = (0)$ by considering all nuclei on $\operatorname{Mod}_D(F)$.  Since $\F(D)$ inherits its ``near'' properties (such as being a near $\K$-lattice and near $\U$-lattice) from corresponding ``complete'' properties of $\operatorname{Mod}_D(F)$, all of the results of this paper generalizing results on semistar operations apply as well to all nuclei on $\operatorname{Mod}_D(F)$.

A $D$-submodule $I$ of $F$ is said to be a {\it fractional ideal} of $D$ if $aI \subset D$ for some nonzero element $a$ of $F$.  Let $\Fp(D)$ denote the ordered monoid of all nonzero fractional ideals of $D$.  % One has $\Fp(D) = \FS(\F(D))$.
Although $\Fp(D)$ is not necessarily a near prequantale, it is a semiprequantale.  Moreover, $\Fp(D)$ is coherent and residuated.  A {\it star operation} $*$ on $D$ is a closure operation on the poset $\Fp(D)$ such that $D^* = D$ and $(aI)^* = aI^*$ for all nonzero $a \in F$ and all $I \in \Fp(D)$.  The same argument as in the proof of Theorem \ref{mainprop1} yields the following.

\begin{theorem}\label{mainprop2}
A star operation on an integral domain $D$ is equivalently a nucleus $*$ on the semiprequantale $\Fp(D)$ such that $D^* = D$.  Moreover, the following conditions are equivalent for any self-map $*$ of $\Fp(D)$ such that $D^* = D$.
\begin{enumerate}
\item $*$ is a star operation on $D$.
\item $*$ is a closure operation on the poset $\Fp(D)$ and $*$-multiplication is associative.
\item $\star$ is a closure operation on the poset $\F(D)$ and $(I^\star J^\star)^\star = (IJ)^\star$ for all $I,J \in \F(D)$.
\item $IJ \subset K^*$ if and only if $I J^* \subset K^*$ for all $I,J,K \in \Fp(D)$.
\item $* = \star|_{\Fp(D)}$ for some semistar operation $\star$ on $D$ such that $D^\star = D$.
\end{enumerate}
\end{theorem}

It is well-known that, for any star operation $*$ on $D$, there is a unique largest semistar operation $l(*)$ on $D$ such that $I^{l(*)} = I^*$ for all $I \in \Fp(D)$.  Indeed, one may extend $*$ to $\F(D)$ by defining $I^{l(*)} = F$ for all $I \in \F(D)\setmin\Fp(D)$.   (See Proposition \ref{extendingclosures}(3).)  Using Proposition \ref{extendingclosures}(2), we may construct the unique smallest semistar operation $s(*)$ on $D$ such that $I^{s(*)} = I^*$ for all $I \in \Fp(D)$, as follows.

\begin{proposition}\label{semistarfromstar}
Let $D$ be an integral domain.
\begin{enumerate}
\item For any star operation $*$ on $D$, there exists a unique smallest semistar operation $s(*)$ on $D$ such that ${s(*)}|_{\Fp(D)} = *$, and for all $I \in \F(D)$ one has
$$I^{s(*)} = \bigcap \{K \in \F(D): \ K \supset I \mbox{ and } \forall J \in \Fp(D) \ (J \subset K \Rightarrow J^* \subset K)\}.$$
\item More generally, for any nucleus $\gamma$ on a submonoid $\mathcal{M}$ of $\F(D)$ containing $\P(D)$, there exists a unique smallest semistar operation $\star$ on $D$ such that $\star|_{\mathcal{M}} = \gamma$, and for all $I \in \F(D)$ one has
$$I^\star = \bigcap \{K \in \F(D): \ K \supset I \mbox{ and } \forall J \in \mathcal{M} \ (J \subset K \Rightarrow J^\gamma \subset K)\}.$$
\end{enumerate}
\end{proposition}

Although $l(*)$ has been studied predominantly in the literature, the semistar operation $s(*)$ is a useful alternative.  Note in particular that $s(d) = d$ on any domain $D$, and by Proposition \ref{sstarfintype} below $*$ is of finite type if and only if $s(*)$ is of finite type, where a star or semistar operation $\star$ is said to be {\it of finite type} if $\star = \star_f$.  By contrast, although the star operation $d$ is of finite type, the semistar operation $l(d)$ on $D$ is of finite type if and only if $l(d) = d$ on $D$, if and only if $D$ is {\it conducive}, that is, every overring of $D$ other than $F$ is a fractional ideal of $D$.  

\begin{proposition}\label{sstarfintype}
Let $D$ be an integral domain.  For any semistar operation $\star$ and star operation $*$ on $D$, one has $\star|_{\Fp(D)} = *$ if and only if $s(*) \leq \star \leq l(*)$, in which case
$s(*_f) = s(*)_f = \star_f = l(*)_f$.  In particular, $*$ is a finite type star operation if and only if $s(*)$ is a finite type semistar operation, in which case $s(*)$ is the unique finite type semistar operation $\star$ on $D$ such that $\star|_{\Fp(D)} = *$.
\end{proposition}

\begin{proof}
The first equivalence is clear from the construction of $s(*)$ and $l(*)$.  Suppose that $\star|_{\Fp(D)} = *$.   For any $I \in \F(D)$, one has $I^{\star_f} = \bigcup_{J \subset I \textup{ f.g.}} J^*$.  It follows that $s(*)_f = \star_f = l(*)_f$.  It also follows that if $*$ is of finite type, then $s(*)_f|_{\Fp(D)} = *$, whence $s(*) \leq s(*)_f$, whence $s(*)$ is of finite type.  In particular, one has $s(*_f) = s(*_f)_f \leq s(*)_f$.  To prove the reverse inequality, let $I \in \F(D)$, and note as above that $I^{s(*)_f} =\bigcup_{J \subset I \textup{ f.g.}} J^*$.  It follows that if $K \supset I$, and if $J \subset K$ implies $J^{*_f} \subset K$ for all $J \in \F(D)$, then $I^{s(*)_f} \subset K$.  Therefore $I^{s(*)_f} \subset I^{s(*_f)}$.  Thus we have $s(*)_f \leq s(*_f)$.  The rest of the proposition follows.
\end{proof}

The star operation $v = v(D)$ is the largest star operation on $D$, and $t = v_f$ is the largest finite type star operation on $D$.  One has $l(v) = v$ and $l(t) = t$, but one may have $s(v) < l(v)$ and $s(t) < l(t)$.  Indeed, let $D$ be any divisorial domain, that is, any domain for which $I^v = I$ for all $I \in \Fp(D)$.  Since $v = t = d$ on $D$, it follows that $s(v) = s(t) = d$ on $D$ as well.  However, if $D$ is non-conducive, say, if $D = \ZZ$, then $d < l(t) \leq l(v)$, whence $s(v) < l(v)$ and $s(t) < l(t)$.

The following proposition, which follows from the results of Sections \ref{sec:COCL} and \ref{sec:PFN}, lists properties of the posets $\SStar(D)$, $\SStar_f(D)$, $\Star(D)$, and $\Star_f(D)$ of all semistar operations, finite type semistar operations, star operations, and finite type star operations, respectively, on $D$.  To our knowledge statements (2) and (4) of the proposition are new, as are the formulas given for $\bigvee \Gamma$ for any set $\Gamma$ of star or semistar operations on $D$.

\begin{proposition}
Let $D$ be an integral domain.  One has $\SStar(D) = \N(\F(D))$, $\SStar_f(D) = \N_f(\F(D))$, $\Star(D) = \N(\Fp(D))_{\leq v}$, and $\Star_f(D) = \N_f(\Fp(D))_{\leq t}$, and all four of these posets are complete.  Let $\Gamma \subset \SStar(D)$ and $\Delta \subset \Star(D)$, and let $\langle \Gamma \rangle$ (resp., $\langle \Delta \rangle$) denote the submonoid generated by $\Gamma$ (resp., $\Delta$) of the monoid of all self-maps of $\F(D)$ (resp., $\Fp(D)$).
\begin{enumerate}
\item  One has $$I^{\bigwedge \Gamma} = \bigcap\{I^\star: \ \star \in \Gamma\},$$ $$I^{\bigvee \Gamma} = \bigcap\{J \in \F(D): \ J \supset I \mbox{ and } \forall \star \in \Gamma \ (J^\star = J)\},$$ for all $I \in \F(D)$.
\item If $\Gamma \subset \SStar_f(D)$, then $$I^{\bigvee \Gamma} = \bigcup\{I^\gamma: \ \gamma \in \langle \Gamma \rangle\}$$ for all $I \in \F(D)$, one has $\bigvee_{\SStar_f(D)} \Gamma = \bigvee \Gamma$ and $\bigwedge_{\SStar_f(D)} \Gamma = (\bigwedge \Gamma)_f$, and if $\Gamma$ is finite then $\bigwedge_{\SStar_f(D)} \Gamma = \bigwedge \Gamma$.
\item One has $$I^{\bigwedge \Delta} = \bigcap\{I^\star: \ * \in \Delta\},$$ $$I^{\bigvee \Delta} = \bigcap\{J \in \Fp(D): \ J \supset I \mbox{ and } \forall * \in \Delta \ (J^\star = J)\},$$ for all $I \in \Fp(D)$.
\item If $\Delta \subset \Star_f(D)$, then $$I^{\bigvee \Delta} = \bigcup\{I^\gamma: \ \gamma \in \langle \Delta \rangle\}$$  for all $I \in \Fp(D)$, one has $\bigvee_{\Star_f(D)} \Delta = \bigvee \Delta$ and $\bigwedge_{\Star_f(D)} \Delta =  (\bigwedge \Delta)_f$, and if $\Delta$ is finite then $\bigwedge_{\Star_f(D)} \Delta = \bigwedge \Delta$.
\end{enumerate}
\end{proposition}

Next, let $J \in \F(D)$.  For all $I \in \F(D)$ we let $I^{v(J)} = (J:_F (J:_F I))$.   The operation $v(J)$ on $\F(D)$ is called {\it divisorial closure on $D$ with respect to $J$} \cite[Example 1.8(a)]{pic}. % Note that $v = v(D)$ is the divisorial closure semistar operation on $D$. 
We set $t(J) = v(J)_f$, $\overline{v}(J) = \overline{v(J)}$, and $w(J) = v(J)_w = \overline{t(J)}$, where for any semistar operation $\star$ one defines
$\overline{\star}$ by $I^{\overline{\star}} = \bigcup\{(I:_F J):\ J \subset D, \ J^\star = D^\star\}$ for all $I \in \F(D)$ and $\star_w = \overline{\star_f}$.  
A semistar operation $\star$ on $D$ is said to be {\it stable} if $(I \cap J)^\star = I^\star \cap J^\star$ for all $I,J \in\F(D)$. 
The results of Sections \ref{sec:FN}, \ref{sec:DC},  and \ref{sec:stable} yield the following.

\begin{proposition}
Let $D$ be an integral domain and $J \in \F(D)$.
\begin{enumerate}
\item $v(J)$ is the largest semistar operation on $D$ such that $J^{v(J)} = J$.
\item $t(J)$ is the largest finite type semistar operation on $D$ such that $J^{t(J)} = J$.
\item $\overline{v}(J)$ is the largest stable semistar operation on $D$  such that $J^{\overline{v}(J)} = J$.
\item $w(J)$ is the largest stable finite type semistar operation on $D$ such that $J^{w(J)} = J$.
\end{enumerate}
\end{proposition}

\begin{proposition}\label{semidivprop}
Let $\star$ be a semistar operation on an integral domain $D$.
\begin{enumerate}
\item $\star = \bigwedge\{v(J): \ J \in \F(D)^\star\}$.
\item $\star_f = \bigwedge\{t(J): \ J \in \F(D)^{\star_f}\}$.
\item $\overline{\star} = \bigwedge\{\overline{v}(J): \ J \in \F(D)^{\overline{\star}}\}$.
\item $\star_w = \bigwedge\{w(J): \ J \in \F(D)^{\star_w}\}$.
\end{enumerate}
Moreover, $\star$ is stable if and only if $\star = \bigwedge\{\overline{v}(J): \ J \in \S\}$ for some $\S \subset \F(D)$.
\end{proposition}

%Similar results hold for star operations. 

Unfortunately we are unable to characterize all subsets $\S$ of $\F(D)$ such that $\bigwedge\{t(J): \ J \in \S\}$ is finitary.  We leave this as an open problem.

Proposition \ref{semidivprop} shows that all semistar operations can be obtained as infima of generalized divisorial closure semistar operations.  This has the potential for yielding new results on the number of semistar operations on an integral domain, which is a commonly studied problem in the theory of semistar operations.  For example, combining Corollary \ref{simplenearmult} with \cite[Theorem 48]{oka}, we obtain the following.

\begin{proposition}
The following are equivalent for any integral domain $D$ with quotient field $F \neq D$.  .
\begin{enumerate}
\item $D$ is a DVR.
\item The only semistar operations on $D$ are $d$ and $e$.
\item $|\SStar(D)| = 2$.
\item The near $\K$-lattice $\F(D)$ is simple.
\item $v(J) = d$ for all $J \in \F(D)$ with $J \neq F$.
\item $(J :_F (J:_F I)) = I$ for all $I,J \in \F(D)$ with $I,J \neq F$.
\end{enumerate}
\end{proposition}

%%%%%%%%%%%%%%%%%%%%%%%%%%%%%%%%%
%%%%%Section on Dedekind domains
%%%%%%%%%%%%%%%%%%%%%%%%%%%%%%%%%

\section{Integral closure, complete integral closure, and tight closure}\label{sec:ICCICTC}

Let us recall the definition of the integral closure semistar operation.  Let $D$ be an integral domain with quotient field $F$.  If $I \in \F(D)$, then $x \in F$ is {\it integral over I} if there exists a positive integer $n$ and elements $a_i \in I^i$ such that $x^n + a_1 x^{n-1} + \cdots + a_{n-1}x + a_n = 0$.  The {\it semistar integral closure of $I$} is the set $\overline{I}$ of all elements of $F$ that are integral over $I$ and is given by $\overline{I} = \bigcap\{IV: V \textup{ is a valuation overring of } D\}$.  It follows that semistar integral closure is a semistar operation on $D$.  In the literature of semistar operations it is often called the {\it $b$-operation}.  Semistar integral closure is of finite type.  Indeed, if $x \in \overline{I}$, say, if $x^n + a_1 x^{n-1} + \cdots + a_{n-1}x + a_n = 0$ for elements
$a_i = \sum_j b_{ij1}b_{ij2}\cdots b_{iji} \in I^i$, where each $b_{ijk}$ lies in $I$, then $x \in \overline{J}$, where $J \subset I$ is the ideal generated by the $b_{ijk}$.  %Thus, if $D$ is integrally closed, then $I \subset \overline{I} \subset I^t \subset I^v$ for all $I \in \F(D)$.  %In particular, if $D$ is an integrally closed divisorial domain, or more generally if $D$ is integrally closed and $t = d$ on $\F(D)$, then every K-fractional ideal of $D$ is integrally closed.
For any ideal $I$ of $D$, the ideal $\overline{I} \cap D$ is known as the {\it integral closure} of $I$.

The operation of complete integral closure can also be used to define a semistar operation. If $I \in \F(D)$, then $x \in F$ is {\it almost integral over I} if there exists a nonzero element $c$ of $F$ such that $c x^n \in I^n$ for all positive integers $n$.  We define $\widetilde{I}$ to be the set of all elements of $F$ that are almost integral over $I$, and we say that $I$ is {\it completely integrally closed} if $\widetilde{I} = I$.  Although one may have $\widetilde{D}\subsetneq \widetilde{\widetilde{D}}$, the map $I \longmapsto \widetilde{I}$ is at least a preclosure on $\F(D)$. We define
$$I^\sharp = \bigcap\{J \in \F(D): J \supset I \mbox{ and } J \mbox{ is completely integrally closed}\},$$
which we call the {\it semistar complete integral closure} of $I$.  Since $a\widetilde{J} = \widetilde{aJ}$ for all nonzero $a \in F$, and therefore $I\widetilde{J} \subset \widetilde{IJ}$, for all $I,J \in \F(D)$, by Lemma \ref{preclosurelemma} the operation $\sharp$ is a semistar operation on $D$.  In particular, we have the following.

\begin{proposition}
Let $D$ be an integral domain.  The operation $\sharp$ is the unique semistar operation on $D$ such that $I^\sharp = I$ for $I \in \F(D)$ if and only if $I$ is completely integrally closed.  Moreover,  for all $I \in \F(D)$ one has $\overline{I} \subset \widetilde{I} \subset \widetilde{I^\sharp} = I^\sharp$, with equalities holding if $D$ is Noetherian.  Thus $I^\sharp$ is the smallest completely integrally closed K-fractional ideal of $D$ containing $I$, and $\sharp$ equals the semistar integral closure operation on $D$ if $D$ is Noetherian.  Moreover, $D^\sharp$ is the smallest completely integrally closed domain containing $D$ contained in the quotient field of $D$. 
\end{proposition}

Next, let $D^+$ denote the integral closure of $D$ in an algebraic closure of $F$.  For all $I \in \F(D)$, let $I^\pi = ID^+ \cap F$, which we call the {\it semistar plus closure} of $I$, and which is easily seen to define a semistar operation $\pi$ on $D$.   Note that $D^\pi = \overline{D}$ and therefore $I\overline{D} \subset I^\pi$ for all $I \in \F(D)$.  Like semistar integral closure, semistar plus closure is of finite type.  Also, for any ideal $I$ of $D$, the ideal $I^\pi \cap D = ID^+ \cap D$ is known as the {\it plus closure} of $I$.

The operation of tight closure also yields a semistar operation, as follows.  Let $D$ be any integral domain (not necessarily Noetherian) of characteristic $p > 0$ with quotient field $F$.  For any $I \in \F(D)$ let  $$I^T = \{x \in F: \ \exists c \in F\setmin\{0\} \mbox{ such that } cx^q \in I^{[q]} \mbox{ for all  } q = p^e, e \geq 0\},$$
where $I^{[q]}$ for any $q = p^e$ denotes the $D$-submodule of $F$ generated by the image of $I$ under the $q$-powering Frobenius endomorphism of $F$.  It is clear that $T$ is a preclosure on $\F(D)$.   Moreover, one has $aJ^T = (aJ)^T$ for all nonzero $a \in F$, and therefore $IJ^T \subset (IJ)^T$, for all $I,J \in \F(D)$.  Therefore by Lemma \ref{preclosurelemma} the operation $\tau$ on $\F(D)$ of {\it semistar tight closure}, defined by $$I^\tau = \bigcap\{J \in \F(D): \ J \supset I \mbox{ and } J^T = J\}$$
for all $I \in \F(D)$, is a semistar operation on $D$.  Note that by \cite[Corollary 4]{AZ} one has $\widetilde{D} \subset D^T \subset \widetilde{\widetilde{D}}$,
from which it follows that $D^\tau = D^\sharp$.  In particular, $\tau$ restricts to a star operation on $D$ if and only if $D^\sharp = D$ if and only if $D$ is completely integrally closed.  Also, if $D$ is Noetherian then $I^\tau = I^T \subset \overline{I}$ for all $I \in \F(D)$.

\begin{example}\label{tightexample}
Let $k$ be a finite field of characteristic $p > 0$, and let $D$ be the subring $k[X,XY,XY^p, XY^{p^2}, \ldots]$ of $k[X,Y]$.  Then $D^\tau = D^T = \widetilde{\widetilde{D}} = k[X,Y]$, but $Y \notin \widetilde{D}$, whence $\widetilde{D} \subsetneq D^\tau$.
\end{example}

The {\it tight closure} $I^*$ of an ideal $I$ of $D$ may be defined as the ideal $I^* = I^\tau \cap D$.  Since $\tau = T$ if $D$ is Noetherian, this definition coincides with the well-known definition of tight closure in the Noetherian case.  In fact, we may generalize this definition to an arbitrary commutative ring $R$ of prime characteristic $p$, as follows.  For any $I \in \mathcal{I}(R)$ we let $$I^T = \{x \in R: \ \exists c \in R^o \mbox{ such that } cx^q \in I^{[q]} \mbox{ for all  } q = p^e \gg 0\},$$ where $R^o$ is the complement of the union of the minimal primes of $R$ and $I^{[q]}$ is the ideal of $R$ generated by the image of $I$ under the $q$-powering Frobenius endomorphism of $R$.  We say that $I \in \mathcal{I}(R)$ is {\it tightly closed} if $I^T = I$.  It is clear that $T$ is a preclosure on the complete lattice $\mathcal{I}(R)$.  For any $I \in \mathcal{I}(R)$ we let
$$I^* = \bigcap\{J \in \mathcal{I}(R): \ J \supset I \mbox{ and } J \mbox{ is tightly closed}\}.$$
Since $T$ is a preclosure on $\mathcal{I}(R)$ such that $IJ^T \subset (IJ)^T$ for all $I,J \in \mathcal{I}(R)$, we obtain from Lemma \ref{preclosurelemma} the following result.

\begin{proposition}
Let $R$ be a commutative ring of prime characteristic $p$.  The operation $*$ is the unique semiprime operation on $R$ such that $I^* = I$ for $I \in \mathcal{I}(R)$ if and only if $I$ is tightly closed.  Moreover,  for all $I \in \mathcal{I}(R)$ one has $I^T \subset (I^*)^T = I^*$, with equalities holding if $R$ is Noetherian.  In particular, $I^*$ for any $I \in \mathcal{I}(R)$ is the smallest tightly closed ideal of $R$ containing $I$, and $*$ equals the ordinary tight closure operation if $R$ is Noetherian.
\end{proposition}

Thus $*$ provides a potential definition of tight closure for non-Noetherian commutative rings of prime characteristic.  The finitary nucleus $*_f$ on $\mathcal{I}(R)$ is another candidate for a definition of tight closure in the non-Noetherian case.  Further investigation may determine whether or not either of these generalizations of tight closure is useful.  One may seek, for example, a generalization of Tight Closure via Colon-Capturing \cite[Theorem 3.1]{hun} with respect to some definition of non-Noetherian Cohen-Macaulay rings, such as the notion presented in \cite{ham}.

\end{document}